\documentclass[12pt, reqno, a4paper]{amsart}

\usepackage{amssymb}
\usepackage{hyperref}
\usepackage{cancel}
\usepackage{amsmath}

\usepackage[english]{babel}

\usepackage[all]{xy}

\theoremstyle{plain}
\newtheorem{theorem}{Theorem}[section]
\newtheorem{lemma}[theorem]{Lemma}
\newtheorem{proposition}[theorem]{Proposition}

\newtheorem{remark}[theorem]{Remark}
\newtheorem{example}[theorem]{Example}

\newtheorem{conjecture}{Conjecture}

\newcommand{\supp}{\operatorname{supp}}

\newcommand{\Alt}{\operatorname{Alt}}

\DeclareMathOperator{\ch}{char}
\DeclareMathOperator{\im}{im}

\DeclareMathOperator{\GL}{GL}
\DeclareMathOperator{\UT}{UT}

\DeclareMathOperator{\End}{End}

\DeclareMathOperator{\sign}{sign}

\DeclareMathOperator{\Ann}{Ann}
\DeclareMathOperator{\Id}{Id}

\DeclareMathOperator{\PIexp}{PIexp}

\begin{document}

\title{Semigroup graded algebras and   graded PI-exponent}
\author{Alexey Gordienko \and Geoffrey Janssens \and Eric Jespers}

\address{Department of Mathematics\newline
Vrije Universiteit Brussel\newline
Pleinlaan 2, 1050 Brussel, Belgium
}
\email{efjesper@vub.ac.be, geofjans@vub.ac.be and alexey.gordienko@vub.ac.be} 


\begin{abstract}

Let $S$ be a finite semigroup and let $A$ be a finite dimensional $S$-graded algebra. We investigate the exponential rate of growth of the sequence of graded codimensions $c_n^S(A)$ of $A$, i.e $\lim\limits_{n \rightarrow \infty} \sqrt[n]{c_n^S(A)}$. For group gradings this is always an integer. Recently in \cite{ASGordienko13} the first example of an algebra with a non-integer growth rate was found. We present a large class of algebras for which we prove that their growth rate can be equal to arbitrarily large non-integers. An explicit formula is given. Surprisingly, this class consists of an infinite family of algebras simple as an $S$-graded algebra. This is in strong contrast to the group graded case for which the growth rate of such algebras always equals $\dim (A)$. In light of the previous, we also handle the problem of  classification of  all $S$-graded simple algebras, which is of independent interest. We achieve this goal for an important class of semigroups that is crucial for a solution of the general problem.

\end{abstract}

\thanks{The first author is supported by Fonds voor Wetenschappelijk Onderzoek~--- Vlaanderen (FWO) Pegasus Marie Curie post doctoral fellowship (Belgium) and RFBR grant 13-01-00234a (Russia).
The second author is supported by FWO Ph.D. fellowship (Belgium). The third author is supported by
Onderzoeksraad of Vrije Universiteit Brussel and FWO (Belgium).}

\maketitle
\section{Introduction}

It is well-known that the study of polynomial identities of an associative algebra $A$ over a field $F$ of characteristic $0$ is equivalent to the study of multilinear identities.
If one denotes by $P_n$  the space of multilinear polynomials in the non-commuting variables  $x_1,\ldots, x_n$ with coefficients from $F$, then the $n$th codimension  of $A$, denoted $c_{n}(A)$,  is 
$\dim (P_n/P_n \cap \Id (A))$, where $\Id(A)$ denotes the ideal of all identities of $A$.
In 1972 Regev, in \cite{RegevATensorB},  proved  that if $A$ satisfies an non-trivial polynomial identity,
then $(c_{n}(A))_n$ is exponentially bounded.
In recent years the  asymptotic behaviour of the sequence $(c_n(A))_n$ has been  investigated extensively and this led to  classification results of several varieties of algebras.
Giambruno and Zaicev~\cite{GiaZai99} (in case $A$ is finitely generated, see \cite{GiaZai98}) proved the following fundamental result  for an algebra $A$ satisfying a polynomial identity: the limit  $\lim_{n\rightarrow\infty}\sqrt[n]{c_n(A)}$ exists and it is an integer, which  is called the PI-exponent. This confirmed a conjecture of Amitsur.

In recent years, one has investigated this problem for several classes of rings that have some refined information, such as being graded by a group. To do so, one has to give appropriate definitions for the identities and notions considered. 
In the case of algebras $A$ graded by a (semi)group $S$, then one obtains a sequence $\left(c_n^{S\text{-gr}}(A)\right)_n$ of $S$-graded codimensions.
Aljadeff, Giambruno and La Mattina \cite{AljaGia, AljaGiaLa, GiaLa} proved that if an  associative PI-algebra is graded by a finite group, then also the graded-PI exponent exists and it is an integer, i.e. the graded analog of Amitsur's conjecture holds for group graded algebras.

The first author, in~\cite[Theorem~1]{ASGordienko5} and~\cite[Theorem~3]{ASGordienko9}, proved the same for finite dimensional associative and Lie algebras graded by any group.
It is well known that in the case of non-associative algebras non-integer exponents can arise. 
The first example of an infinite dimensional Lie algebra with a non-integer ordinary PI-exponent was constructed by S.\,P.~Mishchenko and M.\,V.~Zaicev in \cite{ZaiMishchFracPI}.
We refer to \cite{VerZaiMishch} for a complete proof and to \cite{BogdMishVer} for some recent developments.  In \cite{GiaMischZaic} Giambruno, Mishchenko and Zaicev constructed an infinite family of non-associative algebras with  an arbitrary exponent. Note however that all these examples are not  associative.

Recently, in \cite{ASGordienko13}, the first author  constructed the first example of a finite dimensional associative algebra over a field of characteristic $0$ with a non-integer exponent of some kind of polynomial identities.
The example is a finite dimensional algebra $A$ that is graded by a finite semigroup $S$ with two elements that is not a group and the identities considered are the $S$-graded polynomial identities. Equivalently, $A$ is either the direct sum of two left (or right) ideals or it is the direct sum  of a  two-sided ideal and a subring.
This naturally leads  to investigating the graded PI-exponent for algebras  $A$ graded by an arbitrary finite semigroups $S$.
 In case $A$ is $S$-graded-simple, such a ring turns out to be a generalized matrix ring, i.e.  $A=\bigoplus_{1\leqslant i\leqslant n,\; 1\leqslant j\leqslant m} A_{ij}$, a direct sum of additive subgroups $A_{ij}$ such that $A_{ij}A_{kl}\subseteq A_{il}$ (note that $A_{ij}A_{kl}$ is not necessarily equal to $\{ 0\}$, even if $j\neq k$). In this paper we investigate the graded PI-exponent in this  larger context of finite dimensional algebras $A$ graded by a finite semigroup $S$, with $A$ being graded-simple.

Recall that Bahturin, Zaicev and Sehgal in \cite{BahturinZaicevSegal} have described such algebras in the case when $S$ is a finite group and $F$ is algebraically closed; it turns out that such algebras are matrix algebras over a finite dimensional graded skew field. Hence, for our investigations and to reduce the problem to the group graded case, it is crucial to describe first the structure of $S$-graded-simple algebras in the case when the maximal subgroups of $S$ are trivial. In particular, we prove a graded version of the Malcev-Wedderburn theorem, i.e. 
there exists an $S$-graded subalgebra, say $B$ of $A$, such that $B$ is semisimple and $A=B\oplus J(A)$ (a direct sum as vector spaces). Moreover, $J(A)$ has a very specific decomposition as a direct sum of semisimple $B$-modules. Although the homogenous part of $J(A)$ is trivial, these summands  are  strongly related to the description of the  homogeneous components of $A$. Earlier results on semigroup graded rings that satisfy a polynomial identity have been obtained by Kelarev \cite{KelarevPI} and Clase and Jespers \cite{ClaseThesis, Clase}.
A characterisation of semigroup graded rings that are simple has received a lot of attention, see for example \cite{BahZaicAllGradings,BahturinZaicevSurvey,BahturinZaicevSem, jespers3,oinert3}.  
For more results on graded rings we refer to \cite{KelarevBook}.
In the second part of the paper, the obtained information on graded-simple algebras is used to give an upper bound for graded codimensions in the case $A/J(A)\cong M_{n}(F)$, a matrix algebra.
Moreover, if $n=2$, then we are able to calculate the graded PI-exponent and hence we obtain many examples of algebras with a non-integer graded PI-exponent.
It turns out that for any positive integer $m$, the  number $1+m+\sqrt{m}$  is the graded PI-exponent of a graded-simple algebra.

The outline of the paper is as follows.
In Section~\ref{SectionGeneralReductionSemigroup} we show first  that the graded simplicity can easily be reduced to rings graded by finite simple semigroups and thus to generalized matrix rings. 
Hence, from now on we assume that graded rings are generalized matrix rings, i.e. rings graded by a $0$-simple finite semigroup with trivial maximal subgroups.
Second, we   prove a  criterion for two graded-simple rings to be  graded isomorphic, this is done in terms of their factor rings by their respective   Jacobson radical (Theorem~\ref{TheoremEquivalenceSemigroupGradedSimple}).
 In Section~\ref{SectionLeftIdealsOfMatrixAlgebras} we state some background on left ideals in finite dimensional simple algebras.
 In Section~\ref{SectionReesSemigroupSimpleDescription} a description is given of finite dimensional graded-simple algebras
 (Theorem~\ref{TheoremBGradedReesSemigroupGrSimple} and Theorem~\ref{W-M S-graded-simple}), and, in
 particular, we prove a graded Malcev-Wedderburn Theorem.
In Section~\ref{SectionTGradedReesExistence} it is shown that this description is complete.
In Section~\ref{SectionGradedPIAss} we give the necessary background on graded polynomial identities and graded codimensions. In Section~\ref{SectionFTActionGr} we introduce polynomial $H$-identities
for algebras with a generalized $H$-action. Polynomial $H$-identities provide a powerful tool to study
$T$-graded polynomial identities since  graded codimensions equal $H$-codimensions when $H$
is the algebra dual to the semigroup bialgebra of $T$.
In Section~\ref{SectionGeoffreyOutline} we provide a brief outline of the rest of the paper where
we deal with graded codimensions.
In Section~\ref{SectionUpperFrac} we prove an   upper bound for graded codimensions of semigroup
graded-simple algebras $A$ over a field $F$ of characteristic $0$ with $A/J(A)\cong M_k(F)$ (Theorem~\ref{TheoremUpperFrac}). 
In Sections~\ref{SectionFracM2Equal} and~\ref{SectionFracM2Less}
we prove the existence of the graded PI-exponent for finite dimensional $T$-graded-simple algebras
$A$ over a field $F$ of characteristic $0$ with $A/J(A)\cong M_2(F)$. Also an  explicit formula is obtained. 
For some algebras the graded   PI-exponent equals
$\dim A$ (Theorem~\ref{TheoremGrPIexpTriangleFracM2}) 
and for some algebras  the graded PI-exponent is strictly less than
$\dim A$ (Theorem~\ref{TheoremGrPIexpNonTriangleFracM2}). The latter is in  contrast to the group graded case for which the growth rate of such algebras always equals $\dim(A)$.

Throughout this paper all rings are associative but do not necessarily have a unit element, except when explicitly mentioned otherwise. Also we use the abbreviations $\mathbb N = \lbrace n \in \mathbb Z \mid n > 0 \rbrace$ and $\mathbb Z_+ = \lbrace n \in \mathbb Z \mid n \geqslant 0 \rbrace$.
If $k\in \mathbb N$ and $R$ is a unital ring, then the $(i,j)$ matrix unit of $M_k(R)$ is denoted by $e_{ij}$.
  If $W$ is a subset of a left module $V$ over a ring $\Lambda$,
   we denote the $\Lambda$-linear span of $W$ in $V$ by $\langle W \rangle_\Lambda$.
   If $W$ is a subset of a semigroup $S$, then by $\langle W \rangle$ (without any subscripts) we denote the subsemigroup of $S$ generated by $W$.
  Given $a \in \mathbb R$, denote by $[a]$ its floor $\max \lbrace n\in\mathbb Z \mid n \leqslant a \rbrace$.

\section{General Reduction}\label{SectionGeneralReductionSemigroup}

Let $S$ be an arbitrary semigroup and $R$ an \textit{$S$-graded ring}, that is  
$R=\bigoplus_{s\in S} R^{(s)}$, a direct sum of additive subgroups $R^{(s)}$  of $R$ such that $R^{(s)} R^{(t)} \subseteq R^{(st)}$, for all $s,t\in S$.
One says that $R$ is  an \textit{($S$-)graded-simple} ring if  $R^{2}\neq 0$ and $\{ 0\}$ and $R$ are the only homogeneous ideals of $R$.
Recall that an additive subgroup $P$ of $R$ is said to be \textit{homogeneous} or \textit{graded} if $P=\bigoplus_{s\in S} (R^{(s)}\cap P)$.
Without loss of generality, we may replace $S$ by the semigroup generated by $\supp (R)=\{ s\in S \mid R^{(s)}\neq  0\} $.
Note that if $S$ has a zero element $\theta$, then $R^{(\theta)}$ is a homogeneous  ideal of $R$, and thus 
if $R^{(\theta)}\neq 0$ and $R$ is graded-simple, then  $R^{(\theta)}=R$ and $\supp (R)=\{ \theta \}$, a trivial graded ring. 
So, in order to investigate graded-simple rings, without loss of generality, we may assume that if $\theta \in S$, then $R^{(\theta)}=0$. 
Replacing $S$ by $S^0:=S \cup \lbrace \theta \rbrace$, if necessary, we assume that $S$ has a zero element, $R^{(\theta)}=0$ and $S=\langle\supp (R)\rangle \cup \{ \theta \}$.
This assumption will be implicitly used throughout the paper without mentioning.

From now on we assume that $R$ is $S$-graded-simple and   $S=\langle \supp (R)\rangle  \cup \{ \theta \}$. 
Note that if $I$ is an ideal of $S$, then $R_{I}=\bigoplus_{s\in I} R^{(s)}$ is an ideal of $R$.
Hence, by the graded-simplicity,
either $R_{I}=0$ or $R_{I}=R$. 
The latter implies that $I=S$.
Thus, if $I$ is a proper ideal of $S$, then  $I\cap \supp (R) =\varnothing$. 
Therefore, we may replace  $S$ by the Rees factor semigroup $S/I$ and thus we may assume that $S$ itself does not have proper ideals and $S^2\ne \lbrace \theta \rbrace$. Such semigroup $S$ is  called \textit{$0$-simple semigroup}.
If $T=S\setminus \{ \theta \}$ is a semigroup, then we may replace $S$ by the simple semigroup $T$.

Therefore, if an $S$-graded ring $R$ is $S$-graded-simple, then without loss of generality we may assume that $S$ is a $0$-simple semigroup.

Recall that there are three types of $0$-simple semigroups:
\begin{enumerate}
\item $S$ is \textit{completely $0$-simple}, i.e. $S$ is $0$-simple and contains a non-zero primitive idempotent;
\item $S$ is $0$-simple and contains a non-zero idempotent, but does not contain primitive idempotents (recall that in this case each idempotent is the identity  element of a subsemigroup isomorphic to the bicylic semigroup~\cite[Theorem 2.54]{clipre} and, in particular, $S$ is an infinite semigroup);
\item $S$ does not have non-zero  idempotents and again $S$ is infinite.
\end{enumerate}

Let $G$ be a group, let $I,J$ be sets, and
let $P=(p_{ji})$ be a $J\times I$ matrix with entries in $G^{0}:=G\cup \{ \theta \}$.
 Let $\mathcal{M}(G^{0};I,J;P)$ denote the set $\{ (g,i,j) \mid i\in I,\, j\in J,\, g\in G^{0}\}$
with all elements $(\theta ,i,j)$ being identified with the zero element $\theta$. 
One has the following associative multiplication on $\mathcal{M}(G^{0};I,J;P)$: $(g,i,j)(h,k,\ell)=(gp_{jk}h,i,\ell)$.
The semigroup $\mathcal{M}(G^{0};I,J;P)$ is called the \textit{Rees $I\times J$ matrix semigroup over the group
with zero $G^0$ with sandwich matrix $P$}.

Recall that by the Rees theorem~\cite[Theorem 3.5]{clipre} every completely $0$-simple semigroup $S$ is isomorphic to 
$\mathcal{M}(G^{0};I,J;P)$
for a maximal subgroup $G$ of $S$, sets $I,J$
and a matrix $P$ such that every row and every column of $P$ has at least one nonzero element.


If we assume that $S$ is finite, then the $S$-graded simplicity of $R$ implies
that $S$ is non-nilpotent and contains a nonzero primitive idempotent. In other words,
for finite semigroups we may restrict our consideration to gradings by completely $0$-simple semigroups $$S=\mathcal{M}(G^{0},n,m;P)=\{ (g,i,j) \mid 1\leqslant i\leqslant n,\; 1\leqslant j \leqslant m,\;  g\in G^{0}\}$$ where $G$ is a group,
$n$ and $m$ are positive integers, and $P$ is an $n\times m$ matrix with entries in $G^0$.
Note that $R$ is also graded by the semigroup 
$S'=\mathcal{M}(\{ e\}^{0},n,m;P')$,
where the $(i,j)$ component of $P'$ is $e$ if $p_{ij}\neq 0$ and $\theta$ otherwise.

In this paper we are focussing on this type of grading and  we give a complete classification of when such graded rings are graded simple (so we will deal with the case that the maximal subgroups are trivial). Recall that group-graded rings (over an algebraically closed field) that are graded-simple have been described by Bahturin, Sehgal and Zaicev in \cite{BahturinZaicevSegal,BahZaicAllGradings}. Combining this result with our classification might lead to a solution of the general problem (i.e. with maximal subgroups not necessarily trivial).

The following elementary result shows that the difference between being simple and $S$-graded-simple lies in the annihilators of the ring $R$ being nonzero. 
We call a ring \textit{faithful} if its left and right annihilators are trivial,
i.e. if $aR = 0$ or $Ra=0$ for some $a\in R$, then $a=0$.

\begin{proposition}\label{simple iff}
Let  $R$ be a ring graded by a finite $0$-simple semigroup $S$ with trivial maximal subgroups.
Then, $R$ is simple if and only if $R$ is $S$-graded-simple and $R$ is faithful.
\end{proposition}

\begin{proof}
The necessity of the conditions is obvious. Suppose  $R$ is $S$-graded-simple and $R$ is faithful.
Since $S$ is a finite $0$-simple semigroup with trivial maximal subgroups,
$S \cong \mathcal{M}(\{ e\}^{0};I,J;P)$. If  $L$ is a nonzero ideal of $R$, then $RLR$ is an $S$-homogeneous ideal of $R$. Since $R$ is faithful, $L$ is nonzero. Hence $R$ is simple.
\end{proof}

The following easy example shows that the faithfulness condition can not be removed.

Recall that a semigroup $T$ is a \textit{right zero band}
if $st=t$ for any $s,t\in T$.
Let $T$ be a right zero band of two elements.
Clearly, $T^0$ is a $0$-simple semigroup and the semigroup algebra $FT$, where $F$ is a field, is graded-simple. 
However, this algebra is not simple as it contains the proper two-sided ideal $F(e-f)$. 
Note that $(e-f)FT=0$ and thus $FT$ is not faithful.

Denote the Jacobson radical of a ring $R$ by $J(R)$.
We finish the general part by showing that if  $R \ne J(R)$ and $R$ is $S$-graded-simple,
then $J(R)$ does not contain any specific information concerning the $S$-grading
and even the structure of $R$.

First we show that a non-trivial ideal cannot contain homogeneous elements.

\begin{lemma}\label{LemmaRadicalSemigroupGradedSimple}
Let $I \ne R$ be a two-sided ideal of an $S$-graded-simple ring $R=\bigoplus_{s\in S} R^{(s)}$
for some semigroup $S$. Then $R^{(s)} \cap I = 0$
for all $s\in S$.
\end{lemma}
\begin{proof}
Suppose $r \in R^{(s)} \cap I$ for some $s\in T$.
Then the smallest two-sided ideal $I_0$ containing $r$ is homogeneous.
Since $I_0 \subseteq I \subsetneqq R$, we get $I_0=0$ and $r=0$.
\end{proof}

Recall that a homomorphism $\varphi \colon R_1 \rightarrow R_2$ of $S$-graded rings $R_1$ and $R_2$
is \textit{graded} if $\varphi\left(R_1^{(s)} \right) \subseteq R_2^{(s)}$
 for all $s\in S$. Two $S$-graded rings $R_1$ and  $R_2$ are \textit{isomorphic as graded rings}
 if there exists a graded isomorphism $R_1 \rightarrow R_2$.
In this case we say that the gradings on $R_1$ and $R_2$ are \textit{isomorphic}.

The theorem below will be used later in the case
when $R/J(R)$ is simple.

\begin{theorem}\label{TheoremEquivalenceSemigroupGradedSimple}
Let $S$ be a semigroup and let $R_i=\bigoplus_{s\in S} R_i^{(s)}$, $i=1,2$, be two $S$-graded-simple rings, $R_i \ne J(R_i)$ for both $i=1,2$.
If there exists a ring isomorphism $\bar\varphi \colon R_1/J(R_1) \rightarrow  R_2/J(R_2)$
such that $\bar\varphi\left(\pi_1\left(R_1^{(s)}\right)\right)=\pi_2\left(R_2^{(s)}\right)$
for every $s\in S$ 
where $\pi_i \colon R_i \rightarrow R_i/J(R_i)$, $i=1,2$, are the natural epimorphisms, then there exists an isomorphism $\varphi \colon R_1 \rightarrow R_2$ of graded rings
such that $\pi_2\varphi=\bar\varphi\pi_1$.

Conversely, if $\varphi \colon R_1 \rightarrow R_2$ is an isomorphism of graded rings,
we can define the ring isomorphism $\bar\varphi \colon R_1/J(R_1) \rightarrow R_2/J(R_2)$
by $\bar\varphi(\pi_1(a)) = \pi_2\varphi(a)$ for all $a \in R_1$
and get $\bar\varphi\left(\pi_1\left(R_1^{(s)}\right)\right)=\pi_2\left(R_2^{(s)}\right)$
for every $s\in S$.
\end{theorem}
\begin{proof}
Suppose that there exists such an isomorphism $\bar\varphi$.
Lemma~\ref{LemmaRadicalSemigroupGradedSimple}
implies that $$\pi_i\bigr|_{R^{(s)}_i} \colon R^{(s)}_i \rightarrow \pi_i\left(R^{(s)}_i\right)$$ is an isomorphism of additive groups for every $s\in S$ and $i=1,2$.
Define $\varphi \colon R_1 \rightarrow R_2$ by  $$\varphi\left(r\right):=\left(\pi_2\bigr|_{R^{(s)}_2}\right)^{-1}\bar\varphi\pi_1(r)\text{
for }r\in R_1^{(s)}\text{ and }s\in S$$
and extend it additively. Clearly, $\varphi \left(R^{(s)}_1\right)=R^{(s)}_2$ and $\varphi$ is a graded surjective homomorphism of additive groups. Moreover $\pi_2\varphi=\bar\varphi\pi_1$ holds.

Suppose $\varphi\left( \sum_{s\in S} r^{(s)}\right)=0$
for some $r^{(s)}\in R_1^{(s)}$ and $s\in S$. Since $\varphi$ is graded, we have $\varphi\left(r^{(s)}\right)=0$ for every $s\in S$. Hence $\pi_1\left(r^{(s)}\right)=0$
and thus $r^{(s)}=0$, since by Lemma~\ref{LemmaRadicalSemigroupGradedSimple} we have $R_1^{(s)}\cap J(R_1)=0$
for every $s\in S$. Therefore, $\varphi$ is a bijection.

Now we prove that $\varphi$ is an isomorphism of rings.
Indeed, suppose $r^{(s)} \in R_1^{(s)}$ and $r^{(t)} \in R_1^{(t)}$.
Then $$\pi_2\varphi\left(r^{(s)}r^{(t)}\right)=\bar\varphi\pi_1\left(r^{(s)}r^{(t)}\right)
=\bar\varphi\pi_1\left(r^{(s)}\right)\bar\varphi\pi_1\left(r^{(t)}\right)=\pi_2\left(
\varphi\left(r^{(s)}\right)\varphi\left(r^{(t)}\right)\right).$$
Since both $\varphi\left(r^{(s)}r^{(t)}\right)$ and $\varphi\left(r^{(s)}\right)\varphi\left(r^{(t)}\right)$ belong to $ R_2^{(st)}$ and
 $\pi_2\bigr|_{R^{(st)}_2}$ is an isomorphism, we get $$\varphi\left(r^{(s)}r^{(t)}\right)=\varphi\left(r^{(s)}\right)\varphi\left(r^{(t)}\right)
 \text{ for all }r^{(s)} \in R_1^{(s)}\text{ and }r^{(t)} \in R_1^{(t)}$$
 and the first assertion is proved.
 
 The second assertion is obvious since the Jacobson radical is stable under isomorphisms.
\end{proof}

\begin{remark} {\rm Clearly, Theorem~\ref{TheoremEquivalenceSemigroupGradedSimple} holds not only for
the Jacobson radical, but for any radical.
(See the general definition e.g. in~\cite[p. 66]{KelarevBook}.)}
\end{remark}

\section{Left ideals of matrix algebras}\label{SectionLeftIdealsOfMatrixAlgebras}

Here we state some propositions which turn out to be very useful in order to classify all possible finite dimensional $T$-graded-simple algebras for some (right) zero band $T$.

These results are known, however, for the reader's convenience, we include their proofs.

\begin{lemma}\label{LemmaKerMnFActDuality} Let $F$ be a field and let $k\in\mathbb N$. Consider the natural $M_k(F)$-action
on the coordinate space $F^k$ by linear operators. Then there exists
a one-to-one correspondence between left ideals $I$ in $M_k(F)$
and subspaces $W \subseteq F^k$ such that \begin{equation}\label{EquationIWonetoone}
I=\Ann W :=\lbrace a\in M_k(F) \mid aW = 0\rbrace,\qquad W=\bigcap_{a\in I} \ker a,\end{equation}
and $\dim I = k(k-\dim W)$.
Moreover, if $I_1=\Ann W_1$ and $I_2=\Ann W_2$,
then $I_1+I_2= \Ann (W_1 \cap W_2)$ and $I_1 \cap I_2= \Ann(W_1+W_2)$.
\end{lemma}
\begin{proof} 
Let $W$ be a subspace of $F^k$. Let $w_{k+1-\dim W}, \ldots, w_k$ be a basis in $W$.
Choose $w_1, \ldots, w_{k-\dim W} \in F^k$ such that $w_1, \ldots, w_k$ is a basis in $F^k$.
Then $\Ann W$ consists of all $a\in M_k(F)$ that have matrices in the basis $w_1, \ldots, w_k$
with zeros in the last $\dim W$ columns.
Note that $\bigcap_{a\in \Ann W} \ker a = W$ and $\dim \Ann W=k(k-\dim W)$.

Let $I\subseteq M_k(F)$ be a left ideal.
Since $I$ is a left ideal in the semisimple artinian algebra $M_k(F)$, by~\cite[Theorem~1.4.2]{Herstein},
there exists an idempotent $e\in I$ such that $I=M_k(F)e$. Thus $I(\ker e)=0$. Note that $e$ is acting on $F^k$ as a projection. Hence $F^k = \ker e \oplus \im e$. We choose a basis in $F^k$ that is
the union of bases in $\im e$ and $\ker e$. Then the matrix of $e$ in this basis is $\left(\begin{smallmatrix}
E & 0 \\
0 & 0 
\end{smallmatrix}\right)$ and $I$ contains all operators with zeros in the last $\dim\ker e$ columns.
Thus $\bigcap_{a\in I} \ker a = \ker e$ and $\Ann\ker e = I$.
Together with the first paragraph this implies that~(\ref{EquationIWonetoone}) is indeed
a one-to-one correspondence.

Suppose $I_1=\Ann W_1$ and $I_2=\Ann W_2$.
Then $$I_1 \cap I_2 = \Ann W_1 
\cap \Ann W_2 = \Ann(W_1+W_2).$$
Moreover, $(I_1 + I_2)(W_1\cap W_2)=0$ and $I_1 + I_2 \subseteq 
\Ann (W_1\cap W_2)$.
Now \begin{equation*}\begin{split}\dim(I_1 + I_2)=\dim I_1 + \dim I_2 - \dim(I_1\cap I_2)
=\\ k(2k-\dim W_1 - \dim W_2)-k(k-\dim(W_1+W_2))=\\ k(k-(\dim W_1 + \dim W_2-\dim(W_1+W_2)))
= \dim \Ann(W_1\cap W_2)\end{split}\end{equation*} implies the lemma.
\end{proof}

\begin{theorem}\label{TheoremSumLeftIdealsMatrix} Let
$k,s \in \mathbb N$ and let $F$ be a field.
Assume $I_i$ are left ideals of $M_k(F)$
such that $M_k(F)=\bigoplus_{i=1}^s I_i$.
Suppose $\dim I_i = n_i k$, $n_i\in\mathbb Z_+$.
Then there exists $P\in \GL_k(F)$ such that $P^{-1} I_i P$
consists of all matrices with zeros in all columns except those that have numbers
$$1+\sum_{j=1}^{i-1} n_j,\ 2+\sum_{j=1}^{i-1} n_j,\ \ldots,\ n_i+\sum_{j=1}^{i-1} n_j.$$
\end{theorem}
\begin{proof}
 Consider the standard action of $M_k(F)$ on the coordinate space $F^k$.
 By Lemma~\ref{LemmaKerMnFActDuality}, $I_i=\Ann V_i$ for some $V_i\subseteq F^k$.
  Applying the duality from Lemma~\ref{LemmaKerMnFActDuality} to~$M_k(F)=\bigoplus_{i=1}^s I_i$, we get $\bigcap_{i=1}^s V_i = 0$
and $$V_i + \bigcap_{\substack{j=1,\\ j\ne i}}^s V_j = 
F^k\text{ for all }1\leqslant i\leqslant s.$$ 
 Denote $W_i = \bigcap\limits_{\substack{j=1,\\ j\ne i}}^s V_j$.
 Then $F^k=V_i\oplus W_i$. 
 Note that $$\Ann W_i = \bigoplus_{\substack{j=1,\\ j\ne i}}^s I_j.$$
 Since $\bigcap\limits_{\substack{j=1,\\ j\ne i}}^s \Ann W_j = 
 I_i$,
 we have $V_i = \bigoplus\limits_{\substack{j=1,\\ j\ne i}}^s W_j$.
 
 Now, choose a basis in $F^k$ that is a union of bases in $W_i$. 
 Denote the transition matrix from the standard basis to this basis by $P\in \GL_k(F)$. Then each $P^{-1} I_i P$
 consists of all matrices with zeros in all columns except those
 that correspond to $W_i$.  
  \end{proof}
  
  \begin{lemma}\label{LemmaLeftIdealMatrix}
  Let $I$ be a minimal left ideal of $M_k(F)$ where $k\in\mathbb N$, $F$ is a field.
  Then there exist $\mu_j\in F$, $1\leqslant j \leqslant k$, such that
  $I = \left\langle \sum\limits_{j=1}^k \mu_j e_{ij} \mathrel{\biggl|} 1\leqslant i \leqslant k  \right\rangle_F$.
  \end{lemma}
  \begin{proof} Let $a=\sum\limits_{i,j=1}^k \mu_{ij} e_{ij} \in I \backslash \lbrace 0 \rbrace$.
Since $\sum_{\ell=1}^k e_{\ell\ell} a = a$, we have $e_{\ell\ell}a \ne 0$ for some $1\leqslant \ell \leqslant k$. Define $\mu_j := \mu_{\ell j}$ for all $1\leqslant j \leqslant k$.
Then $$\left\langle \sum\limits_{j=1}^k \mu_j e_{ij} \mathrel{\biggl|} 1\leqslant i \leqslant k  \right\rangle_F = \left\langle e_{i\ell}a \mid 1\leqslant i \leqslant k  \right\rangle_F
$$ is a left ideal contained in $I$. Since $I$ is a minimal left ideal, we get the lemma.
  \end{proof}

\begin{lemma}\label{LemmaMatrixIdealsProduct}
Let $D$ be a finite dimensional division algebra over a field $F$ and let $k\in \mathbb N$.
Let $I$ and $V$ be, respectively, a left and a right ideal of $M_k(D)$.
Then $\dim_F (VI)=\frac{\dim_F V \dim_F I}{k^2 \dim D}$.
\end{lemma}
\begin{proof} 
Note that $I \cong \underbrace{M_k(D)e_{11}\oplus \ldots \oplus M_k(D)e_{11}}_{{\dim_F I}/(k\dim_F D)}$
and $V \cong \underbrace{e_{11}M_k(D)\oplus \ldots \oplus e_{11}M_k(D)}_{{\dim_F V}/(k\dim_F D)}$
as respectively, left and right $M_k(D)$-modules.
Hence $\dim_F (VI)= \frac{\dim_F I}{k\dim_F D} \dim_F(V M_k(D)e_{11})
= \frac{\dim_F V \dim_F I}{k^2(\dim_F D)^2} \dim_F(e_{11} M_k(D)e_{11})
=  \frac{\dim_F V \dim_F I}{k^2 \dim D}$.
\end{proof}

\section{Graded-simple algebras}\label{SectionReesSemigroupSimpleDescription}

Throughout this section $A$ is a finite dimensional $S$-graded $F$-algebra where $F$ is a field
and
  $$S=\mathcal{M}(\{ e\}^{0},n,m;P)=\langle\supp (A)\rangle \cup \{ \theta \}$$
is a finite completely $0$-simple semigroup
having trivial maximal subgroups. 
Denote the homogeneous component corresponding to $(e,i,j)$ by $A_{ij}$.
Then
    $$A=\bigoplus_{\substack{1\leqslant i \leqslant n,\\ 1\leqslant j \leqslant m}} A_{ij}$$
and 
  $$A_{ij}A_{k\ell}\subseteq A_{i\ell}.$$
In particular, each homogeneous component $A_{ij}$ is a subalgebra of $A$.
If $p_{jk}=0$ where $(p_{jk})_{j,k}:= P$, then $A_{ij}A_{k\ell} = 0$.

Note that $A$ is $\mathcal{M}(\{ e\}^{0},n,m;P)$-graded-simple
for some matrix $P$ if and only if $A$ is $\mathcal{M}(\{ e\}^{0},n,m;P')$-graded-simple,
where $P'$ is the matrix with all the entries being equal to $e$.

We begin with some basic observations.

\begin{lemma} \label{lem1} 
The following properties hold for an $S$-graded-simple algebra $A$:
\begin{enumerate}
\item $A_{ij} \cap J(A)= 0$ for all $i,j$;
\item if $I\subseteq A$ is a set, then $AIA$ is a homogeneous ideal (and thus $AIA$ equals either  $0$ or $A$);
\item $AJ(A)A=0$.
\end{enumerate}
\end{lemma}
\begin{proof}
Part (1) is a direct consequence of Lemma~\ref{LemmaRadicalSemigroupGradedSimple}.
Part (2) is  obvious, Part (3) is a direct consequence of (2).
\end{proof}

If $n=1$, i.e $S$ has only one row, then all the graded components are left ideals and thus $J(A)A$ is homogeneous. So, if $A$ is graded-simple, then $J(A)A=0$. Hence in this case one can reformulate Theorem~\ref{W-M S-graded-simple} in a simpler form.

Define the left ideals $L_i:=\bigoplus_{k=1}^n A_{ki}$ and the right ideals 
$R_i:=\bigoplus_{k=1}^m A_{ik}$. Then $L_j \cap R_i = A_{ij}$.
Denote by $1_R$ the identity element of a ring $R$ (if $1_R$ exists).

The following theorem is a graded version of the  Malcev-Wedderburn theorem. It is shown that there exist  orthogonal ``column'' (respectively, ``row'') homogeneous idempotents that define a semisimple complement of the radical. This result is a first step towards the classification of $S$-graded-simple algebras.

\begin{theorem}\label{TheoremBGradedReesSemigroupGrSimple}
Let $A=\bigoplus_{i,j} A_{ij}$ be a finite dimensional $S$-graded $F$-algebra over a field $F$ such that $AJ(A)A=0$. Then, there exist
orthogonal  idempotents $f_{1}, \ldots, f_m$ and orthogonal idempotents $f_1', \ldots, f_n'$ (some of them could be zero) such that
$$B=\bigoplus_{i,j} f_i' A f_j =\bigoplus_{i,j} (B\cap A_{ij})$$ is an $S$-graded maximal semisimple subalgebra of $A$, $f_i' \in B \cap R_i$ for $1\leqslant i \leqslant n$, $f_j \in B \cap L_j$ for $1\leqslant j \leqslant m$, $\sum_{i=1}^n f'_i=\sum_{j=1}^m f_j=1_B$, and 
 $A=B\oplus J(A)$ (direct sum of subspaces).
\end{theorem}

\begin{proof}
We write $\bar X$ for the image of a subset $X$ of $A$ in the algebra $A/J(A)$ under the natural epimorphism
$A \rightarrow A/J(A)$.

Note that $\bar A = \sum_{j=1}^m \bar L_j$.
Since $\bar A = A/J(A)$ is semisimple and completely reducible as a left $A/J(A)$-module,
there exist left ideals $\tilde L_i \subseteq \bar  L_i$ complementary to $\bar{L_i} \bigcap \sum\limits_{j=1}^{i-1} \bar{L_j}$ in $\bar{L_i}$. Clearly, $\bar A= \bigoplus\limits_{i=1}^{m} \tilde L_i$.
 The decomposition $1_{\bar{A}} = \sum\limits_{i=1}^{m} \bar{\omega_i}$ of the identity element of $\bar A$  yields orthogonal idempotents $\bar{\omega_i} \in \tilde{L_i}$. The idempotents $\bar \omega_i$ can be lifted to homogeneous idempotents $\omega_i \in L_i$ of $A$ using the natural epimorphisms $\pi\bigl|_{L_i} \colon L_i \rightarrow L_i / L_i \cap J(A)$ since $J(A)$ is nilpotent. The idempotents $\omega_1,\ldots, \omega_m$ are orthogonal too since $\omega_i \omega_j = \omega_i(\omega_i \omega_j)\omega_j \in AJ(A)A= 0$.
 
 Analogously, one gets orthogonal idempotents $\omega_1', \ldots, \omega_n' \in A$, $\omega_i' \in R_i$ for $1\leqslant i \leqslant n$, such that $\sum_{i=1}^n \bar f'_i = 1_{\bar A}$.
Define now $B = \bigoplus\limits_{\substack{1\leqslant i \leqslant n, 
\\ 1 \leqslant j \leqslant m}} \omega_i^{'} A \omega_j$. Note that $\omega_i^{'} A \omega_j \subseteq A_{ij}$ and $B$ is an $S$-graded subalgebra of $A$.
Suppose $a = \sum\limits_{\substack{1\leqslant i \leqslant n, 
\\ 1 \leqslant j \leqslant m}} \omega_i^{'} a_{ij} \omega_j \in J(A)$
for some $a_{ij} \in A$. Then $AJ(A)A=0$ implies $\omega_i^{'} a_{ij} \omega_j = \omega_i^{'} a \omega_j = 0$
for all $i,j$. Hence $a=0$ and $J(A) \cap B = 0$. Moreover, $\bar B = 1_{\bar A}  \bar A1_{\bar A}=\bar A$.
Hence $B$ is an $S$-graded maximal semisimple subalgebra of $A$ and
$A = B \oplus J(A)$ (direct sum of subspaces).

 Decomposing $1_B$ with respect to the left ideals $\bigoplus_{i=1}^n \omega_i' A \omega_j$, $1\leqslant j \leqslant m$, and with respect to the right ideals $\bigoplus_{j=1}^m \omega_i' A \omega_j$, $1\leqslant i \leqslant n$, we get orthogonal idempotents $f_i \in B \cap L_i$ for $1\leqslant i \leqslant n$,
 and orthogonal idempotents $f_j' \in B \cap R_j$ for $1\leqslant j \leqslant m$ 
 such that $\sum_{i=1}^n f'_i = \sum_{j=1}^m f_j = 1_B$.
 Then $$B = 1_B B 1_B = \bigoplus\limits_{\substack{1\leqslant i \leqslant n, 
\\ 1 \leqslant j \leqslant m}} f_i' B f_j = \bigoplus\limits_{\substack{1\leqslant i \leqslant n, 
\\ 1 \leqslant j \leqslant m}} f_i' A f_j$$ since $A=B \oplus J(A)$ and $AJ(A)A=0$.
\end{proof}

Example~\ref{ExampleTwoBGradingsNonIso} below shows that the gradings on different $B$ in Theorem~\ref{W-M S-graded-simple} can be non-isomorphic.

\begin{example}\label{ExampleTwoBGradingsNonIso}
Let $F$ be a field, let $I$ be the left $M_2(F)$-module isomorphic to $\langle e_{12}, e_{22}\rangle_F$,
and let $\varphi \colon I \mathrel{\widetilde{\rightarrow}} \langle e_{12}, e_{22}\rangle_F$
be the corresponding isomorphism.
Let $A=M_2(F)\oplus I$ (direct sum of $M_2(F)$-modules) where $IM_2(F)=I^2=0$.
Define on $A$ the following $T_3$-grading:
$A^{(e_1)}= (M_2(F),0)$ and $A^{(e_2)}=\lbrace (\varphi(a),a) \mid a\in I\rbrace$.
Then the algebra $A$ is $T_3$-graded-simple and both $B_1 = A^{(e_1)}$ and $B_2=\langle (e_{11},0),(e_{21},0)\rangle_F\oplus
A^{(e_2)}$ are graded maximal semisimple subalgebras of $A$. However $B_1 \ncong B_2$ as graded algebras.
\end{example}

Now we present a finite dimensional $S$-graded non-graded-simple algebra 
that does not  have an $S$-graded maximal semisimple subalgebra complementary to the radical.
So, in Theorem~\ref{TheoremBGradedReesSemigroupGrSimple}, the assumption $AJ(A)A=0$ is essential.

\begin{example}
Let 
$R = F[X] / (X^2)$  and let $A= M_{2}(R)$. Put
$v_1 = \left(\begin{smallmatrix}
1 \\
0 
\end{smallmatrix}\right),
v_2 = \left(\begin{smallmatrix}
0 \\
1 
\end{smallmatrix}\right),
w_1 = \left(\begin{smallmatrix}
1 & X 
\end{smallmatrix}\right), 
w_2 = \left(\begin{smallmatrix}
0 & 1
\end{smallmatrix}\right).$
Consider the following $F$-subspaces of $A$:
 $$\begin{array}{ll}
A_ {11} = R v_1 w_1 = R \left(\begin{smallmatrix}
1 & X\\
0 & 0
\end{smallmatrix}\right),
&
A_{12} = R v_1 w_2 = R \left(\begin{smallmatrix}
0 & 1\\
0 & 0
\end{smallmatrix}\right), \\
 &\\
A_{21} = R v_2 w_1 = R \left(\begin{smallmatrix}
0 & 0 \\
1 & X
\end{smallmatrix}\right),
&
A_{22} = R v_2 w_2 = R \left(\begin{smallmatrix}
0 & 0\\
0 & 1
\end{smallmatrix}\right).
\end{array}$$ Then $A = A_{11} \oplus A_{12} \oplus A_{21} \oplus A_{22}$ is an $S$-grading for $S=\mathcal{M}\left(\{e \}^0, 2,2, \left( \begin{smallmatrix}
1 &0 \\
0 &1
\end{smallmatrix} \right)\right)$. 
However, there  does not exist an $S$-graded maximal semisimple subalgebra $B$ of $A$ such that $A= B \oplus J(A)$.
\end{example}

\begin{proof}
First we notice that $A = A_{11} \oplus A_{12} \oplus A_{21} \oplus A_{22}$ is indeed an $S$-grading since $(a\, v_i w_j) (b\, v_\ell w_k) = ab (w_j v_\ell)\, v_i w_k$ for all $a,b \in R$ and $1 \leqslant i,j,k,\ell \leqslant 2$ since $w_j v_\ell$ is a $1\times 1$ matrix which can be identified with the corresponding element of the field $F$.

 Clearly, $J(A) = \left(\begin{smallmatrix}
(X) & (X) \\
(X) & (X)
\end{smallmatrix}\right)$
and $A/J(A) \cong M_2(F)$.

Fix the following bases in the homogeneous components:
$$\begin{array}{ll}
A_{11} = \left\langle \left(\begin{smallmatrix}
1 & X \\
0 & 0
\end{smallmatrix}\right), \left(\begin{smallmatrix}
X & 0 \\
0 & 0
\end{smallmatrix}\right) \right\rangle_F, &
A_{12} = \left\langle \left(\begin{smallmatrix}
0 & 1 \\
0 & 0
\end{smallmatrix}\right), \left(\begin{smallmatrix}
0 & X \\
0 & 0
\end{smallmatrix}\right) \right\rangle_F, \\
 &\\
A_{21} = \left\langle \left(\begin{smallmatrix}
0 & 0 \\
1 & X
\end{smallmatrix}\right), \left(\begin{smallmatrix}
0 & 0 \\
X & 0
\end{smallmatrix}\right) \right\rangle_F, &
A_{22} = \left\langle \left(\begin{smallmatrix}
0 & 0 \\
0 & 1
\end{smallmatrix}\right), \left(\begin{smallmatrix}
0 & 0 \\
0 & X
\end{smallmatrix}\right) \right\rangle_F.
\end{array}$$

Now it is clear that $J(A)$ is a homogeneous ideal and $A/J(A) \cong M_2(F)$ is an $S$-graded algebra too.

Suppose that there exists a $S$-graded maximal semisimple subalgebra $B=\bigoplus_{1\leqslant i,j \leqslant 2} B_{ij}$ such that $A = B \oplus J(A)$
and $B_{ij} \subseteq A_{ij}$.
Then there exists a graded isomorphism 
 $\varphi \colon B\rightarrow   M_2(F)$. 

In particular, $B_{ii}=\langle b_{ii} \rangle_F$ where $\varphi(b_{ii})=e_{ii}$ and $b_{ii}^2=b_{ii}$ for $i=1,2$.
Hence $b_{11}=(1+\alpha X)\left(\begin{smallmatrix}
1 & X \\
0 & 0
\end{smallmatrix}\right)$ for some $\alpha \in F$
and $b_{22}=(1+\beta X)\left(\begin{smallmatrix}
0 & 0 \\
0 & 1
\end{smallmatrix}\right)$
for some $\beta \in F$.
(In fact, if $\ch F \ne 2$, then $\alpha=\beta = 0$.)
Then $$b_{11} b_{22} = (1+(\alpha+\beta)X)\left(\begin{smallmatrix}
0 & X \\
0 & 0
\end{smallmatrix}\right) = \left(\begin{smallmatrix}
0 & X \\
0 & 0
\end{smallmatrix}\right) \in J(A)$$
and we get a contradiction.
\end{proof}

Theorem \ref{TheoremBGradedReesSemigroupGrSimple}  describes the semisimple part of an $S$-graded-simple algebra. We proceed with the description of the radical and hence we obtain a  characterization of  the finite dimensional $S$-graded-simple algebras. In Section \ref{SectionTGradedReesExistence} we will show that this description delivers a complete classification.

For $r\in A$, we denote $x-xr$ (respectively $x-rx$) by $x(1-r)$ (respectively, $(1-r)x$), even if $A$ does not contain unity.

\begin{lemma} \label{LemmaReesGrSimpleOtherProperties} 
Suppose $A$ is $S$-graded-simple and let $A=B\oplus J(A)$ (direct sum of subspaces) be the decomposition
from Theorem~\ref{TheoremBGradedReesSemigroupGrSimple}. Then the following properties hold:
\begin{enumerate}
\item $J(A)^2 A = A J(A)^2=0$;
\item $B$ is a simple subalgebra;
\item $A=A1_B A$;
\item $J(A)=(1-1_B)A 1_B \oplus 1_B A(1-1_B)\oplus J(A)^2$ (direct sum of subspaces);
\item $J(A)^2= (1-1_B)A 1_B A (1-1_B) = (1-1_B)A (1-1_B)$.
\end{enumerate}
\end{lemma}
\begin{proof}
Then Part (3) is a direct consequence of Part~(2) of Lemma~\ref{lem1}.
Let $f$ be a primitive central idempotent of $B$. 
Then $A=AfA$ and, by Part~(3) of Lemma~\ref{lem1}, $$B= 1_B (B\oplus J(A)) 1_B =  1_B A 1_B= 1_B AfA 1_B = 1_B A 1_B f 1_B A 1_B = BfB=Bf=fB.$$ Hence $f$ is the identity of $B$ and thus $f=1_B$. Therefore, $B$ is simple and we get Part (2). 

Using $B=1_B A 1_B$ and the Pierce decomposition with respect to the idempotent $1_B$, we get
 $$J(A)=(1-1_B)A 1_B \oplus 1_B A (1-1_B) \oplus (1-1_B)A (1-1_B) \text{ (direct sum of subspaces)}.$$
 Part (3) implies $$(1-1_B)A (1-1_B) = (1-1_B)A 1_B A (1-1_B) \subseteq J(A)^2.$$
Hence $$(1-1_B)A (1-1_B) J(A) \subseteq J(A)^3=0,\quad J(A)(1-1_B)A (1-1_B) \subseteq J(A)^3=0.$$  
  Since $$1_B A (1-1_B)A 1_B \subseteq 1_B J(A) 1_B \subseteq AJ(A)A = 0,$$
we get $J(A)^2 \subseteq (1-1_B)A (1-1_B)$
and Propositions (1), (4) and (5) follow.
\end{proof}

Remark that if $S$ has only one row then $a - 1_B a \subseteq J(A)\cap A_{1i}=0$ for every $a \in A_{1i}$ and $1 \leqslant i \leqslant m$. Thus in this case $1_B$ acts as a left identity on $J(A)$.

It is also interesting to note that  condition  (1) of Lemma~\ref{lem1}
and condition  (2) of Lemma~\ref{LemmaReesGrSimpleOtherProperties},
together with $A^2=A$, are equivalent to the graded $S$-simplicity.

\begin{proposition}\label{TheoremReesGradedSimplicityCriterion}
Suppose that the base field $F$ is perfect, $A/J(A)$ is a simple algebra, $A^2=A$, and $A_{ij} \cap J(A) = 0$ for all $1\leqslant i \leqslant n$ and $1\leqslant j \leqslant m$. Then $A$ is $S$-graded-simple.
\end{proposition}
\begin{proof}
Let $I$ be a nonzero two-sided homogeneous ideal of $A$. Denote by $\pi \colon A \twoheadrightarrow A/J(A)$
the natural epimorphism. Then $\pi(I)\ne 0$. Since $A/J(A)$ is simple, we get $\pi(I)=A/J(A)$
and $A=I+J(A)$. By the Wedderburn~--- Mal'cev theorem, there exists a maximal semisimple
subalgebra $B \subseteq I$ such that $I=B\oplus J(I)$ (direct sum of subspaces).
Recall that $J(I)=J(A)\cap I$. Thus $A = B \oplus J(A)$.
Note that  $\pi(A (1-1_B) A)=0$. Hence $A(1-1_B) A \subseteq J(A)$.
Since $A (1-1_B) A$ is a graded ideal, we get $A(1-1_B)A=0$
and $ab = a 1_B b \in I$ for all $a,b \in A$. Thus $A=A^2\subseteq I$
and $I=A$.
\end{proof}

From Lemma~\ref{LemmaReesGrSimpleOtherProperties} we have 
$$J(A)=1_B A (1-1_B) \oplus  (1-1_B)A 1_B \oplus J(A)^2=
\sum_{j=1}^m 1_B L_j (1-1_B) \oplus \sum_{i=1}^n   (1-1_B) R_i 1_B
\oplus J(A)^2.$$
For $1\leqslant i\leqslant n$ and $1\leqslant j \leqslant m$, put
   $$J^{10}_{ij}:=f'_i L_j (1-1_B) \quad \text{ and } \quad  J^{01}_{ij}:=(1-1_B)R_i f_j.$$
Also, put
   $$J^{10}_{*j} :=\sum_{1\leqslant i \leqslant n} J^{10}_{ij} = 1_B L_j (1-1_B) \quad \text{ and } \quad  
     J^{01}_{i*}  : = \sum_{1\leqslant j \leqslant m} J^{01}_{ij}  = (1-1_B) R_i 1_B.$$
We will show that these subspaces form  the buildings blocks  of $J(A)$.

\begin{theorem} \label{W-M S-graded-simple}
Let $A$ be a finite dimensional $S$-graded-simple $F$-algebra. 
Let  $B$ and let 
$f_{1}, \ldots, f_m,\ f_1', \ldots, f_n'$
 be, respectively, a graded subalgebra and orthogonal  idempotents
 from Theorem~\ref{TheoremBGradedReesSemigroupGrSimple}.  

Then each $J^{10}_{*j}$ is a left $B$-submodule of $J(A)$ and 
$J^{10}_{*j} =\bigoplus_{i=1}^n J^{10}_{ij}$.
Also each $J^{01}_{i*}$ is a right $B$-submodule of $J(A)$ and 
$J^{01}_{i*}=\bigoplus_{j=1}^m J^{01}_{ij}$.
Moreover, $$J(A)=\bigoplus_{i=1}^n J^{01}_{i*} \oplus \bigoplus_{j=1}^m J^{10}_{*j}
\oplus J(A)^2 \quad \text{ and } \quad
J(A)^2=\bigoplus_{i=1}^n \bigoplus_{j=1}^m  J^{01}_{i*}J^{10}_{*j},$$
direct sums of subspaces.


In addition, there exists an $F$-linear map 
   $$\varphi \colon \bigoplus_{i=1}^n J^{01}_{i*} \oplus \bigoplus_{j=1}^m J^{10}_{*j} \rightarrow B$$ 
defined by  
  \begin{equation}\label{EquationPhiDefinitionReesSimple}
     \varphi(a-1_B a 1_B) = 1_B a 1_B-f_i' a f_j, \quad \text{for } a\in f_i'A_{ij}+A_{ij}f_j,
   \end{equation}
and such that 
 $\varphi\bigl|_{\bigoplus_{j=1}^m J^{10}_{*j}}$ is a homomorphism of left $B$-modules,
  \begin{equation}\label{EquationReesSimpleJ10*jfj0}
     J^{10}_{*j}\cap \ker \varphi = 0,\qquad \varphi (J^{10}_{*j}) \cap B f_j =\varphi (J^{10}_{*j}) f_j= 0,
        \quad    \text{ for every } 1\leqslant j \leqslant m,
  \end{equation}
$\varphi\bigl|_{\bigoplus_{i=1}^n J^{01}_{i*}}$ is a homomorphism of right $B$-modules,
  \begin{equation}\label{EquationReesSimplefi'J01i*0}
       J^{01}_{i*} \cap \ker \varphi = 0,\qquad \varphi (J^{01}_{i*}) \cap f_i' B =
        f_i'\varphi (J^{01}_{i*}) = 0, \quad \text{ for every }1\leqslant i \leqslant n.
   \end{equation}
Moreover, 
   \begin{eqnarray}\label{EquationDecompAijReesSimple}
      \lefteqn{A_{ij}=f_i' B f_j\oplus \left\{ \varphi(v)+v \mid  
          v \in J^{10}_{ij}\oplus J^{01}_{ij} \right\} } \nonumber \\
       &&\hspace{1cm}   \oplus \left\langle \varphi(v)\varphi(w)+v\varphi(w) +  \varphi(v) w + vw\mid  
            v \in J^{01}_{i*},\ w\in J^{10}_{*j} \right\rangle_F
            \end{eqnarray}
(direct sum of subspaces), for all  $1\leqslant i \leqslant n$, $1\leqslant j \leqslant m$.

If $s\in\mathbb N$, $v_\ell \in J^{01}_{i*}$ and $w_\ell \in J^{10}_{*j}$
for $1\leqslant \ell \leqslant s$,
then $\sum_{\ell=1}^s v_\ell w_\ell = 0$ if and only if $\sum_{\ell=1}^s \varphi(v_\ell) \varphi(w_\ell) = 0$.

Furthermore, $B \cong M_k(D)$ for some $k \in \mathbb N$ and a division algebra $D$
and
 \begin{eqnarray}\label{EquationDimJ01}
   \dim_{F} \bigoplus_{i=1}^n J^{01}_{i*} &\leqslant & (n - 1) \dim_{F} B = (n-1)k^2\dim_F D,\\
 \label{EquationDimJ10}
     \dim_{F} \bigoplus_{j=1}^m J^{10}_{*j}  &\leqslant &  (m - 1) \dim_{F} B = (m-1)k^2\dim_F D,\\
  \label{EquationDimJAReesSemiGr}
    \dim_{F} J(A) &\leqslant & (nm - 1) \dim_{F} B =(|S|-1) \dim_{F} B = (|S|-1)k^2\dim_F D.    
  \end{eqnarray}
\end{theorem}

\begin{proof}
By Lemma~\ref{LemmaReesGrSimpleOtherProperties},
$$J(A)=1_B A (1-1_B) \oplus  (1-1_B)A 1_B \oplus J(A)^2=
\sum_{j=1}^m 1_B L_j (1-1_B) \oplus \sum_{i=1}^n   (1-1_B) R_i 1_B
\oplus J(A)^2.$$
Note that if $\sum_{j=1}^m 1_B a_j (1-1_B) = 0$ for some
$a_j \in L_j$, then $\sum_{j=1}^m 1_B a_j 1_B = \sum_{j=1}^m 1_B a_j$.
Since $1_B A 1_B = B$ is a graded subalgebra, we get $1_B a_j \in B \cap L_j$,
$1_B a_j 1_B = 1_B a_j$ and  all $1_B a_j (1-1_B) = 0$.
Hence the sum $\bigoplus_{j=1}^m 1_B L_j (1-1_B)$ is direct.
Analogously, the sum $\bigoplus_{i=1}^n (1-1_B) R_i 1_B$ is direct too
and \begin{equation*}\begin{split}J(A)=\bigoplus_{j=1}^m 1_B L_j (1-1_B) \oplus \bigoplus_{i=1}^n (1-1_B) R_i 1_B
\oplus J(A)^2  = \bigoplus_{j=1}^m J^{10}_{*j}\oplus\bigoplus_{i=1}^n J^{01}_{i*} 
\oplus J(A)^2.\end{split}\end{equation*}
Using Part (5) of Lemma~\ref{LemmaReesGrSimpleOtherProperties}, we get \begin{equation}\label{EquationJASquareSum} J(A)^2 = \sum_{i=1}^n\sum_{j=1}^m
(1-1_B)R_i 1_B 1_B L_j (1-1_B) = \sum_{i=1}^n\sum_{j=1}^m J^{01}_{i*}J^{10}_{*j}.\end{equation}

Now we show that $\varphi$ can be defined by~(\ref{EquationPhiDefinitionReesSimple}).
First, $f'_i L_j \subseteq A_{ij}$ and thus $f'_i L_j = f'_i A_{ij}$, $R_i f_j \subseteq A_{ij}$ and thus $R_i f_j = A_{ij} f_j$,
$J^{10}_{ij} = f'_i L_j (1-1_B) = \lbrace a - 1_B a 1_B \mid a \in f'_i A_{ij} \rbrace$,
$J^{01}_{ij} = (1-1_B) R_i f_j = \lbrace a - 1_B a 1_B \mid a \in A_{ij} f_j \rbrace$.
If $a - 1_B a 1_B = 0$ for some $a\in A_{ij}$,
then $a\in B \cap A_{ij}=f'_i B f_j$
and $1_B a 1_B -f_i' a f_j = 0$.
Hence~(\ref{EquationPhiDefinitionReesSimple}) can indeed be used to define $\varphi$ on $J^{10}_{ij}$ and
$J^{01}_{ij}$ and the definition is consistent. We extend $\varphi$ on $\bigoplus_{i=1}^n J^{01}_{i*} \oplus \bigoplus_{j=1}^m J^{10}_{*j}$ by the linearity.

Note that if $a \in f'_i A_{ij}=f'_i L_j$, then we have $a - 1_B a 1_B = a(1-1_B)$
and $\varphi(a(1-1_B))=a(1_B - f_j)$. By the linearity, this formula holds for every $a\in 1_B L_j$. Hence  $\varphi\bigl|_{J^{10}_{*j}}$ is a homomorphism of left $B$-modules. By the linearity,
$\varphi\bigl|_{\bigoplus_{j=1}^m J^{10}_{*j}}$ is a homomorphism of left $B$-modules too.
Analogously, $\varphi((1-1_B)a)=(1_B - f'_i)a$ for all $a\in R_i 1_B$
and $\varphi\bigl|_{\bigoplus_{i=1}^n J^{01}_{i*}}$ is a homomorphism of right $B$-modules.

Suppose $\varphi(1_B a (1-1_B))=0$ for some $a\in L_j$.
Then $1_B a(1_B - f_j) = 0$, $1_B a f_j = 1_B a 1_B$
and $f_i' a f_j = f_i' a 1_B$ for all $1\leqslant i \leqslant n$.
Hence $f_i' a (1-1_B)=f_i'a - f_i' a f_j \in A_{ij} \cap J(A)=0$.
Thus $1_B a (1-1_B) = 0$ and $1_B L_j (1-1_B)\cap \ker \varphi = J^{10}_{*j} \cap \ker \varphi = 0$.
Analogously, $$(1-1_B) R_i 1_B \cap \ker \varphi = J^{01}_{i*} \cap \ker \varphi= 0.$$

Moreover $$\varphi(J^{10}_{*j}) \cap B f_j =\varphi (1_B L_j (1-1_B))
\cap B f_j \subseteq \varphi (1_B L_j (1-1_B))f_j = 0$$ for every $1\leqslant j \leqslant m$
and~(\ref{EquationReesSimpleJ10*jfj0}) is proved.
Analogously, $$\varphi(J^{01}_{i*}) \cap f_i' B = \varphi ((1-1_B) R_i 1_B)
\cap f_i' B  \subseteq f_i'\varphi ((1-1_B) R_i 1_B)
  = 0$$ for every $1\leqslant i \leqslant n$ and~(\ref{EquationReesSimplefi'J01i*0}) is proved.
  
By Part (2) of Lemma~\ref{LemmaReesGrSimpleOtherProperties}, we have $B \cong M_k(D)$ for some $k \in \mathbb N$ and a division algebra $D$.
  Now  \begin{equation*}\begin{split}\dim_{F} \bigoplus_{i=1}^n J^{01}_{i*}
= \sum_{i=1}^n \dim_F \varphi(J^{01}_{i*})
\leqslant \sum_{i=1}^n (\dim_F B  - \dim_F f_i' B)
    \\ =(n - 1) \dim_{F} B = (n-1)k^2\dim_F D\end{split}\end{equation*}
   and we get~(\ref{EquationDimJ01}).
   Analogously we obtain~(\ref{EquationDimJ10}).
  
  Now we prove~(\ref{EquationDecompAijReesSimple}).
  Let 

  \begin{eqnarray*}
  \lefteqn{ \tilde A_{ij}=\left(f_i' B f_j\oplus \left\lbrace \varphi(v)+v \mid  
v \in J^{10}_{ij}\oplus J^{01}_{ij} \right\rbrace\right) }\\
      &&\hspace{1cm}   + \left\langle \varphi(v)\varphi(w)+v\varphi(w) +  \varphi(v) w + vw\mid  
       v \in J^{01}_{i*},\ w\in J^{10}_{*j} \right\rangle_F .
  \end{eqnarray*}
  
We first show that $\tilde A_{ij}=A_{ij}$. To do so, note that if $a \in f'_i L_j$, then $\varphi(a(1-1_B))+a(1-1_B) = a(1_B-f_j)+a(1-1_B)
= a- a f_j \in A_{ij}$. Therefore $\varphi(w)+w \in A_{ij}$
for every $w\in J^{10}_{ij}$.
Analogously, $\varphi(v)+v \in A_{ij}$
for every $v\in J^{01}_{ij}$. Hence $\tilde A_{ij} \subseteq A_{ij}$.

Obviously, $B \subseteq \bigoplus_{i=1}^n\bigoplus_{j=1}^m \tilde A_{ij}$
and therefore $1_B \in \bigoplus_{i=1}^n\bigoplus_{j=1}^m \tilde A_{ij}$.
By Lemma~\ref{LemmaReesGrSimpleOtherProperties}, $$1_B \bigoplus_{i=1}^n\bigoplus_{j=1}^m \tilde A_{ij} (1-1_B) = \bigoplus_{j=1}^m J^{10}_{*j}$$
and $$(1-1_B) \bigoplus_{i=1}^n\bigoplus_{j=1}^m \tilde A_{ij} 1_B = \bigoplus_{i=1}^n J^{01}_{i*}.$$
Hence
$$\bigoplus_{i=1}^n J^{01}_{i*} \oplus \bigoplus_{j=1}^m J^{10}_{*j}
\subseteq \bigoplus_{i=1}^n\bigoplus_{j=1}^m \tilde A_{ij}.$$
In addition, (\ref{EquationJASquareSum}) implies $$(1-1_B) \bigoplus_{i=1}^n\bigoplus_{j=1}^m \tilde A_{ij} (1-1_B)= J^2(A)$$ and
$J(A)^2 \subseteq \bigoplus_{i=1}^n\bigoplus_{j=1}^m \tilde A_{ij}$.
Hence $\bigoplus_{i=1}^n\bigoplus_{j=1}^m \tilde A_{ij} = A$ and
$\tilde A_{ij}=A_{ij}$. Equality~(\ref{EquationDecompAijReesSimple})
will follow from the fact that the sum in the definition of $\tilde A_{ij}$
is direct. We prove the last fact below.

Suppose $s\in\mathbb N$ and $v_\ell \in J^{01}_{i*}$ and $w_\ell \in J^{10}_{*j}$
for $1\leqslant \ell \leqslant s$.
Assume $\sum_{\ell=1}^s \varphi(v_\ell) \varphi(w_\ell) = 0$.
Since $\varphi\bigl|_{\bigoplus_{j=1}^m J^{01}_{i*}}$ is a homomorphism of right $B$-modules,
we get that 
$\varphi\left(\sum_{\ell=1}^s v_\ell \varphi(w_\ell)\right)=0$ and thus, by (3),
$\sum_{\ell=1}^s v_\ell \varphi(w_\ell) = 0$.
Analogously, $\sum_{\ell=1}^s \varphi(v_\ell) w_\ell = 0$.
Hence $$\sum_{\ell=1}^s v_\ell w_\ell = \sum_{\ell=1}^s (\varphi(v_\ell)+v_\ell) (\varphi(w_\ell)+w_\ell)\in A_{ij}\cap J(A)=0.$$

Conversely, suppose $\sum_{\ell=1}^s v_\ell w_\ell = 0$
for some $v_\ell \in J^{01}_{i*}$ and $w_\ell \in J^{10}_{*j}$, $1\leqslant \ell \leqslant s$, $s\in\mathbb N$. Let $a=\sum_{\ell=1}^s (\varphi(v_\ell)+v_\ell) (\varphi(w_\ell)+w_\ell)$, $$b=\sum_{\ell=1}^s  (\varphi(\varphi(v_\ell)w_\ell)+\varphi(v_\ell)w_\ell)
 + \sum_{\ell=1}^s  (\varphi(v_\ell\varphi(w_\ell))+v_\ell\varphi(w_\ell)) - \sum_{\ell=1}^s \varphi(v_\ell)\varphi(w_\ell).$$
Then $a - b=\sum_{\ell=1}^s v_\ell w_\ell = 
 0$. Thus $b=a\in A_{ij}$.
 However 
 \begin{eqnarray*}
 \lefteqn{
   b = \sum_{q=1}^n \sum_{\ell=1}^s  (f_q'\varphi(\varphi(v_\ell)w_\ell)+f_q'\varphi(v_\ell)w_\ell)}\\
  &&\hspace{1cm} + \sum_{r=1}^m\sum_{\ell=1}^s  (\varphi(v_\ell\varphi(w_\ell))f_r+v_\ell\varphi(w_\ell)f_r) -
 \sum_{q=1}^n\sum_{r=1}^m \sum_{\ell=1}^s f'_q \varphi(v_\ell)\varphi(w_\ell) f_r.
 \end{eqnarray*}
  Taking the homogeneous component of $b$, corresponding to $A_{ij}$, i.e. the summand with $q=i$ and $r=j$, we obtain $a=b=0$ since by~(\ref{EquationReesSimpleJ10*jfj0}) and~(\ref{EquationReesSimplefi'J01i*0}) we have
 $$f_i'\varphi(v_\ell)=\varphi(w_\ell)f_j = 0.$$
The projection of $a$ on $B$ with the kernel $J(A)$ yields $\sum_{\ell=1}^s \varphi(v_\ell) \varphi(w_\ell)=0$.
 Hence $$\sum_{\ell=1}^s \varphi(v_\ell) \varphi(w_\ell) = \sum_{\ell=1}^s v_\ell \varphi(w_\ell)
= \sum_{\ell=1}^s \varphi(v_\ell) w_\ell = 0.$$

Now we are ready to prove that the sum in the definition of $\tilde A_{ij}$
is direct.

Suppose \begin{equation*}\begin{split}\sum_{\ell=1}^s \left(\varphi(v_\ell)\varphi(w_\ell)+v_\ell \varphi(w_\ell)
+  \varphi(v_\ell) w_\ell + v_\ell w_\ell \right) \in \\ \in f_i' B f_j\oplus \left\lbrace \varphi(v)+v \mid  
v \in J^{10}_{ij}\oplus J^{01}_{ij} \right\rbrace \subseteq B \oplus \bigoplus_{j=1}^m J^{10}_{*j}\oplus\bigoplus_{i=1}^n J^{01}_{i*}\end{split}\end{equation*} for some $  
v_\ell \in J^{01}_{i*}$ and $w_\ell \in J^{10}_{*j}$.
Since $\sum_{\ell=1}^s v_\ell w_\ell \in J(A)^2$
and $$\sum_{\ell=1}^s \left(\varphi(v_\ell)\varphi(w_\ell)+v_\ell \varphi(w_\ell)
+  \varphi(v_\ell) w_\ell\right) \in B \oplus \bigoplus_{j=1}^m J^{10}_{*j}\oplus\bigoplus_{i=1}^n J^{01}_{i*},$$ we have $\sum_{\ell=1}^s  v_\ell w_\ell=0$
and, by the previous remarks, $$\sum_{\ell=1}^s \varphi(v_\ell)\varphi(w_\ell)
=\sum_{\ell=1}^s v_\ell\varphi(w_\ell) = \sum_{\ell=1}^s \varphi(v_\ell)w_\ell = 0$$
and $\sum_{\ell=1}^s \left(\varphi(v_\ell)\varphi(w_\ell)+v_\ell \varphi(w_\ell)
+  \varphi(v_\ell) w_\ell + v_\ell w_\ell \right) = 0$.

In particular, the sum in the definition of $\tilde A_{ij}$
is direct and the proof of~(\ref{EquationDecompAijReesSimple}) is complete.

Now we prove that the sum $J(A)^2=\sum_{i=1}^n \sum_{j=1}^m  J^{01}_{i*}J^{10}_{*j}$
is direct. Indeed, suppose $\sum_{i=1}^n \sum_{j=1}^m u_{ij} = 0$
for some $u_{ij}\in J^{01}_{i*}J^{10}_{*j}$.
By~~(\ref{EquationDecompAijReesSimple}), $u_{ij}= a_{ij}-v_{ij}$ where $a_{ij}\in A_{ij}$
and $v_{ij}$ is a linear combination of homogeneous elements
from $B$ and homogeneous elements from 
$\left\lbrace\varphi(v)+ v\mid  
v \in J^{01}_{i*} \oplus J^{10}_{*j} \right\rbrace$.
Now $\sum_{i=1}^n \sum_{j=1}^m (a_{ij}-v_{ij}) = 0$
implies that each $a_{ij}$ is a linear combination of elements
from $f_i' B f_j$ and elements from 
$\left\lbrace\varphi(v)+ v\mid  
v \in J^{01}_{ij} \oplus J^{10}_{ij} \right\rbrace$.
Thus $u_{ij}=(1-1_B)u_{ij}(1-1_B)=(1-1_B)(a_{ij}-v_{ij})(1-1_B) = 0$
and the sum $J(A)^2=\bigoplus_{i=1}^n \bigoplus_{j=1}^m  J^{01}_{i*}J^{10}_{*j}$
is indeed direct.

Only ~(\ref{EquationDimJAReesSemiGr}) still has to be proved. 
Note that Lemma~\ref{LemmaMatrixIdealsProduct}
implies  
    \begin{eqnarray*}
      \dim_F J(A)^2 &= & \sum_{i=1}^n \sum_{j=1}^m \dim_F  J^{01}_{i*}J^{10}_{*j}
                                =\sum_{i=1}^n \sum_{j=1}^m \dim_F  \varphi(J^{01}_{i*})\varphi(J^{10}_{*j})\\
                               &=&\sum_{i=1}^n \sum_{j=1}^m \frac{\dim_F \varphi( J^{01}_{i*})\dim_F 
                               \varphi(J^{10}_{*j})}{k^2 \dim_F D}\\
                               & \leqslant & \sum_{i=1}^n \sum_{j=1}^m \frac{(\dim_F B - \dim_F f_i' B )
                                    (\dim_F B - \dim_F B f_j)}{k^2 \dim_F D}\\
                                &=&\frac{(\dim_F B)^2(n-1)(m-1)}{k^2 \dim_F D}\\
                                &=&(n-1)(m-1) \dim_F B.   
   \end{eqnarray*}
Finally, (\ref{EquationDimJAReesSemiGr}) follows at once from
(\ref{EquationDimJ01}) and (\ref{EquationDimJ10}).
\end{proof}

As an example, we conclude this section with a specific class of algebras for which we give an explicit description of the graded Malcev-Wedderburn decomposition.
Let $R$ be a finite dimensional $F$-algebra with  identity element  and let $P$ be an  $m \times n$ matrix with entries in $R$ such that each row and each column contains at least one invertible element in $R$. The Munn algebra $A:=\mathcal{M}(R,n,m,P)$ is, by definition, the $F$-vector space of all $n\times m$-matrices over $R$ with multiplication defined by $ D E : = D \circ P \circ E$, for  $D, E \in \mathcal{M}(R,n,m,P)$, and  where $D \circ P$ is the usual matrix multiplication.
Clearly, $A=\oplus_{i,j}Re_{ij}$ and it is an  $S$-graded-simple algebra, where $S$
is the completely $0$-simple semigroup $S=\mathcal{M}(\{ e \}^0, n,m,P')$, with $P'$ the $n\times m$-matrix with  $e$ in every  entry of $P'$.

By  Lemma~\ref{lem1} and  Theorem~\ref{TheoremBGradedReesSemigroupGrSimple},   there exists an $S$-graded maximal semisimple subalgebra $B$ such that $A=B\oplus J(A)$. 
In case $R=F$, one can give an explicit description of $B$.
Indeed, let $k$ denote the rank of $P$. Reindexing if needed, we may assume that the first $k$ rows and the first $k$ columns are $F$-linearly independent.
Let $B = \{ (a_{ij}) \in A \mid a_{ij}=0 \mbox{ for } i>k \mbox{ or } j>k\}$.
Note that $B$ can be identified with $\mathcal{M}(F,k,k,Q)$
where  the sandwich matrix $Q$ consists of the first $k$ rows and columns of $P$.
Clearly, $B$ is a graded subalgebra and the mapping $N \mapsto N \circ Q$ is an algebra isomorphism $B \rightarrow M_k(F)$.
Thus $B \cong M_k(F)$ is a simple algebra.
Furthermore,  $A/J(A) \cong M_{k}(F)\cong B$ (Proposition 23 in Chapter 5 of \cite{okninski1}) and thus $A=B\oplus J(A)$ and $A$ is Wedderburn-Malcev graded.
Also recall that  $J(A) =\{ N \in A \mid P\circ N \circ P=0\}$ (see e.g. Corollary 15 in Chapter 5 of \cite{okninski1}).

\section{Existence theorems for graded-simple algebras}\label{SectionTGradedReesExistence}

In Theorem~\ref{TheoremBGradedReesSemigroupGrSimple} and Theorem~\ref{W-M S-graded-simple} we obtained a description of an $S$-graded-simple algebra $A$.
It is shown that $A$ has a graded Malcev-Wedderburn decomposition $B+J(A)$ and that $J(A)$ is roughly the direct sum of left and right $B$-modules that are isomorphic to left and right ideals of $B$ and that  also satisfy some other restrictions.
To complete the description, we now show that any such collection of left and right ideals of a finite dimensional simple algebra $B$ yields an $S$-graded-simple  algebra $A$ with $B$ as a maximal semsimple graded subalgebra.
We now formulate this in precise detail.

Let $k,m,n\in\mathbb N$ and let $D$ be a division algebra.
Put $B \cong M_k(D)$ and assume
 $f_1, \ldots, f_n \in B$ and  $f'_1, \ldots, f'_n \in B$ are  two sets of idempotents
(some of them could be zero) such that $\sum_{i=1}^n f_i' = \sum_{j=1}^m f_j = 1_B$
and so that the idempotents in each set are pairwise orthogonal.

Let $J^{10}_{*1}, \ldots, J^{10}_{*m}$ and 
$J^{01}_{1*}, \ldots, J^{01}_{n*}$ be, respectively, left and right $B$-modules
such that there exist embeddings 
   $$\varphi \colon J^{10}_{*j} \hookrightarrow B \quad \text{ and } \quad 
   \varphi \colon J^{01}_{i*} \hookrightarrow B$$
which are, respectively, left and right $B$-module homomorphisms and we have 
   $$\varphi(J^{10}_{*j}) f_j = 0 \quad \text{ and  } \quad  f_i' \varphi (J^{01}_{i*})= 0.$$
(For  convenience, we denote both the maps by the same letter $\varphi$
and we also assume that the linear map $\varphi$ is defined on the additive group $\bigoplus_{i=1}^n J^{01}_{i*} \oplus \bigoplus_{j=1}^m J^{10}_{*j}$.)
Define additive groups $J_{ij}$ as isomorphic copies of $\varphi(J^{01}_{i*})\varphi(J^{10}_{*j}) \subseteq B$ for $1\leqslant i \leqslant n$, $1\leqslant j \leqslant m$.
Let 
    $$\Theta_{ij} \colon \varphi(J^{01}_{i*})\varphi(J^{10}_{*j}) \rightarrow
J_{ij}$$ 
be the corresponding linear isomorphisms and let 
  $$\mu \colon J^{01}_{i*} \times J^{10}_{*j}
\rightarrow J_{ij}$$ be the bilinear map defined by 
  $$\mu(v,w):=\Theta_{ij}(\varphi(v)\varphi(w)) ,$$
for $v\in J^{01}_{i*}$ and $w\in J^{10}_{*j}$.
We extend $\mu$ by  linearity to the map 
   $$\mu \colon \bigoplus_{i=1}^n J^{01}_{i*} \times
\bigoplus_{j=1}^m J^{10}_{*j}
\rightarrow \bigoplus_{i=1}^n \bigoplus_{j=1}^m J_{ij} .$$
Let 
$Q$ denote  the $n\times m$ matrix with each entry being equal to $e$.

\begin{theorem}\label{TheoremTGradedReesExistence}
The additive group
   $$A=B \oplus \bigoplus_{i=1}^n J^{01}_{i*} \oplus \bigoplus_{j=1}^m J^{10}_{*j} 
\oplus \bigoplus_{i=1}^n \bigoplus_{j=1}^m J_{ij}$$
is a $\mathcal{M}(\{ e\}^{0},n,m;Q)$-graded-simple ring for  the multiplication defined by 
   $$(b_1,v_1,w_1,u_1)(b_2,v_2,w_2,u_2) =(b_1 b_2, v_1 b_2, b_1 w_2, \mu(v_1, w_2)),$$ 
for $b_1,b_2 \in B$, $v_1,v_2\in \bigoplus_{i=1}^n J^{01}_{i*}$, $w_1,w_2 \in \bigoplus_{j=1}^m J^{10}_{*j}$, $u_1,u_2 \in \bigoplus_{i=1}^n \bigoplus_{j=1}^m J_{ij}$.
The homogeneous components are    
  \begin{eqnarray*}
  \lefteqn{\hspace{-2cm}
  A_{ij}=(f_i' B f_j,0,0,0) \oplus \left\lbrace (\varphi(v),v,0, 0) \mid v \in 
J^{01}_{i*}f_j \right\rbrace 
 \oplus \left\lbrace (\varphi(w),0,w, 0) \mid w \in 
 f_i'J^{10}_{*j} \right\rbrace }\\
  && \oplus \left\langle 
(\varphi(v)\varphi(w), v\varphi(w), \varphi(v)w, \mu(v,w)) \mid 
v \in 
J^{01}_{i*},\ w \in J^{10}_{*j} \right\rangle_{\mathbb Z}.
  \end{eqnarray*}
   
\end{theorem}

\begin{proof} 
Making use of all the assumptions, direct computations show that the multiplication defines a graded ring structure  and also that 
$A = \bigoplus_{i=1}^n \bigoplus_{j=1}^m A_{ij}$
%

Clearly, $J(A)=(0,\bigoplus_{i=1}^n J^{01}_{i*}, \bigoplus_{j=1}^m J^{10}_{*j}, 
 \bigoplus_{i=1}^n \bigoplus_{j=1}^m J_{ij})$, since  the third power of the right hand sided is zero.
 
  Note that for every $1\leqslant i \leqslant n$ and $1\leqslant j \leqslant m$
 we have $$1_B (f_i' B f_j,0,0,0) 1_B = \left(f_i' B f_j,0,0,0\right),$$
 $$1_B \left\lbrace (\varphi(v),v,0, 0) \mid v \in 
J^{01}_{i*}f_j \right\rbrace 1_B \subseteq \left(\bigoplus_{\ell \ne i} f_\ell' B f_j,0,0,0\right),$$
$$1_B \left\lbrace (\varphi(w),0,w, 0) \mid w \in 
 f_i'J^{10}_{*j} \right\rbrace 1_B \subseteq \left(\bigoplus_{r \ne j} f_i' B f_r,0,0,0\right),$$
 $$1_B \left\langle 
(\varphi(v)\varphi(w), v\varphi(w), \varphi(v)w, \mu(v,w)) \mid 
v \in 
J^{01}_{i*},\ w \in J^{10}_{*j} \right\rangle_{\mathbb Z} 
  1_B \subseteq \left(\bigoplus_{\substack{\ell\ne i, \\ r \ne j}} f_\ell' B f_r,0,0,0\right).$$
  Thus, if $1_B a 1_B  = 0$ for some $a\in A_{ij}$, then $a = 0$.
  Hence $A_{ij} \cap J(A)=0$.
  
 Suppose $I$ is a graded two-sided ideal of $A$. Let $a\in I$, $a\ne 0$, be a homogeneous
element. By the previous,  $a=(b,u,v,w)$ with  $b\ne 0$. Hence $(1_B,0,0,0) a (1_B,0,0,0) = (b,0,0,0)\in I$.
Since $B$ is a simple ring, $(B,0,0,0)\subseteq I$.
Thus $(1_B,0,0,0) A \subseteq I$ and $A(1_B,0,0,0)\subseteq I$. Since $$A = (B,0,0,0)+ (1_B,0,0,0) A
+ A(1_B,0,0,0) + (1_B,0,0,0) A^2 (1_B,0,0,0),$$
we get $I=R$, and $A$ is graded-simple.
\end{proof}

In case all modules involved, for example $J_{i*}^{01}$ and $J_{ij}$, are left and right $F$-vector spaces on which the left and right $F$-structure is compatible and all maps involved are (left and right) $F$-linear, then the construction in Theorem~\ref{TheoremTGradedReesExistence} yields an $S$-graded algebra. In this case, if  $B$ is a finite dimensional algebra over $F$, then a left $B$-module embedding 
$\varphi  \colon J^{10}_{*j} \rightarrow \bigoplus_{\substack{ r=1,\\r \ne j}}^m B f_r$
 exists if and only if $\dim_F J^{10}_{*j} \leqslant \dim_F B - \dim_F ( B f_j).$
A right $B$-module embedding $$\varphi \colon J^{01}_{*i} \rightarrow \bigoplus_{\ell \ne i} f'_\ell B$$ exists
if and only if $$\dim_F J^{01}_{i*} \leqslant \dim_F B - \dim_F (f'_i B).$$


Theorem~\ref{TheoremEquivalenceSemigroupGradedSimple} shows that the grading on an algebra $A$ is completely
defined by the images of the graded components in $A/J(A)$.
We show that every such decomposition determines some $S$-grading.

\begin{theorem}\label{TheoremImagesOfGradedComponentsReesExistence}
Let $D$ be a division algebra over a field $F$ and let $B \cong M_k(D)$.
Suppose $B = \sum_{i=1}^n \sum_{j=1}^m B_{ij}$ , a sum of subspaces $B_{ij}$ of $B$, with 
 $B_{ij} B_{\ell r} \subseteq B_{ir}$,
for all $1\leqslant i,\ell \leqslant n$, $1\leqslant j,r \leqslant m$.
Let $P=(p_{ij})_{i,j}$ be an $n\times m$ matrix where $p_{ij}\in \{ 0,e\}$ such that $B_{ij} B_{\ell r}=0$ for every $(j,\ell)$ with $p_{j\ell} = 0$.
Then there exists an $\mathcal{M}(\{ e\}^{0},n,m;P)$-graded algebra $A=\bigoplus_{i=1}^n \bigoplus_{j=1}^m A_{ij}$ and a surjective algebra homomorphism $\psi \colon A \rightarrow  B$ such that $\ker \psi = J(A)$ and $\psi(A_{ij})=B_{ij}$,  or all $1\leqslant i,\ell \leqslant n$, $1\leqslant j,r \leqslant m$.

\end{theorem}
\begin{proof} Let $\bar L_j := \bigoplus_{i=1}^n B_{ij}$ for $1\leqslant j \leqslant m$
and $\bar R_i := \bigoplus_{j=1}^m B_{ij}$ for $1\leqslant i \leqslant n$. Note that
$\bar L_1, \ldots, \bar L_m$ are left ideals and $\bar R_1, \ldots, \bar R_n$
are right ideals. Moreover, since $B \cong M_k(D)$ is semisimple, $B$
is completely reducible as a left and a right $B$-module.
Define $\tilde L_j$ as a complementary left $B$-submodule to $\bar L_j\cap \sum_{\ell=1}^{j-1} \bar L_\ell$ in $\bar L_j$, $1\leqslant j \leqslant m$. Analogously, define $\tilde R_i$ as a complementary right $B$-submodule to $\bar R_i\cap \sum_{\ell=1}^{i-1} \bar R_\ell$ in $\bar R_i$, $1\leqslant i \leqslant n$.
Then $\bigoplus_{\ell=1}^i \tilde R_\ell = \sum_{\ell=1}^i \bar R_\ell$ for $1\leqslant i \leqslant n$
and $\bigoplus_{\ell=1}^j \tilde L_\ell = \sum_{\ell=1}^j \bar L_\ell$ for $1\leqslant j \leqslant m$.
In particular, $B = \bigoplus_{i=1}^n \tilde R_i = \bigoplus_{j=1}^m \tilde L_j$.
Decomposing $1_B$ into the sum of elements of, respectively, $\tilde L_1, \ldots, \tilde L_m$ and $\tilde R_1, \ldots, \tilde R_n$,
we find two sets of orthogonal idempotents $f_1, \ldots, f_m$ and $f_1', \ldots, f_n'$
such that $\tilde L_j = B f_j$, $\tilde R_i = B f_i'$, $\sum_{i=1}^n f_i' = \sum_{j=1}^m f_j = 1_B$.

Let $W^{10}_{*j}=\bar L_j (1-f_j) \subseteq \bar L_j$ for $1\leqslant j \leqslant m$ 
and $W^{01}_{i*}=(1-f_i')\bar R_i \subseteq \bar R_i$ for $1\leqslant i \leqslant n$.
Then $\bar L_j = \bar L_j f_j \oplus W^{10}_{*j}$ and $\bar R_i = f_i' \bar R_i \oplus W^{01}_{i*}$ (direct sums of, respectively, left and right ideals). Denote by $J^{10}_{*j}$ the isomorphic copy of $W^{10}_{*j}$
and by $J^{01}_{i*}$ the isomorphic copy of $W^{01}_{i*}$.
Denote the corresponding left $B$-module isomorphisms $\varphi \colon J^{10}_{*j} \rightarrow W^{10}_{*j}$ and right $B$-module isomorphisms $\varphi \colon J^{01}_{i*} \rightarrow W^{01}_{i*}$ by the same letter $\varphi$.
Now extend $\varphi$ to an $F$-linear map
$$\varphi \colon 
\bigoplus_{i=1}^n J^{01}_{i*} \oplus \bigoplus_{j=1}^m J^{10}_{*j} \rightarrow
\sum_{i=1}^n W^{01}_{i*} + \sum_{j=1}^m W^{10}_{*j} \subseteq B.$$

Let $A=\bigoplus_{i=1}^n\bigoplus_{j=1}^m A_{ij}$ be the ring constructed in Theorem~\ref{TheoremTGradedReesExistence}. Note that this then will be an algebra by the remark given after the proof of the theorem.
We claim that $A$ satisfies all the conditions of Theorem~\ref{TheoremImagesOfGradedComponentsReesExistence}.
Indeed, define $\psi$ as the projection of $$A=B \oplus \bigoplus_{i=1}^n J^{01}_{i*} \oplus \bigoplus_{j=1}^m J^{10}_{*j} 
\oplus \bigoplus_{i=1}^n \bigoplus_{j=1}^m J_{ij}$$ on $B$ with the kernel $$J(A)=\bigoplus_{i=1}^n J^{01}_{i*} \oplus \bigoplus_{j=1}^m J^{10}_{*j} 
\oplus \bigoplus_{i=1}^n \bigoplus_{j=1}^m J_{ij}.$$
Since $B$ is semisimple, there exist idempotents $e_i'$ and $e_j$
such that $\bar R_i = e'_i B$ and $\bar L_j = B e'_j$.
Then 
   \begin{eqnarray*}
   \psi(A_{ij}) &=& f_i' B f_j' + \varphi(J^{01}_{i*}f_j) + \varphi(f_i'J^{10}_{*j})
+\varphi(J^{01}_{i*}) \varphi(J^{10}_{*j}) 
    \\ &=& f_i' B f_j + (1-f_i')\bar R_i f_j + f_i' \bar L_j (1-f_j) +(1-f_i')\bar R_i \bar L_j (1-f_j)\\
       &=&f_i' \bar R_i f_j + (1-f_i')\bar R_i f_j +f_i' \bar R_i \bar L_j (1-f_j) +(1-f_i')\bar R_i \bar L_j (1-f_j)
   \\ &=& \bar R_i f_j + \bar R_i \bar L_j (1-f_j)  = \bar R_i \bar L_j =e_i' B e_j 
     = \bar R_i \cap \bar L_j\\
     &=&B_{ij}.
  \end{eqnarray*}
Moreover, $p_{j\ell}=0$ always implies $B_{ij}B_{\ell r} = 0$, $\psi(A_{ij}A_{\ell r})=0$ and $A_{ij}A_{\ell r} = 0$. The last equality follows from $(\ker \psi)  \cap A_{ir}=J(A)\cap A_{ir}=0$.
\end{proof}

Note that any simple algebra $B=M_{k}(D)$ can be decomposed into the sum of $B_{ij}$'s (as needed in Theorem~\ref{TheoremImagesOfGradedComponentsReesExistence}) by taking  
a collection $\bar L_1, \ldots, \bar L_m$ and a collection $\bar R_1,\ldots, \bar R_n$ of, respectively, left and right ideals of
$B$ with  $B=\sum_{i=1}^n \bar R_i = \sum_{j=1}^m \bar L_j$ and define  $B_{ij}=\overline{R}_{i} \cap \overline{L}_{j}$.

\section{Graded polynomial identities, their codimensions and cocharacters}\label{SectionGradedPIAss}

In the rest of the paper we study numeric characteristics of graded polynomial identities in semigroup graded algebras.

Let $T$ be a semigroup and let $F$ be a field. Denote by $F\langle X^{T\text{-}\mathrm{gr}} \rangle $ the free $T$-graded associative  algebra over $F$ on the set $$X^{T\text{-}\mathrm{gr}}:=\bigcup_{t \in T}X^{(t)},$$ $X^{(t)} = \{ x^{(t)}_1,
x^{(t)}_2, \ldots \}$,  i.e. the algebra of polynomials
 in non-commuting variables from $X^{T\text{-}\mathrm{gr}}$.
  The indeterminates from $X^{(t)}$ are said to be homogeneous of degree
$t$. The $T$-degree of a monomial $x^{(t_1)}_{i_1} \dots x^{(t_s)}_{i_s} \in F\langle
 X^{T\text{-}\mathrm{gr}} \rangle $ is defined to
be $t_1 t_2 \dots t_s$, as opposed to its total degree, which is defined to be $s$. Denote by
$F\langle
 X^{T\text{-}\mathrm{gr}} \rangle^{(t)}$ the subspace of the algebra $F\langle
 X^{T\text{-}\mathrm{gr}} \rangle$ spanned
 by all the monomials having
$T$-degree $t$. Notice that $$F\langle
 X^{T\text{-}\mathrm{gr}} \rangle^{(t)} F\langle
 X^{T\text{-}\mathrm{gr}} \rangle^{(h)} \subseteq F\langle
 X^{T\text{-}\mathrm{gr}} \rangle^{(th)},$$ for every $t, h \in T$. It follows that
$$F\langle
 X^{T\text{-}\mathrm{gr}} \rangle =\bigoplus_{t\in T} F\langle
 X^{T\text{-}\mathrm{gr}} \rangle^{(t)}$$ is a $T$-grading.
  Let $f=f(x^{(t_1)}_{i_1}, \dots, x^{(t_s)}_{i_s}) \in F\langle
 X^{T\text{-}\mathrm{gr}} \rangle$.
We say that $f$ is
a \textit{graded polynomial identity} of
 a $T$-graded algebra $A=\bigoplus_{t\in T}
A^{(t)}$
and write $f\equiv 0$
if $f(a^{(t_1)}_{i_1}, \dots, a^{(t_s)}_{i_s})=0$
for all $a^{(t_j)}_{i_j} \in A^{(t_j)}$, $1 \leqslant j \leqslant s$.
  The set $\Id^{T\text{-}\mathrm{gr}}(A)$ of graded polynomial identities of
   $A$ is
a graded ideal of $F\langle
 X^{T\text{-}\mathrm{gr}} \rangle$.

\begin{example}\label{ExampleIdGr}
Consider the multiplicative semigroup $T=\mathbb Z_2 = \lbrace \bar 0, \bar 1 \rbrace$
and the $T$-grading $\UT_2(F)=\UT_2(F)^{(\bar 0)}\oplus \UT_2(F)^{(\bar 1)}$
on the algebra $\UT_2(F)$ of upper triangular $2\times 2$ matrices over a field $F$
defined by
 $\UT_2(F)^{(\bar 0)}=\left(
\begin{array}{cc}
F & 0 \\
0 & F
\end{array}
 \right)$ and $\UT_2(F)^{(\bar 1)}=\left(
\begin{array}{cc}
0 & F \\
0 & 0
\end{array}
 \right)$. We have $$[x^{(\bar 0)}, y^{(\bar 0)}]:=x^{(\bar 0)} y^{(\bar 0)} - y^{(\bar 0)} x^{(\bar 0)}
\in \Id^{T\text{-}\mathrm{gr}}(\UT_2(F))$$
and $x^{(\bar 1)} y^{(\bar 1)} 
\in \Id^{T\text{-}\mathrm{gr}}(\UT_2(F))$.
\end{example}

Let
$P^{T\text{-}\mathrm{gr}}_n := \langle x^{(t_1)}_{\sigma(1)}
x^{(t_2)}_{\sigma(2)}\ldots x^{(t_n)}_{\sigma(n)}
\mid t_i \in T, \sigma\in S_n \rangle_F \subset F \langle X^{T\text{-}\mathrm{gr}} \rangle$, $n \in \mathbb N$.
Then the number $$c^{T\text{-}\mathrm{gr}}_n(A):=\dim\left(\frac{P^{T\text{-}\mathrm{gr}}_n}{P^{T\text{-}\mathrm{gr}}_n \cap \Id^{T\text{-}\mathrm{gr}}(A)}\right)$$
is called the $n$th \textit{codimension of graded polynomial identities}
or the $n$th \textit{graded codimension} of $A$.

If $T= \lbrace e\rbrace$ is the trivial group, we get ordinary polynomial identities $\Id(A):=\Id^{\lbrace e\rbrace\text{-}\mathrm{gr}}(A)$,
the space of ordinary multilinear polynomials $P_n:=P_n^{\lbrace e\rbrace\text{-}\mathrm{gr}}$ and ordinary codimensions $c_n(A):=c_n^{\lbrace e\rbrace\text{-}\mathrm{gr}}(A)$. 

The proposition provides a relation between the ordinary and the graded codimensions.
\begin{proposition} 
Let $A$ be a $T$-graded algebra over a field $F$ for some semigroup $T$
not necessarily finite.
Then $c_n(A) \leqslant c^{T\text{-}\mathrm{gr}}_n(A)$.
If $T$ is finite, then $c^{T\text{-}\mathrm{gr}}_n(A)
\leqslant |T|^n c_n(A)$ for all $n\in\mathbb N$.
\end{proposition}
\begin{proof}
Let $t_1, \ldots, t_n \in T$.
Denote by $P_{t_1, \ldots, t_n}$ the vector space of
multilinear polynomials in $x_1^{(t_1)}, \ldots, x_n^{(t_n)}$.
Then $P_n^{T\text{-}\mathrm{gr}}=\bigoplus_{t_1, \ldots, t_n \in T}
P_{t_1, \ldots, t_n}$. Let $\bar f_1, \ldots, \bar f_{c_n(A)}$ be 
 a basis in $\frac{P_n}{P_n \cap \Id(A)}$ where $f_i\in P_n$.
Then for every $\sigma \in S_n$
there exist $\alpha_{\sigma,i}\in F$ such that $$x_{\sigma(1)}\ldots x_{\sigma(n)} - \sum_{i=1}^{c_n(A)}\alpha_{\sigma,i} f_i(x_1,\ldots, x_n) \in \Id(A).$$
Then for every $t_1, \ldots, t_n \in T$
we have $$x^{(t_1)}_{\sigma(1)}\ldots x^{(t_n)}_{\sigma(n)} - \sum_{i=1}^{c_n(A)}\alpha_{\sigma,i} f_i\left(x^{(t_1)}_1,\ldots, x^{(t_n)}_n\right) \in \Id^{T\text{-}\mathrm{gr}}(A)$$
and $$\frac{P^{T\text{-}\mathrm{gr}}_n}{P^{T\text{-}\mathrm{gr}}_n \cap \Id^{T\text{-}\mathrm{gr}}(A)}
=\left\langle \bar f_i\left(x^{(t_1)}_1,\ldots, x^{(t_n)}_n\right)\mathrel{\Bigl|}
1\leqslant i \leqslant c_n(A),\  t_1, \ldots, t_n \in T\right\rangle_F.$$
This implies the upper bound.

In order to get the lower bound, for a given $n$-tuple $(t_1, \ldots, t_n) \in T^n$
we consider the map $\varphi_{t_1, \ldots, t_n} \colon P_n \rightarrow \frac{P^{T\text{-}\mathrm{gr}}_n}{P^{T\text{-}\mathrm{gr}}_n \cap \Id^{T\text{-}\mathrm{gr}}(A)}$
where $\varphi_{t_1, \ldots, t_n}(f)=f\left(x^{(t_1)}_1, \ldots, x^{(t_n)}_n\right)$ for $f=f(x_1, \ldots, x_n) \in P_n$.
Note that $f(x_1, \ldots, x_n) \equiv 0$ is an ordinary polynomial
identity if and only if $$f\left(x^{(t_1)}_1, \ldots, x^{(t_n)}_n\right) \equiv 0 $$ is a graded polynomial identity for every $t_1, \ldots, t_n \in T$. In other words, $P_n \cap \Id(A)
= \bigcap\limits_{(t_1, \ldots, t_n) \in T^n} \ker \varphi_{t_1, \ldots, t_n}$.
Since $P_n$ is a finite dimensional vector space, there exists a finite subset $\Lambda \subseteq T^n$
such that $P_n \cap \Id(A)
= \bigcap\limits_{(t_1, \ldots, t_n) \in \Lambda} \ker \varphi_{t_1, \ldots, t_n}$.

Consider the
 diagonal embedding $$P_n \hookrightarrow
 P_n^{T\text{-}\mathrm{gr}}=\bigoplus_{t_1, \ldots, t_n \in T}
P_{t_1, \ldots, t_n}$$ where the image of $f(x_1, \ldots, x_n)\in P_n$ equals $\sum_{(t_1, \ldots, t_n) \in \Lambda} f\left(x^{(t_1)}_1, \ldots, x^{(t_n)}_n\right)$. 
Then our choice of $\Lambda$ implies that the induced map $\frac{P_n}{P_n \cap \Id(A)} \hookrightarrow \frac{P^{T\text{-}\mathrm{gr}}_n}{P^{T\text{-}\mathrm{gr}}_n \cap \Id^{T\text{-}\mathrm{gr}}(A)}$ is an embedding and the lower bound follows.
\end{proof}

The analog of Amitsur's conjecture for graded codimensions can be formulated
as follows.

\begin{conjecture} There exists
 $\PIexp^{T\text{-}\mathrm{gr}}(A):=\lim\limits_{n\rightarrow\infty} \sqrt[n]{c^{T\text{-}\mathrm{gr}}_n(A)}$.
\end{conjecture}

\begin{remark}
The original Amitsur's conjecture was formulated for ordinary polynomial identities and asserted
that the exponent of codimension growth is integer. However, the recent examples~\cite[Theorems 3--5]{ASGordienko13} show that, in the semigroup graded case, the graded PI-exponent could be non-integer.
\end{remark}

\section{Polynomial $H$-identities and their codimensions}\label{SectionFTActionGr}  

In our case, instead of working with graded codimensions directly, it is more convenient to replace the grading with the corresponding dual structure and study the asymptotic behaviour of polynomial $H$-identities for a suitable associative algebra $H$.

  Let $H$ be an arbitrary associative algebra with $1$ over a field $F$.
We say that an associative algebra $A$ is an algebra with a \textit{generalized $H$-action}
if $A$ is endowed with a homomorphism $H \rightarrow \End_F(A)$
and for every $h \in H$ there exist $h'_i, h''_i, h'''_i, h''''_i \in H$, $1\leqslant i\leqslant s$, $s\in\mathbb N$,
such that 
\begin{equation}\label{EqGenHAction}
h(ab)=\sum_{i=1}^s\bigl((h'_i a)(h''_i b) + (h'''_i b)(h''''_i a)\bigr) \text{ for all } a,b \in A.
\end{equation}

\begin{remark}
We use the term ``generalized $H$-action'' in order to distinguish from the case when an algebra is
an $H$-module algebra for some Hopf algebra $H$ which is a particular case of the generalized $H$-action.
\end{remark}

Let $F \langle X \rangle$ be the free associative algebra without $1$
   on the set $X := \lbrace x_1, x_2, x_3, \ldots \rbrace$.
  Then $F \langle X \rangle = \bigoplus_{n=1}^\infty F \langle X \rangle^{(n)}$
  where $F \langle X \rangle^{(n)}$ is the linear span of all monomials of total degree $n$.
    Consider the algebra $$F \langle X | H\rangle
   :=  \bigoplus_{n=1}^\infty H^{{}\otimes n} \otimes F \langle X \rangle^{(n)}$$
   with the multiplication $(u_1 \otimes w_1)(u_2 \otimes w_2):=(u_1 \otimes u_2) \otimes w_1w_2$
   for all $u_1 \in  H^{{}\otimes j}$, $u_2 \in  H^{{}\otimes k}$,
   $w_1 \in F \langle X \rangle^{(j)}$, $w_2 \in F \langle X \rangle^{(k)}$.
We use the notation $$x^{h_1}_{i_1}
x^{h_2}_{i_2}\ldots x^{h_n}_{i_n} := (h_1 \otimes h_2 \otimes \ldots \otimes h_n) \otimes x_{i_1}
x_{i_2}\ldots x_{i_n}.$$ Here $h_1 \otimes h_2 \otimes \ldots \otimes h_n \in H^{{}\otimes n}$,
$x_{i_1} x_{i_2}\ldots x_{i_n} \in F \langle X \rangle^{(n)}$. 

Note that if $(\gamma_\beta)_{\beta \in \Lambda}$ is a basis in $H$, 
then $F\langle X | H \rangle$ is isomorphic to the free associative algebra over $F$ with free formal  generators $x_i^{\gamma_\beta}$, $\beta \in \Lambda$, $i \in \mathbb N$.
 We refer to the elements
 of $F\langle X | H \rangle$ as \textit{associative $H$-polynomials}.
Note that here we do not consider any $H$-action on $F \langle X | H \rangle$.

Let $A$ be an associative algebra with a generalized $H$-action.
Any map $\psi \colon X \rightarrow A$ has the unique homomorphic extension $\bar\psi
\colon F \langle X | H \rangle \rightarrow A$ such that $\bar\psi(x_i^h)=h\psi(x_i)$
for all $i \in \mathbb N$ and $h \in H$.
 An $H$-polynomial
 $f \in F\langle X | H \rangle$
 is an \textit{$H$-identity} of $A$ if $\bar\psi(f)=0$
for all maps $\psi \colon X \rightarrow A$. In other words, $f(x_1, x_2, \ldots, x_n)$
 is an $H$-identity of $A$
if and only if $f(a_1, a_2, \ldots, a_n)=0$ for any $a_i \in A$.
 In this case we write $f \equiv 0$.
The set $\Id^{H}(A)$ of all $H$-identities
of $A$ is an ideal of $F\langle X | H \rangle$.

We denote by $P^H_n$ the space of all multilinear $H$-polynomials
in $x_1, \ldots, x_n$, $n\in\mathbb N$, i.e.
$$P^{H}_n = \langle x^{h_1}_{\sigma(1)}
x^{h_2}_{\sigma(2)}\ldots x^{h_n}_{\sigma(n)}
\mid h_i \in H, \sigma\in S_n \rangle_F \subset F \langle X | H \rangle.$$
Then the number $c^H_n(A):=\dim\left(\frac{P^H_n}{P^H_n \cap \Id^H(A)}\right)$
is called the $n$th \textit{codimension of polynomial $H$-identities}
or the $n$th \textit{$H$-codimension} of $A$.


For an arbitrary semigroup $T$ one can consider the \textit{semigroup algebra} $FT$ over a field $F$ which is the vector space with the formal basis $(t)_{t\in T}$ and the multiplication induced by the one in $T$.

Consider the vector space $(FT)^*$ dual to $FT$. Then $(FT)^*$ is an algebra with the multiplication defined
by $(hw)(t)=h(t)w(t)$ for $h,w \in (FT)^*$ and $t\in T$. The identity element 
is defined by $1_{(FT)^*}(t)=1$ for all $t\in T$. In other words, $(FT)^*$ is the algebra dual to the coalgebra $FT$.

Let $\Gamma \colon A=\bigoplus_{t\in T} A^{(t)}$ be a grading on an algebra $A$. We have the following natural $(FT)^*$-action on $A$: $h a^{(t)}=h(t)a^{(t)}$ for all $h \in (FT)^*$, $a^{(t)}\in A^{(t)}$
and $t\in T$.

\begin{remark}
If $T$ is a finite group, then $A$ is an $FT$-comodule algebra for the Hopf algebra $FT$
and an $(FT)^*$-module algebra for the Hopf algebra $(FT)^*$.
\end{remark}

For every $t\in T$ define $h_t \in (FT)^*$ by $h_t(g):=\left\lbrace \begin{array}{lll}
0 & \text{if} & g\ne t, \\
1 & \text{if} & g = t
\end{array}\right.$ for $g\in T$.

If $A$ is finite dimensional,
the set $\supp \Gamma := \lbrace t\in T
\mid A^{(t)}\ne 0 \rbrace$ is finite and $$h_t(ab)=\sum\limits_{\substack{g, w \in \supp \Gamma,\\
gw=t}} h_g(a)h_w(b)\text{ for all }a,b\in A.$$

Note that $ha = \sum_{t\in \supp \Gamma}h(t)h_t a$ for all $a\in A$
and $h\in (FT)^*$.  By linearity, we get~(\ref{EqGenHAction}). Therefore,  $A$ is an algebra with a generalized $(FT)^*$-action.

The lemma below plays the crucial role in the passage from graded polynomial identities
to polynomial $H$-identities.

\begin{lemma}[{\cite[Lemma~1]{ASGordienko13}}]\label{LemmaCnGrCnGenH}
Let $A$ be a finite dimensional algebra over a field $F$ graded by a semigroup $T$.
Then $c_n^{T\text{-}\mathrm{gr}}(A)=c_n^{(FT)^*}(A)$ for all $n\in \mathbb N$.
\end{lemma}

One of the main tools in the investigation of polynomial
identities is provided by the representation theory of symmetric groups.
 The symmetric group $S_n$  acts
 on the space $\frac {P^H_n}{P^H_{n}
  \cap \Id^H(A)}$
  by permuting the variables.
  Irreducible $FS_n$-modules are described by partitions
  $\lambda=(\lambda_1, \ldots, \lambda_s)\vdash n$ and their
  Young diagrams $D_\lambda$.
   The character $\chi^H_n(A)$ of the
  $FS_n$-module $\frac {P^H_n}{P^H_n
   \cap \Id^H(A)}$ is
   called the $n$th
  \textit{cocharacter} of polynomial $H$-identities of $A$.
  We can rewrite it as
  a sum $$\chi^H_n(A)=\sum_{\lambda \vdash n}
   m(A, H, \lambda)\chi(\lambda)$$ of
  irreducible characters $\chi(\lambda)$.
Let  $e_{T_{\lambda}}=a_{T_{\lambda}} b_{T_{\lambda}}$
and
$e^{*}_{T_{\lambda}}=b_{T_{\lambda}} a_{T_{\lambda}}$
where
$a_{T_{\lambda}} = \sum_{\pi \in R_{T_\lambda}} \pi$
and
$b_{T_{\lambda}} = \sum_{\sigma \in C_{T_\lambda}}
 (\sign \sigma) \sigma$,
be Young symmetrizers corresponding to a Young tableau~$T_\lambda$.
Then $M(\lambda) = FS_n e_{T_\lambda} \cong FS_n e^{*}_{T_\lambda}$
is an irreducible $FS_n$-module corresponding to
 a partition~$\lambda \vdash n$.
  We refer the reader to~\cite{Bahturin, DrenKurs, ZaiGia}
   for an account
  of $S_n$-representations and their applications to polynomial
  identities.

\section{Main steps of the proof of the lower and the upper bounds for graded codimensions}\label{SectionGeoffreyOutline}

Here we give a brief overview of Sections~\ref{SectionUpperFrac}--\ref{SectionFracM2Less}.

Let $A=\bigoplus_{t\in T} A^{(t)}$ be a finite dimensional $T$-graded algebra over a field $F$ of characteristic $0$ for some semigroup $T$ such that $A/J(A) \cong M_k(F)$
for some $k\in\mathbb N$ and $A^{(t)}\cap J(A) = 0$ for all $t\in T$.
(For example, $F$ is algebraically closed and $A$ is a finite dimensional $T$-graded-simple algebra for some $T=\mathcal{M}(\{ e\}^{0},n,m;P)$.)

As mentioned before, since $\frac{P_n^{T\text{-}\mathrm{gr}}(A)}{P_n^{T\text{-}\mathrm{gr}}(A) \cap \Id^{T\text{-}\mathrm{gr}}(A)}$ is an $F S_n$-module, it can be decomposed into the sum of simple $F S_n$-modules $M(\lambda)$. Thus,
$$c_n^{T\text{-}\mathrm{gr}}(A) = \sum\limits_{\lambda \vdash n} m(A, (FT)^*, \lambda) \dim M(\lambda)$$
where $m(A, (FT)^*, \lambda)$ is the multiplicity of $M(\lambda)$ in $\frac{P_n^{T\text{-}\mathrm{gr}}(A)}{P_n^{T\text{-}\mathrm{gr}}(A) \cap \Id^{T\text{-}\mathrm{gr}}(A)}$. A hook with the edge in $(i,j)$
of a Young tableau $T_{\lambda}$ consists of all the boxes to the right and below $(i,j)$ together with the box $(i,j)$ itself. The number of boxes in the $(i,j)$-hook is called its length and is denoted by $h_{ij}$. 

Let $\lambda = (\lambda_1, \ldots, \lambda_q) \vdash n$. Recall that by the hook and the Stirling formula
we have
 \begin{equation} \label{dimensie specht mod}
\begin{array}{lcl}
\dim M(\lambda) & = & \frac{n!}{\prod_{i,j} h_{ij}}\\
& \leqslant & \frac{C \sqrt{n} (\frac{n}{e})^n}{\sqrt{\lambda_1 . \ldots . \lambda_q} (\frac{\lambda_1}{e})^{\lambda_1}. \ldots .(\frac{\lambda_q}{e})^{\lambda_q}}\\
 & = & \frac{C \sqrt{n}}{\sqrt{\lambda_1 . \ldots . \lambda_q}} \left( \frac{1}{(\frac{\lambda_1}{n})^{\frac{\lambda_1}{n}}. \ldots .(\frac{\lambda_q}{n})^{\frac{\lambda_q}{n}}}\right)^n
 \end{array} 
 \end{equation}
 
 for some constant $C \in \mathbb{R}$. Now we compute the exponent in 3 steps. 

\textbf{(a)} First, the multiplicities are polynomially bounded, i.e there exist $a,b \in \mathbb{R}$ such that $\sum\limits_{\lambda \vdash n} m(A, (FT)^*, \lambda) \leqslant a n^b$. (See \cite[Theorem 5]{Gor}.)

\textbf{(b)} Now~(\ref{dimensie specht mod}) implies
$$\mathop{\overline{\lim}}_{n\rightarrow \infty} \sqrt[n]{c_n^{T\text{-}\mathrm{gr}}(A))}
\leqslant \sup\limits_{\substack{\lambda \vdash n, \\ m(A, (FT)^*, \lambda) \ne 0}} \Phi\left(\frac{\lambda_1}{n_1}, \ldots, \frac{\lambda_q}{n_q}\right)$$ where $\Phi(x_1, \ldots, x_q) = \frac{1}{x_1^{x_1}. \ldots . x_q^{x_q}}$.  

 Let $f\in P_n^{(FT)^*}, \lambda \vdash n$ for some $n\in\mathbb N$ and $r = \dim A$. In Lemma \ref{LemmaInequalityLambdaUpperFrac} we prove that if 
\begin{equation} \label{conditions region}
\sum_{i=1}^r \gamma_i \lambda_i \geqslant k \mbox{ or } \lambda_{r+1} > 0,
\end{equation}
for some specific numbers $\gamma_i$ defined in Section \ref{SectionUpperFrac}, then $e^*_{T_\lambda}f \in \Id^{(FT)^*}(A)$ for any Young tableau
$T_\lambda$ of shape $\lambda$. In particular $m(A, (FT)^*, \lambda) =0$ for all partitions satisfying the conditions (\ref{conditions region}) and we may assume that $\Phi$ is defined in $r$ variables.

Furthermore now we can restrict ourselves to the partitions $\lambda = (\lambda_1, \ldots, \lambda_r) \vdash n$ such that \newline $\sum\limits_{i=1}^r \gamma_i \frac{\lambda_i}{n} \leqslant \frac{k}{n}$. In particular, for $n$ large enough, we may consider $\Phi$ on the following area
$$
\Omega:=\lbrace(\alpha_1, \alpha_2, \ldots, \alpha_r)\in \mathbb R^r \mid \alpha_1 \geqslant \ldots\geqslant \alpha_r
\geqslant 0,\ \alpha_1+\alpha_2+\ldots+\alpha_r=1,\ \gamma_1\alpha_1 + \ldots + \gamma_r\alpha_r \leqslant 0 \rbrace.
$$

Formula (\ref{dimensie specht mod}) now yields the upper bound 
$$ \dim M(\lambda) \leqslant \max_{(\alpha_1, \ldots, \alpha_r)\in\Omega} \Phi(\alpha_1, \ldots, \alpha_r)=:d$$
 for irreducible  modules $M({\lambda})$ with $m(A, (FT)^*, \lambda) \neq 0$. Combined with the polynomial bound on the multiplicities, we get the upper bound $\PIexp^{T\text{-}\mathrm{gr}}(A) \leqslant d$. This step will be achieved in Theorem \ref{TheoremUpperFrac} and the formula for $d$ is proved in Lemma \ref{LemmaMaximumFUpperFrac}. 
 
 The idea to use the function $\Phi$ on some region $\Omega$ was introduced in \cite{GiaMischZaic, VerZaiMishch}. However, in order to find an appropriate region $\Omega$, i.e. a region which does not contain most partitions $\lambda$ with  $m(A, (FT)^*, \lambda) = 0$, new techniques are needed. For this we  develop a special method in Section~\ref{SectionUpperFrac}.

\textbf{(c)} Since $c_n^{T\text{-}\mathrm{gr}}(A) \geqslant \dim M(\lambda)$ for all irreducible modules appearing in the decomposition, it is sufficient to find a partition $\mu$ such that $m(A, (FT)^*, \mu) \neq 0$ and 
$$\dim (M_{\mu}) \geqslant \frac{n!}{n^{r(r-1)}\mu_1! \ldots \mu_r!} \geqslant C_1 n^{B_1} \left( \frac{1}{(\frac{\mu_1}{n})^{\frac{\mu_1}{n}}. \ldots .(\frac{\mu_r}{n})^{\frac{\mu_r}{n}}}\right)^n  \approx C_1 n^{B_1} d^n$$
for some constants $B_1, C_1 \in \mathbb{R}$ in order to get the needed lower bound. 

For this we construct for every $n \geqslant n_0$  a multilinear polynomial $f \in P_n^{T\text{-}\mathrm{gr}}$ such that $e_{T_{\mu}}^* f \notin \Id^{T\text{-}\mathrm{gr}}(A)$ for some tableau $T_{\mu}$ of a desired shape $\mu$. In order to get $e_{T_{\mu}}^* f \notin \Id^{T\text{-}\mathrm{gr}}(A)$,
we construct $f$ alternating in the variables of some specific sets. 

Note that no general method is known  to construct such non-vanishing polynomials. In Section~\ref{SectionFracM2Equal} and Section~\ref{SectionFracM2Less} we have to do this manually. 

The polynomial $f$ is constructed in Lemmas~\ref{LemmaGrPIexpTriangleAltFracM2} and~\ref{LemmaAltNonTriangleFracM2}. The proof heavily uses the description of left ideals in $M_2(F)$ and we restrict ourselves to zero bands. The exact exponent is computed in Theorem~\ref{TheoremGrPIexpTriangleFracM2} and Theorem~\ref{TheoremGrPIexpNonTriangleFracM2}. We get $\PIexp^{T\text{-}\mathrm{gr}}(A)= \dim A$ and $$\PIexp^{T\text{-}\mathrm{gr}}(A) = |T_0|+2|T_1|
  + 2\sqrt{(|T_1|+|\bar t_0|)(|T_0|+|T_1|-|\bar t_0|)}<\dim A,$$ respectively. The numbers $|T_0|, |T_1|, |\bar t_0|$ are defined in the beginning of Section~\ref{SectionFracM2Equal}. In particular, any number $m + 1 + \sqrt{m}$ for any $m \in \mathbb{N}$ can be realized as the $T$-graded PI-exponent of some $T$-graded-simple finite dimensional algebra.

\section{Upper bound for $(FT)^*$-codimensions of $T$-graded-simple algebras}
\label{SectionUpperFrac}

Let $A=\bigoplus_{t\in T} A^{(t)}$ be a finite dimensional $T$-graded algebra over a field $F$ of characteristic $0$ for some semigroup $T$ such that $A/J(A) \cong M_k(F)$
for some $k\in\mathbb N$ and $A^{(t)}\cap J(A) = 0$ for all $t\in T$.
(For example, $F$ is algebraically closed and $A$ is a finite dimensional $T$-graded-simple algebra for some right zero band $T$, see Lemmas \ref{LemmaRadicalSemigroupGradedSimple} and \ref{LemmaReesGrSimpleOtherProperties} .) In this section we prove an upper bound for $T$-graded codimensions of $A$.

For every $t\in T$ fix a basis $\mathcal B^{(t)}$ in $A^{(t)}$.
Then $\mathcal B = \bigcup_{t\in T} \mathcal B^{(t)}$ is a basis in $A$. Fix also
some isomorphism $\psi \colon A/J(A) \rightarrow M_k(F)$.
Denote by $\pi \colon A \rightarrow A/J(A)$ the natural epimorphism.
Define the function $\theta \colon \mathcal B \rightarrow \mathbb Z$ by $\theta(a)= \min\left\lbrace i-j 
\mid \alpha_{ij}\ne 0,\ 1\leqslant i,j \leqslant k\right\rbrace$ if $\psi\pi(a)=\sum_{1\leqslant i,j \leqslant k} \alpha_{ij} e_{ij}$, $\alpha_{ij} \in F$.

The observation below plays a central role in the section.

\begin{lemma}\label{LemmaThetaCondition}
Let $f\in P_n^{(FT)^*}$ for some $n\in\mathbb N$ and let $a_i \in \mathcal B$, $1\leqslant i \leqslant n$.
If $f(a_1, \ldots, a_n) \ne 0$, then $1-k \leqslant \sum_{i=1}^n \theta(a_i) \leqslant k-1$.
\end{lemma}
\begin{proof} Note that $f$ is a linear combination of monomials $x^{h_1}_{\sigma(1)} x^{h_2}_{\sigma(2)}
\ldots x^{h_n}_{\sigma(n)}$, $h_i\in (FT)^*$, $\sigma \in S_n$. Denote by $\Gamma \colon A=\bigoplus_{t\in T} A^{(t)}$ the $T$-grading on $A$.
Since $\supp \Gamma$ is finite, we may assume that $f$ is a linear combination of monomials $x^{h_{t_1}}_{\sigma(1)} x^{h_{t_2}}_{\sigma(2)}
\ldots x^{h_{t_n}}_{\sigma(n)}$, $t_i\in \supp \Gamma$, $\sigma \in S_n$. (See the definition of $h_t$ 
in Section~\ref{SectionFTActionGr}.)
Since all $a_i$ are homogeneous, the value of $x^{h_{t_1}}_{\sigma(1)} x^{h_{t_2}}_{\sigma(2)}
\ldots x^{h_{t_n}}_{\sigma(n)}$ equals $a_{\sigma(1)}\ldots a_{\sigma(n)}$
if $a_{\sigma(i)} \in A^{(t_i)}$ for all $1\leqslant i \leqslant n$ and $0$ otherwise.
However $a_{\sigma(1)}\ldots a_{\sigma(n)}$ is again a homogeneous element. Recall that $J(A)\cap A^{(t)}=0$
for every $t\in T$.
 Thus $a_{\sigma(1)}\ldots a_{\sigma(n)} \ne 0$ if and only if $$\psi\pi (a_{\sigma(1)}\ldots a_{\sigma(n)})  = \psi\pi (a_{\sigma(1)})\psi\pi (a_{\sigma(2)})\ldots \psi\pi (a_{\sigma(n)}) \ne 0.$$
 Now we notice that $e_{i_1 j_1} e_{i_2 j_2} \ldots e_{i_n j_n} \ne 0$ for some $1\leqslant i_\ell, j_\ell
 \leqslant k$ only if $j_1 = i_2$, $j_2 = i_3$, \ldots, $j_{n-1}=i_n$, and, in particular,
 $1-k \leqslant \sum_{\ell = 1}^n (i_\ell - j_\ell)=i_1 - j_k \leqslant k-1$.
 Therefore, $a_{\sigma(1)}\ldots a_{\sigma(n)} \ne 0$ only if
 $1-k \leqslant \sum_{\ell = 1}^n \theta(a_i) \leqslant k-1$.
 \end{proof}

Let $r:= \dim A$. Define $\beta_\ell := \min \left\lbrace \sum\limits_{i=1}^\ell \theta(a_i)
\mathrel{\biggl|} a_i \in \mathcal B,\ a_i\ne a_j \text{ for } i\ne j \right\rbrace$,
$\gamma_\ell := \beta_\ell - \beta_{\ell-1}$, $1\leqslant \ell \leqslant r$, $\beta_0 := 0$.
Without loss of generality, we may assume that $\mathcal B = (a_1, \ldots, a_r)$
where $$\theta(a_1) \leqslant \theta(a_2) \leqslant \ldots \leqslant \theta(a_r).$$
Then $\beta_\ell = \sum\limits_{i=1}^\ell \theta(a_i)$
and $\gamma_\ell = \theta(a_\ell)$.
In particular, 
\begin{equation*}1-k=\gamma_1 \leqslant \gamma_2 \leqslant \ldots \leqslant \gamma_r.
\end{equation*}

The equality $\gamma_1 = 1-k$ follows from the fact that $e_{1k}$ has the minimal value of $(i-j)$ among all
matrix units $e_{ij}$ and the matrix unit $e_{1k}$ must appear  with a nonzero coefficient in the decomposition of $\varphi\pi(a)$ for some $a\in\mathcal B$.

Now we prove the main inequality for $(FT)^*$-cocharacters of $A$.

\begin{lemma}\label{LemmaInequalityLambdaUpperFrac}
Let $f\in P_n^{(FT)^*}$ and $\lambda \vdash n$ for some $n\in\mathbb N$. If $\sum_{i=1}^r \gamma_i \lambda_i \geqslant k$ or $\lambda_{r+1} > 0$, then $e^*_{T_\lambda}f \in \Id^{(FT)^*}(A)$ for any Young tableau
$T_\lambda$ of shape $\lambda$.
\end{lemma}
\begin{proof}
Note that for each column of $T_\lambda$,  the polynomial $e^*_{T_\lambda}f = b_{T_\lambda}a_{T_\lambda}f$ is alternating in the variables with indices from that column.
Another remark is that, in order to determine whether a multilinear polynomial
is a polynomial $(FT)^*$-identity of $A$, it is sufficient to substitute only elements from $\mathcal B$.
If we substitute two coinciding elements for the variables of the same set of alternating variables,
we get zero. Thus, if $\lambda_{r+1} > 0$, then the height of the first column is greater than or equal to $(r+1)$ and at least two elements coincide. Therefore, $e^*_{T_\lambda}f \in \Id^{(FT)^*}(A)$.

Suppose $\sum_{i=1}^r \gamma_i \lambda_i \geqslant k$. We can rewrite this inequality in the form
\begin{equation}\label{EquationBetaLambdaUpperFrac}\sum_{i=1}^r (\beta_i-\beta_{i-1}) \lambda_i=\sum_{i=1}^r \beta_i (\lambda_i-\lambda_{i+1}) \geqslant k.\end{equation}
(We may assume that $\lambda_{r+1}=0$.) Note that $(\lambda_i-\lambda_{i+1})$ equals the number
of columns of height $i$ in $T_\lambda$.
Suppose $b_1, \ldots, b_n \in\mathcal B$ are substituted for $x_1, \ldots, x_n$.
By the remark above, we may assume that for the variables of each column different basis elements are substituted. By the definition of $\beta_i$, $\sum_{i=1}^n \theta(b_i) \geqslant \sum_{i=1}^r \beta_i (\lambda_i-\lambda_{i+1})$.
 Combining with~(\ref{EquationBetaLambdaUpperFrac}),
 we get $\sum_{i=1}^n \theta(b_i) \geqslant k$. Now Lemma~\ref{LemmaThetaCondition}
 implies $(e^*_{T_\lambda}f)(b_1, \ldots, b_n)=0$ and $e^*_{T_\lambda}f \in \Id^{(FT)^*}(A)$.
\end{proof}

Let $\Phi(\alpha_1,\ldots,\alpha_r)=\frac{1}{\alpha_1^{\alpha_1} \alpha_2^{\alpha_2} \ldots 
\alpha_r^{\alpha_r}}$. This function becomes continuous in the region $\alpha_1,\ldots,\alpha_r
\geqslant 0$ if we define $0^0 := 1$. 

The lemma below shows the importance of the function $\Phi$ in our investigation.

\begin{lemma}\label{LemmaUpperBoundFUpperFrac}
Let $$
\Omega:=\lbrace(\alpha_1, \alpha_2, \ldots, \alpha_r)\in \mathbb R^r \mid \alpha_1 \geqslant \ldots\geqslant \alpha_r
\geqslant 0,\ \alpha_1+\alpha_2+\ldots+\alpha_r=1,\ \gamma_1\alpha_1 + \ldots + \gamma_r\alpha_r \leqslant 0 \rbrace.$$
Then $\mathop{\overline\lim}\limits_{n\rightarrow\infty}
 \sqrt[n]{c_n^{T\text{-}\mathrm{gr}}(A)}\leqslant \max_{x\in\Omega} \Phi(x)$.
\end{lemma}
\begin{proof}
This is a consequence of Lemmas~\ref{LemmaCnGrCnGenH}, \ref{LemmaInequalityLambdaUpperFrac} and \cite[Lemma~2]{ASGordienko13} (or the original paper~\cite{VerZaiMishch}).
\end{proof}

The rest of the section is devoted to the calculation of the maximum of $\Phi$. We begin with the most simple region.

\begin{lemma}\label{LemmaMaximumFPlainUpperFrac}
Let $r\in\mathbb N$ and $$
\Omega_0:=\left\lbrace(\alpha_1, \alpha_2, \ldots, \alpha_r)\in \mathbb R^r \mathrel{\bigl|} \alpha_1, \ldots, \alpha_r
\geqslant 0,\ \alpha_1+\alpha_2+\ldots+\alpha_r=1 \right\rbrace.$$
Then $\max_{x\in\Omega_0} \Phi(x) = r$
and $\mathop{\mathrm{argmax}}_{x\in\Omega_0} \Phi(x) = \left(\frac{1}{r}, \frac{1}{r}, \ldots,  \frac{1}{r}\right)$.
\end{lemma}
\begin{proof}
We prove the lemma by induction on $r$. The case $r=1$ is trivial. Assume $r\geqslant 2$.
First, we can express $\alpha_1 = 1-\sum_{i=2}^r \alpha_i$ in terms of $\alpha_2, \alpha_3, \ldots, \alpha_r$, and study
$\Phi_1(\alpha_2, \ldots, \alpha_r) = \frac{1}{\left(1-\sum_{i=2}^r \alpha_i
\right)^{\left(1-\sum_{i=2}^r \alpha_i\right)} \alpha_2^{\alpha_2} \ldots 
\alpha_r^{\alpha_r}}$ on
$$
\tilde \Omega_0:=\left\lbrace(\alpha_2, \ldots, \alpha_r)\in \mathbb R^{r-1} \mathrel{\biggl|} \alpha_2, \ldots, \alpha_r
\geqslant 0,\ 1-\sum_{i=2}^r \alpha_i\geqslant 0 \right\rbrace.$$
Note that $\Phi_1$ is continuous on the compact set $\tilde \Omega_0$
and differentiable at all inner points of $\tilde \Omega_0$.
Thus $\Phi_1$ can reach its extremal values only at inner critical points of $\Phi_1$
or on $\partial\tilde\Omega_0$. By the induction assumption, $\Phi_1(x) \leqslant r-1$ for all $x\in \partial \tilde\Omega_0$.
Consider $$\frac{\partial \Phi_1}{\partial \alpha_\ell}(\alpha_2, \ldots, \alpha_r)=
\left(\ln\left(1-\sum_{i=2}^r \alpha_i\right)-\ln \alpha_\ell\right)\Phi_1(\alpha_2, \ldots, \alpha_r).$$
Then $\frac{\partial \Phi}{\partial \alpha_\ell}(\alpha_2, \ldots, \alpha_r) = 0$ for all $2\leqslant \ell \leqslant r$ only for $\alpha_2=\ldots = \alpha_r = 1-\sum_{i=2}^r \alpha_i = \frac{1}{r}$.
Since $\Phi\left(\frac{1}{r}, \frac{1}{r}, \ldots,  \frac{1}{r}\right)=r > r-1$, we get the lemma.
\end{proof}

The positive root $\zeta$ of the polynomial $P$, defined in the lemma below, will be used
in the calculation of the upper bound of codimensions. Here we study the basic properties of $P$.

\begin{lemma}\label{LemmaEquationZetaUpperFrac}
Let $r\in\mathbb N$ and $\gamma_i\in\mathbb Z$, $1\leqslant i \leqslant r$.
Suppose 
\begin{equation}\label{EquationGammaUpperFrac}
\gamma_1 \leqslant \gamma_2 \leqslant \ldots \leqslant \gamma_r,
\end{equation} $\gamma_1 < 0$. Consider the equation 
\begin{equation}\label{EquationZetaUpperFrac}P(\zeta):=\sum_{i=1}^r \gamma_i \zeta^{\gamma_i-\gamma_1} = 0\end{equation}
where $\zeta$ is the unknown variable.
If $\sum_{i=1}^r
\gamma_i \geqslant 0$,
then~(\ref{EquationZetaUpperFrac}) has the only root $\zeta > 0$.
Moreover $\zeta \in (0; 1]$.
If $\sum_{i=1}^r
\gamma_i < 0$, then $P(y) < 0$ for all $y\in[0;1]$.
\end{lemma}
\begin{proof} 
If $\gamma_r \leqslant 0$, then all nonzero coefficients of $P$ are negative and $P(y) < 0$ for all $y\in[0;1]$. 

Suppose $\gamma_r > 0$. Inequality~(\ref{EquationGammaUpperFrac}) implies that there is only one sign difference in the signs of coefficients of~$P$. Therefore,
by Descartes' rule of signs, (\ref{EquationZetaUpperFrac}) has the unique positive root~$\zeta$.
Define $m\in\mathbb N$ by $$\gamma_1=\ldots=\gamma_m<\gamma_{m+1}.$$
Note that $P(0) = m \gamma_1 < 0$ and $P(1)=\sum_{i=1}^r \gamma_i$.
Therefore, if $\sum_{i=1}^r \gamma_i \geqslant 0$, we have $\zeta \in (0; 1]$.  
If $P(1)=\sum_{i=1}^r \gamma_i < 0$, then $\zeta > 1$ and $P(y) < 0$ for all $y\in[0;1]$.
\end{proof}

It turns out that the root $\zeta$ of $P$ is the extremal point of the function $\Psi$
defined below. This will be used in the calculation of the maximum of $\Phi$ on our region $\Omega$. 

\begin{lemma}\label{LemmaMinimumGammaSumUpperFrac}
Let $r\in\mathbb N$ and $\gamma_i\in\mathbb Z$, $1\leqslant i \leqslant r$.
Suppose $\gamma_1 \leqslant \gamma_2 \leqslant \ldots \leqslant \gamma_r$, $\gamma_1 < 0$.
Denote $\Psi(y)=\sum_{i=1}^r y^{\gamma_i}$.
Then \begin{enumerate}
\item
if $\sum_{i=1}^r
\gamma_i \geqslant 0$, then $\min_{y\in(0;1]} \Psi(y) =\Psi(\zeta)$
where $\zeta \in (0; 1]$ is the positive root
 of~(\ref{EquationZetaUpperFrac});
\item if $\sum_{i=1}^r
\gamma_i \leqslant 0$, then $\min_{y\in(0;1]} \Psi(y) = r$.
\end{enumerate}
\end{lemma}
\begin{proof}
Note that $\Psi'(y)= \sum_{i=1}^r \gamma_i y^{\gamma_i-1}$.
Thus $\Psi'(y)$ has the same sign on $(0;1]$  as $P(y)=\sum_{i=1}^r \gamma_i y^{\gamma_i-\gamma_1}$.  Also $\lim_{y\rightarrow 0^+} \Psi(y)= +\infty$. Lemma~\ref{LemmaEquationZetaUpperFrac} implies that if $\sum_{i=1}^r
\gamma_i \geqslant 0$, then $\min_{y\in(0;1]} \Psi(y) =\Psi(\zeta)$, and if $\sum_{i=1}^r
\gamma_i \leqslant 0$, then $\min_{y\in(0;1]} \Psi(y) = \Psi(1)=r$. (In the case $\sum_{i=1}^r
\gamma_i = 0$, we have $\zeta = 1$ and $\Psi(\zeta)=r$.)
\end{proof}

Now we are ready to calculate the maximum of $\Phi$. For our convenience, we replace our region $\Omega$
with a larger region $\tilde\Omega$ and show that the maximum on both regions is the same.

\begin{lemma}\label{LemmaMaximumFUpperFrac}
Let $r\in\mathbb N$ and $\gamma_i\in\mathbb Z$, $1\leqslant i \leqslant r$.
Suppose $\gamma_1 \leqslant \gamma_2 \leqslant \ldots \leqslant \gamma_r$,
$\gamma_1 < 0$, $\sum_{i=1}^r
\gamma_i \geqslant 0$.
Let $$
\Omega:=\lbrace(\alpha_1, \alpha_2, \ldots, \alpha_r)\in \mathbb R^r \mid \alpha_1 \geqslant \ldots\geqslant \alpha_r
\geqslant 0,\ \alpha_1+\alpha_2+\ldots+\alpha_r=1,\ \gamma_1\alpha_1 + \ldots + \gamma_r\alpha_r \leqslant 0 \rbrace$$
and let $$
\tilde\Omega:=\lbrace(\alpha_1, \alpha_2, \ldots, \alpha_r)\in \mathbb R^r \mid \alpha_1,\ldots,\alpha_r \geqslant 0,\ \alpha_1+\alpha_2+\ldots+\alpha_r=1,\ \gamma_1\alpha_1 + \ldots + \gamma_r\alpha_r \leqslant 0 \rbrace.$$
Then $\max_{x\in\Omega} \Phi(x) = \max_{x\in\tilde\Omega} \Phi(x) =\sum_{i=1}^r \zeta^{\gamma_i}$
where $\zeta \in (0; 1]$ is the positive root
 of~(\ref{EquationZetaUpperFrac}).
\end{lemma}
\begin{proof}
Like in Lemma~\ref{LemmaMaximumFPlainUpperFrac}, we use induction on $r$.
The conditions $\gamma_1 < 0$ and $\sum_{i=1}^r
\gamma_i \geqslant 0$ imply that $r \geqslant 2$.
We will not prove the induction base $r=2$ separately, but the base will follow from the arguments
below since for $r=2$ we will not use the induction assumption.
 
Again, we express $\alpha_1 = 1-\sum_{i=2}^r \alpha_i$ in terms of $\alpha_2, \alpha_3, \ldots, \alpha_r$, and study
$\Phi_1(\alpha_2, \ldots, \alpha_r) = \frac{1}{\left(1-\sum_{i=2}^r \alpha_i
\right)^{\left(1-\sum_{i=2}^r \alpha_i\right)} \alpha_2^{\alpha_2} \ldots 
\alpha_r^{\alpha_r}}$ on
$$
\Omega_1:=\left\lbrace(\alpha_2, \ldots, \alpha_r)\in \mathbb R^{r-1} \mathrel{\biggl|} \alpha_2, \ldots, \alpha_r
\geqslant 0,\ 1-\sum_{i=2}^r \alpha_i\geqslant 0,\ \gamma_1+\sum_{i=2}^r(\gamma_i-\gamma_1)\alpha_i \leqslant 0 \right\rbrace.$$
Now the proof of Lemma~\ref{LemmaMaximumFPlainUpperFrac} implies that the only critical point of $\Phi_1$
is $\left(\frac{1}{r},  \ldots,  \frac{1}{r}\right)$. This point belongs to $\Omega_1$ if and only if
$\sum_{i=1}^r \gamma_i \leqslant 0$. If indeed $\sum_{i=1}^r
\gamma_i = 0$, then by Lemma~\ref{LemmaMaximumFPlainUpperFrac} we have $\max_{x\in\Omega} \Phi(x)=r$.
Since in this case $\zeta = 1$, the lemma is proved.

Suppose $\sum_{i=1}^r \gamma_i > 0$. Then the continuous function $\Phi_1$ reaches its maximum on $\partial \Omega_1$. Note that $\partial \Omega_1 = \Omega_2 \cup \bigcup_{i=1}^r \Omega_{1i}$
where 
$$\Omega_{11}=\left\lbrace(\alpha_2, \ldots, \alpha_r)\in \mathbb R^{r-1} \mathrel{\biggl|} \alpha_2, \ldots, \alpha_r
\geqslant 0,\ 1-\sum_{i=2}^r \alpha_i= 0,\ \gamma_1+\sum_{i=2}^r(\gamma_i-\gamma_1)\alpha_i \leqslant 0 \right\rbrace,$$
\begin{equation*}\begin{split}\Omega_{1\ell}=\left\lbrace(\alpha_2, \ldots, \alpha_r)\in \mathbb R^{r-1} \mathrel{\biggl|} \alpha_2, \ldots, \alpha_r 
\geqslant 0,\ \alpha_\ell=0,\right. \\ \left. 1-\sum_{i=2}^r \alpha_i\geqslant 0,\ \gamma_1+\sum_{i=2}^r(\gamma_i-\gamma_1)\alpha_i \leqslant 0 \right\rbrace
\text{ for } 2\leqslant \ell \leqslant r,\end{split}\end{equation*}
$$\Omega_2=\left\lbrace(\alpha_2, \ldots, \alpha_r)\in \mathbb R^{r-1} \mathrel{\biggl|} \alpha_2, \ldots, \alpha_r
\geqslant 0,\ 1-\sum_{i=2}^r \alpha_i\geqslant 0,\ \gamma_1+\sum_{i=2}^r(\gamma_i-\gamma_1)\alpha_i = 0 \right\rbrace.$$

We claim that \begin{equation}\label{EquationOmega1iUpperFrac}\max\limits_{x\in \bigcup_{i=1}^r \Omega_{1i}} \Phi_1(x) < \sum_{j=1}^r \zeta^{\gamma_j}\end{equation} where $\zeta \in (0; 1]$ is the positive root
 of~(\ref{EquationZetaUpperFrac}).
Indeed, switching to the variables $\alpha_1, \ldots, \hat \alpha_i, \ldots, \alpha_r$, we get $\max_{x\in \Omega_{1i}} \Phi_1(x) = \max_{x\in \Omega'_i} \Phi(x)$, $1\leqslant i \leqslant r$,
where \begin{equation*}\begin{split}\Omega_i'=\left\lbrace(\alpha_1, \ldots, \hat \alpha_i, \ldots, \alpha_r)\in \mathbb R^{r-1} \mathrel{\biggl|} \alpha_1, \ldots, \hat \alpha_i,\ldots, \alpha_r 
\geqslant 0, \right. \\ \left.  \alpha_1+\ldots+\hat\alpha_i+\ldots+\alpha_r=1,\ \gamma_1\alpha_1 + \ldots+ \widehat{\gamma_i\alpha_i}+\ldots + \gamma_r\alpha_r \leqslant 0 \right\rbrace.\end{split}\end{equation*}
(For the convenience, we denote the function $\Phi(\theta_1, \ldots, \theta_m)=\frac{1}{\theta_1^{\theta_1}
\ldots \theta_m^{\theta_m}}$ by the same letter $\Phi$ for all $i$.)

If $i=1$ and $\gamma_2 \geqslant 0$, then 
 \begin{equation*}\begin{split}\Omega_i'=\Omega_1'=\left\lbrace(\alpha_2, \ldots, \alpha_r)\in \mathbb R^{r-1} \mathrel{\biggl|} \alpha_2, \ldots,  \alpha_r 
\geqslant 0,\right. \\ \left. \alpha_2+\ldots+\ldots+\alpha_r=1, \alpha_\ell = \alpha_{\ell+1}=\ldots=\alpha_r=0 \right\rbrace\end{split}\end{equation*}
where the number $2\leqslant \ell\leqslant r$ is defined by the equality $\gamma_2=\ldots=\gamma_{\ell-1}=0$
and the inequality $\gamma_\ell > 0$. Then Lemma~\ref{LemmaMaximumFPlainUpperFrac}
implies $\max_{x\in \Omega'_1} \Phi(x)= \ell-2$. Since $\sum_{j=1}^r \zeta^{\gamma_j} > \sum_{j=2}^{\ell-1} \zeta^{\gamma_j} = \ell-2$, we get $\max_{x\in \Omega_{11}} \Phi_1(x) < \sum_{j=1}^r \zeta^{\gamma_j}$.

Suppose that either $i\geqslant 2$ or $\gamma_2 < 0$, and
$\sum\limits_{\substack{\ell=1,\\ \ell\ne i}}^r
\gamma_\ell \geqslant 0$. Then we apply the induction assumption 
for $(r-1)$. We have $\max_{x\in \Omega'_i} \Phi(x) = \sum\limits_{\substack{\ell=1,\\ \ell\ne i}}^r \left(\zeta'\right)^{\gamma_\ell}$ where $\sum\limits_{\substack{\ell=1,\\ \ell\ne i}}^r \gamma_\ell \left(\zeta'\right)^{\gamma_\ell-\gamma_1}=0$ if $i > 1$
and $\sum_{\substack{\ell=2}}^r \gamma_\ell \left(\zeta'\right)^{\gamma_\ell-\gamma_2}=0$
if $i=1$. By Lemma~\ref{LemmaMinimumGammaSumUpperFrac},
\begin{equation}\label{EquationOmegaMaxPsiUpperFrac} \max_{x\in \Omega_{1i}} \Phi_1(x)=\max_{x\in \Omega'_i} \Phi(x) = \min_{y\in(0;1]} \sum_{\substack{\ell=1,\\ \ell\ne i}}^r y^{\gamma_\ell} < \min_{y\in(0;1]} \sum_{\ell=1}^r y^{\gamma_\ell}=\sum_{j=1}^r \zeta^{\gamma_j}.\end{equation}

Suppose that either $i\geqslant 2$ or $\gamma_2 < 0$, and
$\sum\limits_{\substack{\ell=1,\\ \ell\ne i}}^r
\gamma_\ell < 0$. Then $\left(\frac{1}{r-1}, \frac{1}{r-1}, \ldots,  \frac{1}{r-1}\right)\in \Omega'_i $ and by Lemma~\ref{LemmaMaximumFPlainUpperFrac} we have $\max_{x\in \Omega'_i} \Phi(x) = r-1$.
Again, by Lemma~\ref{LemmaMinimumGammaSumUpperFrac}, we get~(\ref{EquationOmegaMaxPsiUpperFrac}).
Therefore~(\ref{EquationOmega1iUpperFrac}) is proved.

We claim that $\max_{x\in \Omega_2} \Phi_1(x) = \sum_{i=1}^r \zeta^{\gamma_i}$
where $\zeta \in (0; 1]$ is the positive root
 of~(\ref{EquationZetaUpperFrac}).
 
   If $r=2$, then $\Omega_2=\left\lbrace -\frac{\gamma_1}{\gamma_2-\gamma_1} \right\rbrace$,
   $\zeta=\left(-\frac{\gamma_1}{\gamma_2}\right)^{\frac{1}{\gamma_2-\gamma_1}}$,
  \begin{eqnarray*}
  \lefteqn{\Phi_1\left(-\frac{\gamma_1}{\gamma_2-\gamma_1}\right)
   =\left(\frac{\gamma_2}{\gamma_2-\gamma_1}\right)^{-\frac{\gamma_2}{\gamma_2-\gamma_1}}\left(-\frac{\gamma_1}{\gamma_2-\gamma_1}\right)^{\frac{\gamma_1}{\gamma_2-\gamma_1}}}\\
 &&  =(\gamma_2-\gamma_1)
   \gamma_2^{-\frac{\gamma_2}{\gamma_2-\gamma_1}}(-\gamma_1)^{\frac{\gamma_1}{\gamma_2-\gamma_1}}
    = \left(-\frac{\gamma_1}{\gamma_2}\right)^{\frac{\gamma_1}{\gamma_2-\gamma_1}}
   +\left(-\frac{\gamma_1}{\gamma_2}\right)^{\frac{\gamma_2}{\gamma_2-\gamma_1}}
   =\zeta^{\gamma_1}+\zeta^{\gamma_2}
  \end{eqnarray*}

   Therefore we may assume $r\geqslant 3$.
   Define $1\leqslant m < r$ by $\gamma_1=\ldots=\gamma_m<\gamma_{m+1}.$
   Then for all $(\alpha_2, \ldots, \alpha_r)\in \Omega_2$
   we have $\gamma_1+\sum_{i=m+1}^r(\gamma_i-\gamma_1)\alpha_i = 0$
   and \begin{equation}\label{EquationAlpha(m+1)UpperFrac}\alpha_{m+1}= -\frac{1}{\gamma_{m+1}-\gamma_1}\left(\gamma_1+\sum_{i=m+2}^r(\gamma_i-\gamma_1)\alpha_i
   \right).\end{equation}
  We express $\alpha_{m+1}$ and notice that
  $\max_{x\in \Omega_2} \Phi_1(x) = \max_{x\in \Omega_3} \Phi_2(x)$
  where 
$$\Phi_2(\alpha_2, \ldots, \alpha_m, \alpha_{m+2},\ldots, \alpha_r) = 
\Phi(\alpha_1, \ldots, \alpha_r),$$
$$
\Omega_3:=\left\lbrace(\alpha_2, \ldots, \alpha_m, \alpha_{m+2},\ldots, \alpha_r)\in \mathbb R^{r-2} \mathrel{\biggl|} \alpha_1, \ldots, \alpha_r
\geqslant 0 \right\rbrace,$$
  $$\alpha_1 = 1-\sum\limits_{\substack{i=2, \\ i\ne m+1}}^r \alpha_i + \frac{1}{\gamma_{m+1}-\gamma_1}\left(\gamma_1+\sum_{i=m+2}^r(\gamma_i-\gamma_1)\alpha_i
\right)$$  
  and $\alpha_{m+1}$ is defined by~(\ref{EquationAlpha(m+1)UpperFrac}).

Consider
\begin{equation*}\begin{split} \frac{\partial \Phi_2}{\partial \alpha_i}(\alpha_2, \ldots, \alpha_m, \alpha_{m+2},\ldots, \alpha_r)=
\left( (-\ln \alpha_1-1)\frac{\partial \alpha_1}{\partial \alpha_i} -\ln \alpha_i-1+ \right.\\
\left.
(-\ln \alpha_{m+1}-1)\frac{\partial \alpha_{m+1}}{\partial \alpha_i}
 \right) \Phi_2(\alpha_2, \ldots, \alpha_m, \alpha_{m+2},\ldots, \alpha_r).
\end{split}\end{equation*}

Let $2\leqslant i \leqslant m$. Then $\frac{\partial \alpha_1}{\partial \alpha_i} = -1$,
$\frac{\partial \alpha_{m+1}}{\partial \alpha_i} = 0$, and
$$\frac{\partial \Phi_2}{\partial \alpha_i}(\alpha_2, \ldots, \alpha_m, \alpha_{m+2},\ldots, \alpha_r)=
(\ln \alpha_1 -\ln \alpha_i) \Phi_2(\alpha_2, \ldots, \alpha_m, \alpha_{m+2},\ldots, \alpha_r).$$

Let $m+2\leqslant i \leqslant r$. Then $\frac{\partial \alpha_1}{\partial \alpha_i} = 
\frac{\gamma_i-\gamma_{m+1}}{\gamma_{m+1}-\gamma_1}$,
$\frac{\partial \alpha_{m+1}}{\partial \alpha_i} = \frac{\gamma_1-\gamma_i}{\gamma_{m+1}-\gamma_1}$, and

\begin{eqnarray*}
\lefteqn{\frac{\partial \Phi_2}{\partial \alpha_i}(\alpha_2, \ldots, \alpha_m, \alpha_{m+2},\ldots, \alpha_r)}\\
&&=\left(-\frac{\gamma_i-\gamma_{m+1}}{\gamma_{m+1}-\gamma_1}\ln \alpha_1 -\ln \alpha_i-
\frac{\gamma_1-\gamma_i}{\gamma_{m+1}-\gamma_1}\ln \alpha_{m+1}\right) \Phi_2(\alpha_2, \ldots, \alpha_m, \alpha_{m+2},\ldots, \alpha_r)
\end{eqnarray*}


Therefore, if $(\alpha_2, \ldots, \alpha_m, \alpha_{m+2},\ldots, \alpha_r)\in \Omega_3$ is a critical point
for $\Phi_2$, we have $$\frac{\partial \Phi_2}{\partial \alpha_i}(\alpha_2, \ldots, \alpha_m, \alpha_{m+2},\ldots, \alpha_r)=0$$ for all $2\leqslant i \leqslant m$ and $m+2\leqslant i \leqslant r$
which is equivalent to
$$\left\lbrace \begin{array}{lllll} \alpha_i & = & \alpha_1 & \text{ for } & 2\leqslant i \leqslant m,  \\
\alpha_i & = & \alpha_1^{\left(\frac{\gamma_{m+1}-\gamma_i}{\gamma_{m+1}-\gamma_1}\right)}
\alpha_{m+1}^{\left(\frac{\gamma_i-\gamma_1}{\gamma_{m+1}-\gamma_1}\right)}
& \text{ for } & m+2\leqslant i \leqslant r, \\
\alpha_1 & = & 1-\sum\limits_{\substack{i=2, \\ i\ne m+1}}^r \alpha_i + \frac{1}{\gamma_{m+1}-\gamma_1}\left(\gamma_1+\sum_{i=m+2}^r(\gamma_i-\gamma_1)\alpha_i
\right), \\
\alpha_{m+1} & = & -\frac{1}{\gamma_{m+1}-\gamma_1}\left(\gamma_1+\sum_{i=m+2}^r(\gamma_i-\gamma_1)\alpha_i
   \right)
 \end{array}\right.$$
  and
$$\left\lbrace \begin{array}{lllll} \alpha_i & = & \alpha_1 & \text{ for } & 2\leqslant i \leqslant m,  \\
\alpha_i & = & \alpha_1
\left(\frac{\alpha_{m+1}}{\alpha_1}\right)^{\frac{\gamma_i-\gamma_1}{\gamma_{m+1}-\gamma_1}}
& \text{ for } & m+2\leqslant i \leqslant r, \\
\alpha_1 & = & 1-\sum\limits_{\substack{i=2, \\ i\ne m+1}}^r \alpha_i + \frac{1}{\gamma_{m+1}-\gamma_1}\left(\gamma_1+\sum_{i=m+2}^r(\gamma_i-\gamma_1)\alpha_i
\right), \\
\alpha_{m+1} & = & -\frac{1}{\gamma_{m+1}-\gamma_1}\left(\gamma_1+\sum_{i=m+2}^r(\gamma_i-\gamma_1)\alpha_i
   \right).
 \end{array}\right.$$

Note that since we are looking for inner critical points of $\Phi_2$ on $\Omega_3 \subset \mathbb R^{r-2}$, we may assume that all $\alpha_i > 0$.

Performing equivalent transformations, we get

\begin{equation}\label{EquationAlphaRelationUpperFrac1}\left\lbrace \begin{array}{lllll} \alpha_i & = & \alpha_1
\left(\frac{\alpha_{m+1}}{\alpha_1}\right)^{\frac{\gamma_i-\gamma_1}{\gamma_{m+1}-\gamma_1}}
& \text{ for } 1\leqslant i \leqslant r, \\
\sum\limits_{i=1}^r \alpha_i & = & 1, \\
\sum_{i=1}^r \gamma_i \alpha_i & = & 0.
 \end{array}\right.\end{equation}

Now we introduce an additional variable $\zeta := \left(\frac{\alpha_{m+1}}{\alpha_1}\right)^{\frac{1}{\gamma_{m+1}-\gamma_1}}$ and get

\begin{equation}\label{EquationAlphaRelationUpperFrac2}\left\lbrace \begin{array}{lllll}
\zeta &= & \left(\frac{\alpha_{m+1}}{\alpha_1}\right)^{\frac{1}{\gamma_{m+1}-\gamma_1}}, \\
 \alpha_i & = & \alpha_1
\zeta^{\gamma_i-\gamma_1}
& \text{ for } 1\leqslant i \leqslant r, \\
\alpha_1 \sum\limits_{i=1}^r \zeta^{\gamma_i-\gamma_1} & = & 1, \\
 \sum_{i=1}^r \gamma_i \zeta^{\gamma_i-\gamma_1} & = & 0.
 \end{array}\right.\end{equation}
Now the first equation is the consequence of the second one for $i=m+1$.
Thus the original system is equivalent to
\begin{equation}\label{EquationAlphaSolutionUpperFrac}\left\lbrace \begin{array}{lllll}
 \alpha_i & = & \frac{\zeta^{\gamma_i-\gamma_1}}{\sum_{i=1}^r \zeta^{\gamma_i-\gamma_1}}
& \text{ for } 1\leqslant i \leqslant r, \\
 \sum_{i=1}^r \gamma_i \zeta^{\gamma_i-\gamma_1} & = & 0.
 \end{array}\right. \end{equation}
 By Lemma~\ref{LemmaEquationZetaUpperFrac},
 the last equation has the unique solution $\zeta \in (0;1]$.
 Thus $$(\alpha_2,\ldots,\alpha_m, \alpha_{m+2}, \ldots, \alpha_r)$$
 defined by~(\ref{EquationAlphaSolutionUpperFrac}) is the unique inner critical point
 of $\Phi_2$. Using~(\ref{EquationAlphaRelationUpperFrac1}) and~(\ref{EquationAlphaRelationUpperFrac2}), we get
  \begin{equation*}\begin{split}\Phi_2(\alpha_2,\ldots,\alpha_m, \alpha_{m+2}, \ldots, \alpha_r)
  = \Phi(\alpha_1,\ldots, \alpha_r)=\frac{1}{\alpha_1^{\alpha_1} \alpha_2^{\alpha_2} \ldots \alpha_r^{\alpha_r}}
   \\ = \frac{1}{\alpha_1^{\alpha_1+\ldots+\alpha_r} \zeta^{\alpha_1(\gamma_1-\gamma_1)} \ldots \zeta^{\alpha_r(\gamma_r-\gamma_1)}}=
  \frac{1}{\alpha_1 \zeta^{-\gamma_1}}=\sum_{i=1}^r \zeta^{\gamma_i}.
  \end{split}\end{equation*}
 
 Note that the values of $\Phi_2$ on $\partial \Omega_3$ equal the values of $\Phi_1$
 at the corresponding points of $\bigcup_{i=1}^r \Omega_{1i}$.
 Therefore, $$\max_{x\in \tilde\Omega} \Phi(x)=\max_{x\in \Omega_1} \Phi_1(x)=\max_{x\in \Omega_3} \Phi_2(x) = \sum_{i=1}^r \zeta^{\gamma_i}.$$
 Since~(\ref{EquationAlphaSolutionUpperFrac}) implies $\alpha_1 \geqslant \alpha_2\geqslant \ldots
 \geqslant \alpha_r$, we get $$\max_{x\in \Omega} \Phi(x)=\max_{x\in \tilde\Omega} \Phi(x) = 
 \sum_{i=1}^r \zeta^{\gamma_i}.$$
 \end{proof}

Now we immediately get the desired upper bound.

\begin{theorem}\label{TheoremUpperFrac}
Let $A=\bigoplus_{t\in T} A^{(t)}$ be a finite dimensional $T$-graded algebra over a field $F$ of characteristic $0$ for some semigroup $T$ such that $A/J(A) \cong M_k(F)$
for some $k\in\mathbb N$ and $A^{(t)}\cap J(A) = 0$ for all $t\in T$. Let
 $\gamma_\ell$, $1\leqslant \ell \leqslant r$, $r:=\dim A$, be the numbers defined 
 before Lemma~\ref{LemmaInequalityLambdaUpperFrac}.
  Suppose $\sum_{i=1}^r
\gamma_i \geqslant 0$. Let $\zeta$ be the positive root
 of~(\ref{EquationZetaUpperFrac}). Then $\mathop{\overline\lim}\limits_{n\rightarrow\infty}
 \sqrt[n]{c_n^{T\text{-}\mathrm{gr}}(A)}\leqslant \sum_{i=1}^r \zeta^{\gamma_i}$.
\end{theorem}
\begin{proof}
This is a consequence of Lemmas~\ref{LemmaUpperBoundFUpperFrac} and~\ref{LemmaMaximumFUpperFrac}.
\end{proof}

\section{The case $A/J(A)\cong M_2(F)$ and $\PIexp^{T\text{-}\mathrm{gr}}(A)=\dim A$}\label{SectionFracM2Equal}

Let $A=\bigoplus_{t\in T} A^{(t)}$ be a finite dimensional $T$-graded-simple algebra
over a field $F$ of characteristic $0$
for some right zero band $T$.
In the next two sections we calculate $\PIexp^{T\text{-}\mathrm{gr}}(A)$ in the case $A/J(A) \cong M_2(F)$.
 
 Let $I_t := \pi(A^{(t)})$ for $t\in T$
 where $\pi \colon A \rightarrow A/J(A)$ is the natural epimorphism.

 Note that since $\dim A < +\infty$, only a finite number of $I_t$ are nonzero.
  Let $T_0 := \lbrace t\in T \mid \dim I_t=2 \rbrace$ and $T_1 = \lbrace t\in T \mid I_t = A/J(A) \rbrace$.
We have $I_t = 0$ for all $t\notin T_0 \sqcup T_1$. Moreover $A^{(t)} \cap \ker \pi =0$
for all $t\in T$ implies $r:= \dim A = 2|T_0|+4|T_1|$.

  Define $t_1 \sim t_2$ if $I_{t_1}=I_{t_2}$. Since $A/J(A) \cong M_2(F)$, all irreducible $A/J(A)$-modules
  are two-dimensional and isomorphic to each other. Thus $A/J(A)=I_{t_1}\oplus I_{t_2}$
  for all $I_{t_1}\ne I_{t_2}$.

Now we show that if the cardinalities of equivalences classes satisfy some kind of triangle inequality,
then we can combine the elements into pairs and, possibly, a triple such that the elements inside each pair or triple are non-equivalent.
   
   \begin{lemma}\label{LemmaTrianglePairsAltFracM2}
Let $T_0$ be a finite non-empty set with an equivalence relation $\sim$. 
    Suppose \begin{equation}\label{EquationEquivTriangleFracM2}|\bar t_0| \leqslant \sum\limits_{\substack{\bar t \in {T_0}/{\sim}, \\
   \bar t \ne \bar t_0}} |\bar t| \text{ for all }\bar t_0 \in {T_0}/{\sim}
   .\end{equation}
   Then we can choose $\lbrace t_1, \ldots, t_{|T_0|} \rbrace = T_0$ such that \begin{enumerate}
\item if $2 \mid |T_0|$, we have $t_{2i-1} \not\sim t_{2i}$
for all $1\leqslant i \leqslant \frac{|T_0|}{2}$;
\item if $2 \nmid |T_0|$, we have $t_{2i-1} \not\sim t_{2i}$
for all $1\leqslant i \leqslant \frac{|T_0|-1}{2}$ and $t_{|T_0|-2}$, $t_{|T_0|-1}$, $t_{|T_0|}$ are pairwise non-equivalent.
\end{enumerate}
   \end{lemma}
  \begin{remark}
  Note that if~(\ref{EquationEquivTriangleFracM2}) does not hold, then there exists an equivalence class
  $\bar t_0 \in {T_0}/{\sim}$ such that $|\bar t_0| \geqslant \frac{|T_0|}{2}$.
  \end{remark}
  \begin{proof}[Proof of Lemma~\ref{LemmaTrianglePairsAltFracM2}.]
  We prove by induction on $|T_0|$. Note that~(\ref{EquationEquivTriangleFracM2}) implies
  $|{T_0}/{\sim}| \geqslant 2$.
  
  Suppose $|{T_0}/{\sim}| = 2$. Then~(\ref{EquationEquivTriangleFracM2})
  implies $|\bar t_1| = |\bar t_2|$ where  ${T_0}/{\sim} = \lbrace \bar t_1, \bar t_2 \rbrace$.
  We can define $\lbrace t_1, \ldots, t_{|T_0|}\rbrace = T_0$
  by $\lbrace t_1,  t_3, \ldots, t_{|T_0|-1}\rbrace := \bar t_1$
  and $\lbrace t_2,  t_4, \ldots, t_{|T_0|}\rbrace := \bar t_2$, and the lemma is proved.
  
  Suppose $|T_0| = 3$. Then all the elements of $T_0$ are pairwise non-equivalent
  and we again get the lemma.
  
  Now we assume that $|T_0| > 3$. Choose the classes $\bar t_1$ and $\bar t_2$
  with the maximal number of elements. Choose some $t_1 \in \bar t_1$
  and $t_2 \in \bar t_2$. Note that for the set $T_0\backslash \lbrace t_1, t_2 \rbrace$
  and the same equivalence relation $\sim$, we still have~(\ref{EquationEquivTriangleFracM2}).
  By the induction assumption we can choose $\lbrace t_3, \ldots, t_{|T|} \rbrace = T_0\backslash \lbrace t_1, t_2 \rbrace$
  such that $t_1, \ldots, t_{|T_0|}$ satisfy the conditions of the lemma.
  \end{proof}
  
Inequality~(\ref{EquationEquivTriangleFracM2}) will be used to distinguish between the cases when $\PIexp^{T\text{-}\mathrm{gr}}(A) = \dim A$ and when $\PIexp^{T\text{-}\mathrm{gr}}(A) < \dim A$.

First, we need the following technical lemma.

\begin{lemma}\label{LemmaAlphaBetaDeterminant}
Let $\alpha, \beta, \tilde \alpha, \tilde \beta \in F$.
Then
\begin{enumerate}
\item $$[\alpha e_{11} + \beta e_{12}, \alpha e_{21} + \beta e_{22}]
= \left(\begin{array}{rr}  \alpha\beta & \beta^2 \\  -\alpha^2 & -\alpha\beta \end{array}\right)
= \left(\begin{array}{rr}  \beta & 0 \\ -\alpha  & 0 \end{array}\right)
\left(\begin{array}{rr}  \alpha & \beta \\  0 & 0 \end{array}\right);$$
\item 

\begin{eqnarray*}
\lefteqn{[\alpha e_{11} + \beta e_{12}, \alpha e_{21} + \beta e_{22}][\tilde\alpha e_{11} + \tilde\beta e_{12}, \tilde\alpha e_{21} + \tilde\beta e_{22}]}\\
&& +[\tilde\alpha e_{11} + \tilde\beta e_{12}, \tilde\alpha e_{21} + \tilde\beta e_{22}][\alpha e_{11} + \beta e_{12}, \alpha e_{21} + \beta e_{22}]\\
&&= -\left|\begin{array}{rr}  \alpha & \beta \\ \tilde\alpha  & \tilde \beta \end{array}\right|^2(e_{11}+e_{22})
 \end{eqnarray*}

\end{enumerate}
\end{lemma}
\begin{proof}
The first equality is verified by explicit calculations.
In order to prove the second one we notice that 

\begin{equation*}\begin{split}[\alpha e_{11} + \beta e_{12}, \alpha e_{21} + \beta e_{22}][\tilde\alpha e_{11} + \tilde\beta e_{12}, \tilde\alpha e_{21} + \tilde\beta e_{22}] \\=
\left(\begin{array}{rr}  \beta & 0 \\ -\alpha  & 0 \end{array}\right)
\left(\begin{array}{rr}  \alpha & \beta \\  0 & 0 \end{array}\right)
\left(\begin{array}{rr}  \tilde\beta & 0 \\ -\tilde\alpha  & 0 \end{array}\right)
\left(\begin{array}{rr}  \tilde\alpha & \tilde\beta \\  0 & 0 \end{array}\right)
\\=\left(\begin{array}{rr}  \beta & 0 \\ -\alpha  & 0 \end{array}\right)
\left(\begin{array}{rr}  \alpha\tilde\beta - \beta\tilde\alpha  & 0 \\  0 & 0 \end{array}\right)
\left(\begin{array}{rr}  \tilde\alpha & \tilde\beta \\  0 & 0 \end{array}\right)
\\
=\left|\begin{array}{rr}  \alpha & \beta \\ \tilde\alpha  & \tilde \beta \end{array}\right|
\left(\begin{array}{rr}  \beta & 0 \\ -\alpha  & 0 \end{array}\right)
\left(\begin{array}{rr}  \tilde\alpha & \tilde\beta \\  0 & 0 \end{array}\right)
=
\left|\begin{array}{rr}  \alpha & \beta \\ \tilde\alpha  & \tilde \beta \end{array}\right|
\left(\begin{array}{rr}  \tilde\alpha \beta & \beta \tilde\beta \\ -\alpha\tilde\alpha & -\alpha\tilde\beta \end{array}\right)
.\end{split}\end{equation*}
\end{proof} 
  
 Now we can prove the existence of a multilinear polynomial $(FT)^*$-non-identity with sufficiently many
 alternations. 
  \begin{lemma}\label{LemmaGrPIexpTriangleAltFracM2}
  Let $T_0, T_1 \subseteq T$ and $\sim$ be, respectively, the subsets and the equivalence relation defined at the beginning of Section~\ref{SectionFracM2Equal}.
  Suppose also that~(\ref{EquationEquivTriangleFracM2}) holds or $T_0=\varnothing$. Then there exist a number $n_0 \in \mathbb N$ such that for every $n\geqslant n_0$
there exist disjoint subsets $X_1$, \ldots, $X_{2k} \subseteq \lbrace x_1, \ldots, x_n
\rbrace$, $k = \left[\frac{n-n_0}{2\dim A}\right]$,
$|X_1| = \ldots = |X_{2k}|=\dim A$ and a polynomial $f \in P^H_n \backslash
\Id^H(A)$ alternating in the variables of each set $X_j$.
  \end{lemma}
  \begin{proof}
  
  Note that $$f_0(x_1,\ldots, x_4, y_1, \ldots, y_4)=\sum_{\sigma,\rho \in S_4} \sign(\sigma\rho) x_{\sigma(1)}\ y_{\rho(1)}\ x_{\sigma(2)}x_{\sigma(3)}x_{\sigma(4)}\ y_{\rho(2)}y_{\rho(3)}y_{\rho(4)}$$
  is a polynomial non-identity for $M_2(F)$ and its values are proportional to the identity matrix.
  (See e.g. \cite[Theorem 5.7.4]{ZaiGia}.)
  
  Let \begin{equation*}\begin{split}f_{t_1, t_2}(x_{t_1,1}, x_{t_1,2}, x_{t_2,1}, x_{t_2,2}) :=\left[x^{h_{t_1}}_{t_1,1}, x^{h_{t_1}}_{t_1,2}\right]\left[x^{h_{t_2}}_{t_2,1}, x^{h_{t_2}}_{t_2,2}\right]+
\left[x^{h_{t_2}}_{t_2,1}, x^{h_{t_2}}_{t_2,2}\right]  \left[x^{h_{t_1}}_{t_1,1}, x^{h_{t_1}}_{t_1,2}\right]\end{split}\end{equation*}
  and \begin{equation*}\begin{split}f_{t_1, t_2, t_3}(x_{t_1,1}, x_{t_1,2}, x_{t_2,1}, x_{t_2,2}, x_{t_3,1}, x_{t_3,2}) := \left[x^{h_{t_1}}_{t_1,1}, x^{h_{t_1}}_{t_1,2}\right] \left[x^{h_{t_3}}_{t_3,1}, x^{h_{t_3}}_{t_3,2}\right] \left[x^{h_{t_2}}_{t_2,1}, x^{h_{t_2}}_{t_2,2}\right] \\ -
   \left[x^{h_{t_2}}_{t_2,1}, x^{h_{t_2}}_{t_2,2}\right] \left[x^{h_{t_3}}_{t_3,1}, x^{h_{t_3}}_{t_3,2} \right]
\left[x^{h_{t_1}}_{t_1,1}, x^{h_{t_1}}_{t_1,2}\right].\end{split}\end{equation*}
  
   Let $\lbrace t_1, \ldots, t_{|T_0|}\rbrace = T_0$ be the elements
  from Lemma~\ref{LemmaTrianglePairsAltFracM2}.
  
  If $2 \mid |T_0|$, then we define \begin{equation*}\begin{split}
  f_1 = z_1 \ldots z_{n-(\dim A)2k}\prod_{i=1}^k
   \left(
   \prod_{t\in T_1}
    f_0(x^{h_t}_{i,t,1},\ldots, x^{h_t}_{i,t,4}, y^{h_t}_{i,t,1}, \ldots, y^{h_t}_{i,t,4})
     \right) \cdot \\
     \cdot
     \left(
     \prod_{\ell=1}^{\frac{|T_0|}2} f_{t_{2\ell-1}, t_{2\ell}}(x_{i,t_{2\ell-1},1}, x_{i,t_{2\ell-1},2}, x_{i,t_{2\ell},1}, x_{i,t_{2\ell},2}) \cdot\right. \\
    \cdot   f_{t_{2\ell-1}, t_{2\ell}}(y_{i,t_{2\ell-1},1}, y_{i,t_{2\ell-1},2}, y_{i,t_{2\ell},1}, y_{i,t_{2\ell},2})
      \Biggr).\end{split}\end{equation*}
  If $2 \nmid |T_0|$, then we define \begin{equation*}\begin{split}
  f_1 = z_1 \ldots z_{n-(\dim A)2k}\prod_{i=1}^k
   \left(
   \prod_{t\in T_1} f_0(x^{h_t}_{i,t,1},\ldots, x^{h_t}_{i,t,4}, y^{h_t}_{i,t,1}, \ldots, y^{h_t}_{i,t,4})
    \right) \cdot \\
     \cdot
    \left(
    \prod_{\ell=1}^{\frac{|T_0|-3}2} f_{t_{2\ell-1}, t_{2\ell}}(x_{i,t_{2\ell-1},1}, x_{i,t_{2\ell-1},2}, x_{i,t_{2\ell},1}, x_{i,t_{2\ell},2}) \cdot \right. \\
    \cdot  
     f_{t_{2\ell-1}, t_{2\ell}}(y_{i,t_{2\ell-1},1}, y_{i,t_{2\ell-1},2}, y_{i,t_{2\ell},1}, y_{i,t_{2\ell},2})
    \Biggr)\cdot \\
     \cdot
    f_{t_{|T_0|-2}, t_{|T_0|-1}, t_{|T_0|}}
    (x_{i,t_{|T_0|-2},1},
     x_{i,t_{|T_0|-2},2};\ 
      x_{i,t_{|T_0|-1},1},
       x_{i,t_{|T_0|-1},2};\ 
        x_{i,t_{|T_0|},1}, 
        x_{i,t_{|T_0|},2})\cdot \\
    \cdot    f_{t_{|T_0|-2}, t_{|T_0|-1}, t_{|T_0|}}
    (y_{i,t_{|T_0|-2},1},
     y_{i,t_{|T_0|-2},2};\ 
      y_{i,t_{|T_0|-1},1},
       y_{i,t_{|T_0|-1},2};\ 
        y_{i,t_{|T_0|},1}, 
        y_{i,t_{|T_0|},2}).\end{split} \end{equation*}

We claim that $f_1 \notin \Id^{(FT)^*}(A)$. 
If $2 \mid |T_0|$, then we take any isomorphism  $\psi \colon A/J(A) \mathrel{\widetilde\rightarrow} M_2(F)$.
If $2 \nmid |T_0|$, we define $\psi$ as follows. First, we notice that $t_{|T_0|-2} \nsim t_{|T_0|-1}$
implies $A/J(A)= I_{t_{|T_0|-2}}\oplus I_{t_{|T_0|-1}}$. By Theorem~\ref{TheoremSumLeftIdealsMatrix}, there exists an isomorphism $\psi \colon A/J(A) \mathrel{\widetilde\rightarrow} M_2(F)$
such that $\psi(I_{t_{|T_0|-2}})=\langle e_{11}, e_{21} \rangle_F$
and $\psi(I_{t_{|T_0|-1}})=\langle e_{12}, e_{22} \rangle_F$.
If $2 \nmid |T_0|$, then we take this isomorphism $\psi$.

Lemma~\ref{LemmaLeftIdealMatrix} implies that for every $t \in T_0$ there exist
$\alpha_t,\beta_t \in F$ such that  $$\psi(I_t)=\langle\alpha_t e_{i1} + \beta_t e_{i2} \mid i=1,2 \rangle_F.$$
If $2 \nmid |T_0|$, by our choice of $\psi$, we may assume that $(\alpha_{t_{|T_0|-2}}, \beta_{t_{|T_0|-2}})=(1, 0)$ and $(\alpha_{t_{|T_0|-1}}, \beta_{t_{|T_0|-1}})=
(0, 1)$. Note that $I_{t_1} = I_{t_2}$ if and only if the rows $(\alpha_{t_1}, \beta_{t_1})$ and $(\alpha_{t_2}, \beta_{t_2})$
are proportional.

Fix some element $e\in A$ such that $\psi\pi(e)$ is the identity matrix.
We substitute $z_1 = \ldots  = z_{n-(\dim A)2k} = e$, $$x_{i,t,1}=y_{i,t,1}=\left(\pi{\Bigr|}_{A^{(t)}}\right)^{-1}\psi^{-1}(e_{11}),\ 
x_{i,t,2}=y_{i,t,2}=\left(\pi{\Bigr|}_{A^{(t)}}\right)^{-1}\psi^{-1}(e_{12}),$$
$$x_{i,t,3}=y_{i,t,3}=\left(\pi{\Bigr|}_{A^{(t)}}\right)^{-1}\psi^{-1}(e_{21}),\ 
x_{i,t,4}=y_{i,t,4}=\left(\pi{\Bigr|}_{A^{(t)}}\right)^{-1}\psi^{-1}(e_{22})$$
for all $t\in T_1$ and $1\leqslant i \leqslant k$
and $$x_{i,t,1}=y_{i,t,1}=\left(\pi{\Bigr|}_{A^{(t)}}\right)^{-1}\psi^{-1}(\alpha_t e_{11} + \beta_t e_{12}),$$ 
$$x_{i,t,2}=y_{i,t,2}=\left(\pi{\Bigr|}_{A^{(t)}}\right)^{-1}\psi^{-1}(\alpha_t e_{21} + \beta_t e_{22})$$
for all $t\in T_0$ and $1\leqslant i \leqslant k$.

In order to show that $f_1$ does not vanish under this evaluation,
we apply $\psi\pi$ to the result of the substitution.
The value of $f_{t_{2\ell-1},t_{2\ell}}$ is nonzero since by Lemma~\ref{LemmaAlphaBetaDeterminant}, \begin{equation*}\begin{split}[\alpha_{t_{2\ell-1}} e_{11} + \beta_{t_{2\ell-1}} e_{12},\ \alpha_{t_{2\ell-1}} e_{21} + \beta_{t_{2\ell-1}} e_{22}][\alpha_{t_{2\ell}} e_{11} + \beta_{t_{2\ell}} e_{12},\ \alpha_{t_{2\ell}} e_{21} + \beta_{t_{2\ell}} e_{22}]\\
+[\alpha_{t_{2\ell}} e_{11} + \beta_{t_{2\ell}} e_{12},\ \alpha_{t_{2\ell}} e_{21} + \beta_{t_{2\ell}} e_{22}][\alpha_{t_{2\ell-1}} e_{11} + \beta_{t_{2\ell-1}} e_{12},\ \alpha_{t_{2\ell-1}} e_{21} + \beta_{t_{2\ell-1}} e_{22}]\\=-\left|\begin{array}{cc}\alpha_{t_{2\ell-1}} & \beta_{t_{2\ell-1}} \\
\alpha_{t_{2\ell}} & \beta_{t_{2\ell}} \end{array}\right|^2(e_{11}+e_{22})\ne 0.\end{split} \end{equation*}
The polynomial
$f_{t_{|T_0|-2}, t_{|T_0|-1}, t_{|T_0|}}$ does not vanish under this evaluation since, by Lemma~\ref{LemmaAlphaBetaDeterminant}, \begin{equation*}\begin{split}[e_{11}, e_{21}][\alpha_{t_{|T_0|}} e_{11} + \beta_{t_{|T_0|}} e_{12},
\alpha_{t_{|T_0|}}e_{21} + \beta_{t_{|T_0|}} e_{22}]
[e_{12}, e_{22}] \\- [e_{12}, e_{22}]
[\alpha_{t_{|T_0|}} e_{11} + \beta_{t_{|T_0|}} e_{12},
\alpha_{t_{|T_0|}}e_{21} + \beta_{t_{|T_0|}} e_{22}]
[e_{11}, e_{21}] = -\alpha_{t_{|T_0|}}\beta_{t_{|T_0|}}(e_{11}+e_{22})\ne 0.\end{split} \end{equation*}
Thus $f_1 \notin \Id^{(FT)^*}(A)$.
Now we define $f = 
\Alt_1 \ldots \Alt_{2k} f_1$ where $\Alt_i$ is the operator of alternation on the set
$X_i$ where $$X_{2i-1} = \lbrace x_{i,t,j} \mid t\in T,\
1\leqslant j \leqslant 2\text{ for }t\in T_0, \
1\leqslant j \leqslant 4\text{ for }t\in T_1 \rbrace,$$
$$ X_{2i} = \lbrace y_{i,t,j} \mid t\in T,\
1\leqslant j \leqslant 2\text{ for }t\in T_0,\
1\leqslant j \leqslant 4\text{ for }t\in T_1 \rbrace,$$ $1\leqslant i \leqslant 2k$.
Note that $f$ does not vanish under the same substitution that we used for $f_1$ since
if the alternation replaces $x_{i_1,t_1,j_1}$ with $x_{i_2,t_2,j_2}$ for $t_1 \ne t_2$,
then the value of $x^{h_{t_1}}_{i_2,t_2,j_2}$ is zero and
 the corresponding item vanishes.

 For our convenience, we rename the variables of $f$ to $x_1, \ldots, x_n$.
  Then $f$ satisfies all the conditions of the lemma.
  \end{proof}

Now we are ready to prove that in the case, when~(\ref{EquationEquivTriangleFracM2}) holds,
$\PIexp^{T\text{-}\mathrm{gr}}(A) =\dim A$.
  
  \begin{theorem}\label{TheoremGrPIexpTriangleFracM2} Let $A$ be a finite dimensional $T$-graded-simple algebra over a field $F$
   of characteristic $0$ for a right zero band $T$.
  Suppose $A/J(A) \cong M_2(F)$.
Let $T_0, T_1 \subseteq T$ and $\sim$ be, respectively, the subsets and the equivalence relation defined at the beginning of Section~\ref{SectionFracM2Equal}.
  Suppose also that~(\ref{EquationEquivTriangleFracM2}) holds or $T_0=\varnothing$.
  Then there exist $C > 0$, $r\in \mathbb R$, such that
  $C n^r (\dim A)^n \leqslant c^{T\text{-}\mathrm{gr}}_n(A) \leqslant (\dim A)^{n+1}$.
In particular, $\PIexp^{T\text{-}\mathrm{gr}}(A) =\dim A$.
  \end{theorem}
\begin{proof}
Note that by Lemma~\ref{LemmaCnGrCnGenH}, $c_n^{T\text{-}\mathrm{gr}}(A)=c_n^{(FT)^*}(A)$ for all $n\in \mathbb N$. The upper bound follows from standard arguments. (See, e.g., \cite[Lemma~4]{ASGordienko3}.)
In order to get the lower bound, we repeat literally the proof of~\cite[Lemma~11 and Theorem~5]{ASGordienko3} using Lemma~\ref{LemmaGrPIexpTriangleAltFracM2}
instead of~\cite[Lemma~10]{ASGordienko3}.
\end{proof}

\section{The case $A/J(A)\cong M_2(F)$ and $\PIexp^{T\text{-}\mathrm{gr}}(A)<\dim A$}\label{SectionFracM2Less}

Let $A$ be a finite dimensional $T$-graded-simple algebra over a field $F$  of characteristic $0$ for a right zero band $T$. Suppose $A/J(A) \cong M_2(F)$.
Let $T_0, T_1 \subseteq T$ and $\sim$ be, respectively, the subsets and the equivalence relation defined at the beginning of Section~\ref{SectionFracM2Equal}.

Suppose that $T_0 \ne \varnothing$ and the inequality~(\ref{EquationEquivTriangleFracM2}) does not hold in $A$. This is equivalent to the existence of $t_0 \in T_0$ such that 
$|\bar t_0| > \frac{|T_0|}{2}$. Using Theorem~\ref{TheoremSumLeftIdealsMatrix}, we fix an isomorphism $\psi \colon A/J(A) \rightarrow M_2(F)$
such that $\psi(I_{t_0})=\langle e_{11}, e_{21} \rangle_F$.
By Lemma~\ref{LemmaLeftIdealMatrix},
for every $t\in T_0$ one can choose $\alpha_t, \beta_t\in F$
such that $( \alpha_t e_{i1}+\beta_t e_{i2} \mathrel{|} 1\leqslant i \leqslant 2)$ is a basis of $\psi(I_t)$. We may assume that $(\alpha_t, \beta_t)=(1,0)$ for $t\sim t_0$. Note that $\beta_t \ne 0$ if $I_t\ne I_{t_0}$.

Now we fix the basis
 $$\left(\left(\pi\bigr|_{A^{(t)}}\right)^{-1}\psi^{-1}(\alpha_t e_{i1}+\beta_t e_{i2}) \mathrel{\Bigl|} 1\leqslant i \leqslant 2 \right)$$ in $A^{(t)}$ for each $t\in T_0$ and
 $$\left(\left(\pi\bigr|_{A^{(t)}}\right)^{-1}\psi^{-1}(e_{ij}) \mathrel{\Bigl|} 1\leqslant i,j\leqslant 2 \right)$$ in $A^{(t)}$ for each $t\in T_1$.
We define the basis $\mathcal B$ in $A$ as the union of the bases in $A^{(t)}$
chosen above.

Now we calculate the numbers $\gamma_i$ introduced at the beginning of Section~\ref{SectionUpperFrac}.
We notice that $$\theta\left(\left(\pi\bigr|_{A^{(t)}}\right)^{-1}\psi^{-1}(\alpha_t e_{i1}+\beta_t e_{i2})\right) 
=i-2$$ for $1\leqslant i \leqslant 2$ and $t\in T_0$, $t\nsim t_0$.
Then \begin{equation}\label{EquationGammaFracM2}(\gamma_1, \ldots, \gamma_r)=(\underbrace{-1,\ldots, -1}_{|T_0|+|T_1|-|\bar t_0|},
 \underbrace{0,\ldots,0}_{|T_0|+2|T_1|},
\underbrace{1,\ldots,1}_{|T_1|+|\bar t_0|})\end{equation} and the graded cocharacter of $A$ satisfies the inequality from Lemma~\ref{LemmaInequalityLambdaUpperFrac}.

Below we prove three lemmas which enable us to choose elements $b_1,\ldots, b_m$ that we will substitute
for the variables corresponding to the numbers in a column of a given Young diagram.
Here is important to control the sum $\sum_{j=1}^m \theta(b_i)$.

\begin{lemma}\label{LemmaMSummaGammaFracM2}
Let $1 \leqslant m \leqslant r$.
Then
$$m-\sum_{j=1}^m \gamma_j \leqslant  3|T_0|+4|T_1|-2|\bar t_0|.$$
 \end{lemma}
 \begin{proof} If $m\leqslant |T_0|+|T_1|-|\bar t_0|$,
 then $\gamma_j=-1$ for $1\leqslant j \leqslant m$ and $\sum_{j=1}^m \gamma_j = -m$. Hence $$m-\sum_{j=1}^m \gamma_j =  2m \leqslant 2|T_0|+2|T_1|-2|\bar t_0| \leqslant  3|T_0|+4|T_1|-2|\bar t_0|.$$
 
 If $|T_0|+|T_1|-|\bar t_0| \leqslant m \leqslant 2|T_0|+3|T_1|-|\bar t_0|$,
 then $\gamma_j= 0$ for $|T_0|+|T_1|-|\bar t_0| < j \leqslant m$ and
 $\sum_{j=1}^m \gamma_j = -(|T_0|+|T_1|-|\bar t_0|)$.
 Hence $$m-\sum_{j=1}^m \gamma_j =  m+(|T_0|+|T_1|-|\bar t_0|)\leqslant 3|T_0|+4|T_1|-2|\bar t_0|.$$
 
 If $m \geqslant 2|T_0|+3|T_1|-|\bar t_0|$,
 then $\gamma_j=1$ for $j> 2|T_0|+3|T_1|-|\bar t_0|$ implies $$m-\sum_{j=1}^m \gamma_j=3|T_0|+4|T_1|-2|\bar t_0|.$$
 \end{proof}

\begin{lemma}\label{LemmaChooseForAColumnPositiveFracM2}
Let $\sum_{j=1}^m \gamma_j > 0$ for some $m\in\mathbb N$.
Then there exists $b_1, \ldots, b_m \in \mathcal B$, $b_i\ne b_j$ for $i\ne j$,
such that $\sum_{j=1}^m \theta(b_j) = \sum_{j=1}^m \gamma_j$
and 
\begin{itemize}
\item if $\lbrace b_1, \ldots, b_m \rbrace \cap A^{(t)}=\lbrace b_i \rbrace$
for some $1\leqslant i \leqslant m$ and $t\in T_0 \sqcup T_1$,
then $t\in T_0$, $t\sim t_0$ and $b_i=\left(\pi\bigr|_{A^{(t)}}\right)^{-1}\psi^{-1}(e_{11})$;
\item if $\lbrace b_1, \ldots, b_m \rbrace \cap A^{(t)}=\lbrace b_i, b_j \rbrace$
for some $1\leqslant i,j \leqslant m$ and $t\in T_0 \sqcup T_1$,
then either $\theta(b_i)\ne 0$ or $\theta(b_j)\ne 0$.
\end{itemize}
\end{lemma}
\begin{proof}
By the definition of $\gamma_i$ (see the beginning of Section~\ref{SectionUpperFrac}),
there exists $b_1, \ldots, b_m \in \mathcal B$, $b_i\ne b_j$ for $i\ne j$,
such that $\sum_{j=1}^m \theta(b_j) = \sum_{j=1}^m \gamma_j$
and $\sum_{j=1}^m \theta(a_j) \geqslant \sum_{j=1}^m \gamma_j$
for all $a_1, \ldots, a_m \in \mathcal B$ where $a_i\ne a_j$ for $i\ne j$.
Since $\sum_{j=1}^m \gamma_j > 0$, the minimality of $\sum_{j=1}^m \theta(b_j)$
implies that the set $\lbrace b_1, \ldots, b_m \rbrace$
contains all elements $b\in \mathcal B$ with $\theta(b) \leqslant 0$.
Now the choice of $\mathcal B$ implies the lemma.
\end{proof}

\begin{lemma}\label{LemmaChooseForAColumnNegativeFracM2}
Let $\sum_{j=1}^m \gamma_j \leqslant q \leqslant 0$ for some $m\in\mathbb N$, $q\in\mathbb Z$.
Then there exists $b_1, \ldots, b_m \in \mathcal B$, $b_i\ne b_j$ for $i\ne j$,
such that $\sum_{j=1}^m \theta(b_j) = q$
and \begin{itemize}
\item if $\lbrace b_1, \ldots, b_m \rbrace \cap A^{(t)}=\lbrace b_i \rbrace$,
$\theta(b_i)=0$,
for some $1\leqslant i \leqslant m$ and $t\in T_0 \sqcup T_1$,
then $t\in T_0$, $t\sim t_0$ and $b_i=\left(\pi\bigr|_{A^{(t)}}\right)^{-1}\psi^{-1}(e_{11})$;
\item if $\lbrace b_1, \ldots, b_m \rbrace \cap A^{(t)}=\lbrace b_i, b_j \rbrace$
for some $1\leqslant i,j \leqslant m$ and $t\in T_0 \sqcup T_1$,
then either $\theta(b_i)\ne 0$ or $\theta(b_j)\ne 0$.
\end{itemize}
\end{lemma}
\begin{proof}

Recall that $$|\lbrace b\in\mathcal B \mid \theta(b)=-1  \rbrace|=|T_0|+|T_1|-|\bar t_0|,$$
$$|\lbrace b\in\mathcal B \mid \theta(b)=0  \rbrace|=|T_0|+2|T_1|,$$
$$|\lbrace b\in\mathcal B \mid \theta(b)=1  \rbrace|=|T_1|+|\bar t_0|.$$
Let $\lbrace t_1, \ldots, t_{|T_1|} \rbrace := T_1$,
$\lbrace \tilde t_1, \ldots, \tilde t_{|T_0|-|\bar t_0|} \rbrace := T_0 \backslash \bar t_0$,
$\lbrace \hat t_1, \ldots, \hat t_{|\bar t_0|} \rbrace := \bar t_0$.

Now we consider two main cases:

1. Suppose $m < 2(|T_0|+|T_1|-|\bar t_0|)+q$. Let $\ell = \left[
\frac{m+q}{2}\right]$. Note that $$\ell - q \leqslant \frac{m-q}{2} < |T_0|+|T_1|-|\bar t_0|$$
and $\ell \leqslant \ell - q < |T_0|+|T_1|-|\bar t_0| < |T_1|+|\bar t_0|$.
These inequalities imply that below we have enough elements from $(T_0\backslash \bar t_0) \sqcup T_1$
and $\bar t_0 \sqcup T_1$, respectively.

 Suppose first that $m=2\ell-q$. 
If $\ell \leqslant |T_1|+q$, then we take
\begin{equation*}\begin{split}\lbrace b_1, \ldots, b_m \rbrace = 
\left\lbrace\left(\pi\bigr|_{A^{(t_i)}}\right)^{-1}\psi^{-1}(e_{12})
\mathrel{\bigl|} 1\leqslant i \leqslant \ell-q
\right\rbrace  \\ \cup \left\lbrace\left(\pi\bigr|_{A^{(t_i)}}\right)^{-1}\psi^{-1}(e_{21})
\mathrel{\bigl|} 1\leqslant i \leqslant \ell
\right\rbrace.\end{split}\end{equation*}
If $|T_1|+q < \ell \leqslant |T_1|$, then we take
\begin{equation*}\begin{split}\lbrace b_1, \ldots, b_m \rbrace = 
\left\lbrace\left(\pi\bigr|_{A^{(t_i)}}\right)^{-1}\psi^{-1}(e_{12})
\mathrel{\bigl|} 1\leqslant i \leqslant |T_1|
\right\rbrace  \\ \cup \left\lbrace\left(\pi\bigr|_{A^{(\tilde t_i)}}\right)^{-1}\psi^{-1}(\alpha_{\tilde t_i} e_{11}+\beta_{\tilde t_i} e_{12})
\mathrel{\bigl|} 1\leqslant i \leqslant \ell-q-|T_1|
\right\rbrace \\ \cup \left\lbrace\left(\pi\bigr|_{A^{(t_i)}}\right)^{-1}\psi^{-1}(e_{21})
\mathrel{\bigl|} 1\leqslant i \leqslant \ell
\right\rbrace.\end{split}\end{equation*}
If $\ell > |T_1|$, then we take
\begin{equation*}\begin{split}\lbrace b_1, \ldots, b_m \rbrace = 
\left\lbrace\left(\pi\bigr|_{A^{(t_i)}}\right)^{-1}\psi^{-1}(e_{12})
\mathrel{\bigl|} 1\leqslant i \leqslant |T_1|
\right\rbrace  \\ \cup \left\lbrace\left(\pi\bigr|_{A^{(\tilde t_i)}}\right)^{-1}\psi^{-1}(\alpha_{\tilde t_i} e_{11}+\beta_{\tilde t_i} e_{12})
\mathrel{\bigl|} 1\leqslant i \leqslant \ell-q-|T_1|
\right\rbrace  \\ \cup \left\lbrace\left(\pi\bigr|_{A^{(t_i)}}\right)^{-1}\psi^{-1}(e_{21})
\mathrel{\bigl|} 1\leqslant i \leqslant |T_1|
\right\rbrace  \\ \cup \left\lbrace\left(\pi\bigr|_{A^{(\hat t_i)}}\right)^{-1}\psi^{-1}(e_{21})
\mathrel{\bigl|} 1\leqslant i \leqslant \ell-|T_1|
\right\rbrace.\end{split}\end{equation*}
If $m=2\ell-q+1$,
then in each of the three cases above we add the element $\left(\pi\bigr|_{A^{(\hat t_1)}}\right)^{-1}\psi^{-1}(e_{11})$.

2. Suppose $m \geqslant 2(|T_0|+|T_1|-|\bar t_0|)+q$.

By Lemma~\ref{LemmaMSummaGammaFracM2},
$$m-\sum_{j=1}^m \gamma_j \leqslant  3|T_0|+4|T_1|-2|\bar t_0|.$$
Hence
$$m-q - 2(|T_0|+|T_1|-|\bar t_0|) \leqslant  |T_0|+2|T_1|
$$ 
and we can choose 
$$0\leqslant k, \ell \leqslant |T_1|,\qquad 0\leqslant s \leqslant |\bar t_0|,\qquad 0\leqslant u \leqslant |T_0|-|\bar t_0|,$$ such that $2(|T_0|+|T_1|-|\bar t_0|)+q + k + \ell + s + u = m$.

Note that $$(|T_0|+|T_1|-|\bar t_0|)+q \geqslant (|T_0|+|T_1|-|\bar t_0|) + \sum_{j=1}^m \gamma_j
\geqslant 0.$$
If $(|T_0|+|T_1|-|\bar t_0|)+q \leqslant |T_1|$, we define
\begin{equation*}\begin{split}\lbrace b_1, \ldots, b_m \rbrace = 
\left\lbrace\left(\pi\bigr|_{A^{(t_i)}}\right)^{-1}\psi^{-1}(e_{12})
\mathrel{\bigl|} 1\leqslant i \leqslant |T_1|
\right\rbrace \\ \cup \left\lbrace\left(\pi\bigr|_{A^{(t_i)}}\right)^{-1}\psi^{-1}(e_{11})
\mathrel{\bigl|} 1\leqslant i \leqslant k
\right\rbrace \\ \cup \left\lbrace\left(\pi\bigr|_{A^{(t_i)}}\right)^{-1}\psi^{-1}(e_{21})
\mathrel{\bigl|} 1\leqslant i \leqslant (|T_0|+|T_1|-|\bar t_0|)+q
\right\rbrace  \\ \cup \left\lbrace\left(\pi\bigr|_{A^{(t_i)}}\right)^{-1}\psi^{-1}(e_{22})
\mathrel{\bigl|} 1\leqslant i \leqslant \ell
\right\rbrace  \\ \cup \left\lbrace\left(\pi\bigr|_{A^{(\tilde t_i)}}\right)^{-1}\psi^{-1}(\alpha_{\tilde t_i} e_{11}+\beta_{\tilde t_i} e_{12})
\mathrel{\bigl|} 1\leqslant i \leqslant |T_0|-|\bar t_0|
\right\rbrace  \\ \cup \left\lbrace\left(\pi\bigr|_{A^{(\tilde t_i)}}\right)^{-1}\psi^{-1}(\alpha_{\tilde t_i} e_{21}+\beta_{\tilde t_i} e_{22})
\mathrel{\bigl|} 1\leqslant i \leqslant u
\right\rbrace  \\ \cup \left\lbrace\left(\pi\bigr|_{A^{(\hat t_i)}}\right)^{-1}\psi^{-1}(e_{11})
\mathrel{\bigl|} 1\leqslant i \leqslant s
\right\rbrace.\end{split}\end{equation*}
If $(|T_0|+|T_1|-|\bar t_0|)+q > |T_1|$, we define
\begin{equation*}\begin{split}\lbrace b_1, \ldots, b_m \rbrace = 
\left\lbrace\left(\pi\bigr|_{A^{(t_i)}}\right)^{-1}\psi^{-1}(e_{12})
\mathrel{\bigl|} 1\leqslant i \leqslant |T_1|
\right\rbrace  \\ \cup \left\lbrace\left(\pi\bigr|_{A^{(t_i)}}\right)^{-1}\psi^{-1}(e_{11})
\mathrel{\bigl|} 1\leqslant i \leqslant k
\right\rbrace  \\ \cup \left\lbrace\left(\pi\bigr|_{A^{(t_i)}}\right)^{-1}\psi^{-1}(e_{21})
\mathrel{\bigl|} 1\leqslant i \leqslant |T_1|
\right\rbrace  \\ \cup \left\lbrace\left(\pi\bigr|_{A^{(t_i)}}\right)^{-1}\psi^{-1}(e_{22})
\mathrel{\bigl|} 1\leqslant i \leqslant \ell
\right\rbrace  \\ \cup \left\lbrace\left(\pi\bigr|_{A^{(\tilde t_i)}}\right)^{-1}\psi^{-1}(\alpha_{\tilde t_i} e_{11}+\beta_{\tilde t_i} e_{12})
\mathrel{\bigl|} 1\leqslant i \leqslant |T_0|-|\bar t_0|
\right\rbrace  \\ \cup \left\lbrace\left(\pi\bigr|_{A^{(\tilde t_i)}}\right)^{-1}\psi^{-1}(\alpha_{\tilde t_i} e_{21}+\beta_{\tilde t_i} e_{22})
\mathrel{\bigl|} 1\leqslant i \leqslant u
\right\rbrace \\ \cup \left\lbrace\left(\pi\bigr|_{A^{(\hat t_i)}}\right)^{-1}\psi^{-1}(e_{11})
\mathrel{\bigl|} 1\leqslant i \leqslant s
\right\rbrace \\ \cup \left\lbrace\left(\pi\bigr|_{A^{(\hat t_i)}}\right)^{-1}\psi^{-1}(e_{21})
\mathrel{\bigl|} 1\leqslant i \leqslant (|T_0|-|\bar t_0|)+q
\right\rbrace.\end{split}
\end{equation*}
\end{proof}

In the lemma below we present a graded polynomial non-identity having sufficiently many alternations.
This non-identity will generate an $FS_n$-submodule that will have the dimension great enough to prove the lower bound for graded codimensions.

\begin{lemma}\label{LemmaAltNonTriangleFracM2}
Let $\lambda=(\lambda_1, \ldots, \lambda_r) \vdash n$ for some $n\in\mathbb N$. If $\sum_{i=1}^r \gamma_i \lambda_i \leqslant 0$ and $2 \mid \lambda_i$ for $i \geqslant 2$, then there exists a Young tableau $T_\lambda$ of shape $\lambda$ and an $(FT)^*$-polynomial $f\in P_n^{(FT)^*}$ such that  $b_{T_\lambda}f :=\sum_{\sigma\in C_{T_\lambda}}(\sign \sigma)\sigma f\notin \Id^{(FT)^*}(A)$.
\end{lemma}
\begin{proof}
Let $\mu_i$ be the number of boxes in the $i$th column of the tableau $T_\lambda$
and let $m_i := \sum_{j=1}^{\mu_i} \gamma_j$, $1\leqslant i \leqslant \lambda_1$.

Note that since  $(\lambda_i-\lambda_{i+1})$
equals the number of columns of height $i$ (here $\lambda_{r+1}:=0$) and $\lambda_1$ equals the number of all columns, the inequality $\sum_{i=1}^r \gamma_i \lambda_i \leqslant 0$ can be rewritten as
\begin{equation}\label{EquationGammaLambdaFracM2}\sum_{i=1}^r \left(\sum_{j=1}^i \gamma_j -
\sum_{j=1}^{i-1} \gamma_j\right) \lambda_i=\sum_{i=1}^r \left(\sum_{j=1}^i \gamma_j\right) (\lambda_i-\lambda_{i+1}) = \sum_{i=1}^{\lambda_1} m_i \leqslant 0. \end{equation} 
By Lemma~\ref{LemmaThetaCondition}, if the sum of the values of $\theta$ on basis elements
substituted for the variables of a multilinear $(FT)^*$-polynomial is less than $(-1)$ or greater than $1$,
then the $(FT)^*$-polynomial vanishes. We will choose elements $b_{itj}\in A^{(t)} \cap \mathcal B$ such that the sum of the values of $\theta$ on them equals $0$. If $m_i > 0$
for some $i$,
we have to make the sum of values of $\theta$ for some other columns negative.

Define the number $1\leqslant \ell \leqslant \lambda_1$ by the conditions $m_\ell > 0$ and $m_{\ell+1} \leqslant 0$.
Since $2 \mid \lambda_i$ for all $i \geqslant 2$, we have $m_{2j-1}=m_{2j}$
for $1 \leqslant j \leqslant \frac{\lambda_2}2$. Hence $2 \mid \ell$ and $2 \mid \sum_{i=1}^{\ell} m_i$.
Recall that $m_i=-1$ for $\lambda_2+1\leqslant i \leqslant \lambda_1$. By~(\ref{EquationGammaLambdaFracM2}), 
$$2\sum_{i=1}^{\lambda_2/2} m_{2i} -(\lambda_1-\lambda_2)=\sum_{i=1}^{\lambda_1} m_i \leqslant 0$$
and 
$$\sum_{i=1}^{\lambda_2/2} m_{2i} - \left[\frac{\lambda_1-\lambda_2}2\right] \leqslant 0.$$

Thus one can choose integers $N$ and $q_{2i}$ for $\ell < 2i \leqslant \lambda_2$,
such that $0\leqslant N \leqslant \left[\frac{\lambda_1-\lambda_2}2\right]$, $m_{2i}\leqslant q_{2i} \leqslant 0$
and $\sum_{i=1}^{\ell/2} m_{2i}+\sum_{i=\ell/2+1}^{\lambda_2} q_{2i} - N = 0$.
Define $q_{2i-1}:=q_{2i}$ for $\frac{\ell}2 < i \leqslant \frac{\lambda_2}2$,
$q_i:=-1$ for $\lambda_2+1 \leqslant i \leqslant \lambda_2+2N$,
$q_i:=0$ for $\lambda_2+2N+1 \leqslant i \leqslant \lambda_1$.
Then \begin{equation}\label{EquationSumMiQiZeroFracM2}
\sum_{i=1}^{\ell} m_i+\sum_{i=\ell+1}^{\lambda_1} q_i=0.
\end{equation}

Now we use 
 Lemma~\ref{LemmaChooseForAColumnPositiveFracM2} for every $1\leqslant i \leqslant \ell$ (there $m=\mu_i$)
and Lemma~\ref{LemmaChooseForAColumnNegativeFracM2}
for every $\ell+1\leqslant i \leqslant \lambda_1$ (there $m=\mu_i$ and $q=q_i$).
We sort the obtained elements $b_j$ in accordance with the homogeneous
components $A^{(t)}$ they belong to.
For a fixed $1\leqslant i \leqslant \lambda_1$
 we get elements $b_{itj}\in A^{(t)} \cap \mathcal B$ where $t\in T_0 \sqcup T_1$, $1\leqslant j \leqslant n_{it}$, the total number of $b_{itj}$ equals $n_{it}\geqslant 0$, $$\sum_{t\in T_0 \sqcup T_1} n_{it}= \mu_i,$$
 $$b_{it_1j_1} \ne b_{it_2j_2} \text{ if } (t_1,j_1)\ne (t_2,j_2),$$
$$\sum_{t\in T_0 \sqcup T_1} \sum_{j=1}^{n_{it}} \theta(b_{itj}) =m_i \text{ if } m_i > 0$$ and
$$m_i \leqslant \sum_{t\in T_0 \sqcup T_1} \sum_{j=1}^{n_{it}} \theta(b_{itj})=q_i \leqslant 0 \text{ if } m_i \leqslant 0.$$

By the virtue of~(\ref{EquationSumMiQiZeroFracM2}),
\begin{equation}\label{EquationSumThetaBitj0FracM2}
\sum_{i=1}^{\lambda_1} 
\sum_{t\in T_0 \sqcup T_1} \sum_{j=1}^{n_{it}} \theta(b_{itj}) = 0.
\end{equation}

Since $q_{2i-1}=q_{2i}$ and $\mu_{2i-1}=\mu_{2i}$ if $2i \leqslant \lambda_2$, we may assume that $n_{{2i-1},t}=n_{{2i},t}$,
$b_{{2i-1},t,j}=b_{{2i},t,j}$
for all $1 < 2i \leqslant \lambda_2$, $t\in T_0 \sqcup T_1$, $1\leqslant j \leqslant
n_{{2i},t}$.

We will substitute elements $b_{itj}$ for the variables with the indices from the $i$th column. 

Denote by $W_{-1}$  the set of all pairs $(i, t)$ where $1\leqslant i \leqslant \lambda_1$, $t\in T_0 \sqcup T_1$, such that $\theta(b_{itj})=-1$ for some $1\leqslant j \leqslant n_{ti}$
and $\theta(b_{itj})\leqslant 0$ for all $1\leqslant j \leqslant n_{ti}$.
Denote by $W_1$  the set of all pairs $(i, t)$ where $1\leqslant i \leqslant \lambda_1$, $t\in T_0 \sqcup T_1$, such that $\theta(b_{itj})=1$ for some $1\leqslant j \leqslant n_{ti}$
and $\theta(b_{itj})\geqslant 0$ for all $1\leqslant j \leqslant n_{ti}$.
Let $W_1^{(i)} := \lbrace t\in T_0 \sqcup T_1 \mid
(i,t)\in W_1 \rbrace$ for $1\leqslant i \leqslant \lambda_1$
and $W_0^{(i)} := \lbrace t\in T_0 \sqcup T_1 \mid
(i,t)\notin W_{-1} \sqcup W_1, n_{it} > 0 \rbrace$.

By~(\ref{EquationSumThetaBitj0FracM2}), $|W_{-1}| = |W_1|$.
Therefore, there exist maps $\varkappa \colon W_1 \rightarrow \lbrace 1,\ldots, \lambda_1\rbrace$
and $\rho \colon W_1 \rightarrow T_0 \sqcup T_1$ such that $(i,t) \mapsto (\varkappa(i,t),\rho(i,t))$
is a bijection $W_1 \rightarrow W_{-1}$.

  Define polynomials $f_{it}$ and $\tilde f_{it}$, $1\leqslant i \leqslant \lambda_2$, $t\in T_0 \sqcup T_1$, as follows.

If $n_{it}=1$, then $f_{it}(x_1)=x_1$.

If $n_{it}=2$, then $f_{it}(x_1, x_2)=x_1 x_2 - x_2 x_1$.

If $n_{it}=3$, then $f_{it}(x_1,x_2, x_3)=\sum_{\sigma \in S_3} \sign(\sigma) x_{\sigma(1)}x_{\sigma(2)}x_{\sigma(3)}$.

If $1\leqslant n_{it}\leqslant 3$, then $$\tilde f_{it}(x_1,\ldots, x_{n_{it}}; y_1,\ldots, y_{n_{it}})
=f_{it}(x_1,\ldots, x_{n_{it}})f_{it}(y_1,\ldots, y_{n_{it}}).$$

If $n_{it}=4$, then $$\tilde f_{it}(x_1,\ldots, x_4, y_1, \ldots, y_4)=\sum_{\sigma,\tau \in S_4} \sign(\sigma\tau) x_{\sigma(1)}\ y_{\tau(1)}\ x_{\sigma(2)}x_{\sigma(3)}x_{\sigma(4)}\ y_{\tau(2)}y_{\tau(3)}y_{\tau(4)}.$$
  Recall this is a polynomial non-identity for $M_2(F)$ and its values are proportional to the identity matrix.
  (See e.g. \cite[Theorem 5.7.4]{ZaiGia}.)

Let $X_i:=\lbrace x_{itj} \mid 1\leqslant j \leqslant n_{it},\ t\in T_0 \sqcup T_1 \rbrace$, $1\leqslant i \leqslant \lambda_1$.
Denote by $\Alt_i$ the operator of alternation in the variables from $X_i$.

Define \begin{equation*}\begin{split}f := \Alt_1 \Alt_2 \ldots \Alt_{\lambda_1}
\prod_{i=1}^{\lambda_2/2} \left(\prod_{t\in W_0^{(2i-1)}}
 \tilde f_{2i-1,t}(x_{2i-1,t,1}^{h_t},\ldots, x_{2i-1,t,n_{2i-1,t}}^{h_t};\ x_{2i,t,1}^{h_t},\ldots, x_{2i,t,n_{2i,t}}^{h_t}) \right. \cdot \\ \cdot \left.
 \left(\prod_{t\in W_1^{(2i-1)}}
 f_{\varkappa(2i-1,t)\rho(2i-1,t)}\left(x_{\varkappa(2i-1,t)\rho(2i-1,t)1}^{h_{\rho(2i-1,t)}},\ldots, x_{\varkappa(2i-1,t)\rho(2i-1,t)n_{\varkappa(2i-1,t)\rho(2i-1,t)}}^{h_{\rho(2i-1,t)}}\right)
\right.\right. \cdot \\ \cdot \left.\left.
 f_{2i-1,t}(x_{2i-1,t,1}^{h_t},\ldots, x_{2i-1,t,n_{2i-1,t}}^{h_t})
  \right.\right. \cdot \\ \cdot 
  f_{\varkappa(2i,t)\rho(2i,t)}\left(x_{\varkappa(2i,t)\rho(2i,t)1}^{h_{\rho(2i,t)}},\ldots, x_{\varkappa(2i,t)\rho(2i,t)n_{\varkappa(2i,t)\rho(2i,t)}}^{h_{\rho(2i,t)}}\right)
   \cdot \\ \cdot \left.\left.
 f_{2i,t}(x_{2i,t,1}^{h_t},\ldots, x_{2i,t,n_{2i,t}}^{h_t}) 
 \right) \right)
 \prod_{\substack{i=\lambda_2+1, \\ t\in W^{(i)}_0}}^{\lambda_1} x_{it1}^{h_t}.\end{split}\end{equation*}
 
Note that by Lemmas~\ref{LemmaChooseForAColumnPositiveFracM2}, \ref{LemmaChooseForAColumnNegativeFracM2}, if $(i,t)\in W_{-1}$, then $\lbrace\psi\pi(b_{it1}),\ldots, \psi\pi(b_{itn_{it}})\rbrace$
coincides with one of the following sets: $\lbrace e_{12}\rbrace$, $\lbrace \alpha_t e_{11}+\beta_t e_{12} \rbrace$, $\lbrace e_{12}, e_{22}\rbrace$, $\lbrace e_{11}, e_{12}\rbrace$,
$\lbrace \alpha_t e_{11}+\beta_t e_{12}, \alpha_t e_{21}+\beta_t e_{22} \rbrace$,
 $\lbrace e_{11}, e_{12}, e_{22}\rbrace$.

If $t\in W_0^{(i)}$, then $\lbrace\psi\pi(b_{it1}),\ldots, \psi\pi(b_{itn_{it}})\rbrace$
coincides with one of the following sets: $\lbrace e_{11}\rbrace$, $\lbrace e_{12}, e_{21}\rbrace$,
$\lbrace e_{12}, e_{22}, e_{21}\rbrace$, $\lbrace e_{11}, e_{12}, e_{21}\rbrace$,
$\lbrace e_{11}, e_{12}, e_{22}, e_{21}\rbrace$.

If $t\in W_1^{(i)}$, then $\lbrace\psi\pi(b_{it1}),\ldots, \psi\pi(b_{itn_{it}})\rbrace$
coincides with one of the following sets: $\lbrace e_{21}\rbrace$, $\lbrace e_{22}, e_{21}\rbrace$,
 $\lbrace e_{21}, e_{11}\rbrace$,
$\lbrace e_{22}, e_{21}, e_{11}\rbrace$.

By Lemma~\ref{LemmaAlphaBetaDeterminant} and the remarks above,
the image of the value  of $f$ under the substitution $x_{itj}=b_{itj}$, $1\leqslant i \leqslant \lambda_1$,
$t\in T_0 \sqcup T_1$, $1\leqslant j \leqslant n_{it}$,
 does not vanish under the homomorphism $\psi\pi$ since
if the alternation replaces $x_{i_1,t_1,j_1}$ with $x_{i_2,t_2,j_2}$ for $t_1 \ne t_2$,
then the value of $x^{h_{t_1}}_{i_2,t_2,j_2}$ is zero and
 the corresponding item vanishes.
 For our convenience, we rename the variables of $f$ to $x_1, \ldots, x_n$.
  Then $f$ satisfies all the conditions of the lemma.
\end{proof}

Now we are ready to calculate $\PIexp^{T\text{-}\mathrm{gr}}(A)$ in the case
when~(\ref{EquationEquivTriangleFracM2}) does not hold.

\begin{theorem}\label{TheoremGrPIexpNonTriangleFracM2}
Let $A$ be a finite dimensional $T$-graded-simple algebra over a field $F$ of characteristic $0$ for a right zero band $T$. Suppose $A/J(A) \cong M_2(F)$.
Let $T_0, T_1 \subseteq T$ and $\sim$ be, respectively, the subsets and the equivalence relation defined at the beginning of Section~\ref{SectionFracM2Equal}.
  Suppose also that $|\bar t_0| > \frac{|T_0|}{2}$ for some $\bar t_0 \in T_0/\sim$.
  Then there exists $$\PIexp^{T\text{-}\mathrm{gr}}(A) = |T_0|+2|T_1|
  + 2\sqrt{(|T_1|+|\bar t_0|)(|T_0|+|T_1|-|\bar t_0|)}<2|T_0|+4|T_1|=\dim A.$$
\end{theorem}
\begin{proof}
Recall that at the beginning of Section~\ref{SectionFracM2Less}
we chose the basis $\mathcal B$ in $A$.
Equation (\ref{EquationGammaFracM2}) implies
that  $0<\zeta=\sqrt{\frac{|T_0|+|T_1|-|\bar t_0|}{|T_1|+|\bar t_0|}}<1$ is the root of~(\ref{EquationZetaUpperFrac}).

Let $$\Omega = \left\lbrace (\alpha_1, \ldots, \alpha_r)\in \mathbb R^r \mathrel{\biggl|} \sum_{i=1}^r \alpha_i = 1,\ 
\alpha_1 \geqslant \alpha_2 \geqslant \ldots \geqslant \alpha_r\geqslant 0,\ \sum_{i=1}^r \gamma_i \alpha_i \leqslant 0\right\rbrace.$$
By Lemma~\ref{LemmaMaximumFUpperFrac},
\begin{equation}\label{EquationMaximumGammaOmegaFracM2} d:=\max_{x\in \Omega} \Phi(x) =\sum_{i=1}^r \zeta^{\gamma_i}= |T_0|+2|T_1|
  + 2\sqrt{(|T_1|+|\bar t_0|)(|T_0|+|T_1|-|\bar t_0|)}.\end{equation}
Denote by $(\alpha_1, \ldots, \alpha_r) \in \Omega$ such a point that
$\Phi(\alpha_1, \ldots, \alpha_r)=d$.

For every $n\in\mathbb N$ define $\mu\vdash n$ by
$\mu_i = 2\left[\frac{\alpha_i n}{2}\right]$ for $2\leqslant i \leqslant r$
and $\mu_1 =  n-\sum_{i=2}^r \mu_i$.

By~(\ref{EquationAlphaRelationUpperFrac1}), $\sum_{i=1}^r \gamma_i \alpha_i = 0$.
Since $$\gamma_1 = \ldots = \gamma_{|T_0|+|T_1|-|\bar t_0|}=-1,$$
$$\gamma_{|T_0|+|T_1|-|\bar t_0|+1} = \ldots = \gamma_{2|T_0|+3|T_1|-|\bar t_0|}=0,$$
and $$\gamma_{2|T_0|+3|T_1|-|\bar t_0|+1} = \ldots = \gamma_r=1,$$
By~(\ref{EquationAlphaRelationUpperFrac2}),
$$\alpha_1 = \ldots = \alpha_{|T_0|+|T_1|-|\bar t_0|},$$
$$\alpha_{|T_0|+|T_1|-|\bar t_0|+1} = \ldots = \alpha_{2|T_0|+3|T_1|-|\bar t_0|},$$
$$\alpha_{2|T_0|+3|T_1|-|\bar t_0|+1} = \ldots = \alpha_r,$$
and we have 
$$\alpha_1 n-2 \leqslant \mu_2 = \ldots=\mu_{|T_0|+|T_1|-|\bar t_0|} \leqslant \alpha_1 n,$$
$$\alpha_{|T_0|+|T_1|-|\bar t_0|+1} n-2 \leqslant \mu_{|T_0|+|T_1|-|\bar t_0|+1} = \ldots=\mu_{2|T_0|+3|T_1|-|\bar t_0|} \leqslant \alpha_{|T_0|+|T_1|-|\bar t_0|+1} n,$$
$$\alpha_r n-2 \leqslant \mu_{2|T_0|+3|T_1|-|\bar t_0|+1} = \ldots=\mu_r \leqslant \alpha_r n.$$
Now $\sum_{i=1}^r \alpha_i = 1$ implies
$\alpha_1 n \leqslant \mu_1 \leqslant \alpha_1 n+2r$.

Note that
\begin{equation}
\label{EquationSumMuGammaFracM2}\begin{split} \sum_{i=1}^r \gamma_i \mu_i=
-\left(n-\sum_{i=2}^r \mu_i\right)-\sum_{i=2}^{|T_0|+|T_1|-|\bar t_0|} \mu_i +  \sum_{i=2|T_0|+3|T_1|-|\bar t_0|+1}^r \mu_i \\= \left(\sum_{i=|T_0|+|T_1|-|\bar t_0|+1}^{2|T_0|+3|T_1|-|\bar t_0|} \mu_i\right)+ 2\left(\sum_{i=2|T_0|+3|T_1|-|\bar t_0|+1}^r \mu_i\right) - n
\\
\leqslant n\left(\sum_{i=|T_0|+|T_1|-|\bar t_0|+1}^{2|T_0|+3|T_1|-|\bar t_0|} \alpha_i\right)
+2n \left(\sum_{i=2|T_0|+3|T_1|-|\bar t_0|+1}^r \alpha_i \right)
 -n
\sum_{i=1}^r \alpha_i = n \sum_{i=1}^r \gamma_i \alpha_i = 0.\end{split}\end{equation}
Analogously,
\begin{equation}
\label{EquationSumMuGamma2FracM2}\begin{split} \sum_{i=1}^r \gamma_i \mu_i=
 \left(\sum_{i=|T_0|+|T_1|-|\bar t_0|+1}^{2|T_0|+3|T_1|-|\bar t_0|} \mu_i\right)+ 2\left(\sum_{i=2|T_0|+3|T_1|-|\bar t_0|+1}^r \mu_i\right) - n
 \\\\
\geqslant n\left(\sum_{i=|T_0|+|T_1|-|\bar t_0|+1}^{2|T_0|+3|T_1|-|\bar t_0|} \alpha_i\right)
+2n \left(\sum_{i=2|T_0|+3|T_1|-|\bar t_0|+1}^r \alpha_i \right)
 -n
\sum_{i=1}^r \alpha_i-4r \\\\= n \sum_{i=1}^r \gamma_i \alpha_i -4r = -4r.\end{split}\end{equation}

 By~(\ref{EquationSumMuGammaFracM2}) and Lemma~\ref{LemmaAltNonTriangleFracM2}, $b_{T_\mu}f \notin \Id^{(FT)^*}(A)$
 for some $f\in P^{(FT)^*}_n$. Let $\bar f$ be the image of $f$ in $\frac{P^{(FT)^*}_n}{P^{(FT)^*}_n
 \cap\ \Id^{(FT)^*}(A)}$. Consider $FS_n b_{T_\mu}\bar f \subseteq \frac{P^{(FT)^*}_n}{P^{(FT)^*}_n
 \cap\ \Id^{(FT)^*}(A)}$. Since all $S_n$-representations over fields of characteristic $0$
 are completely reducible, we have $$FS_n b_{T_\mu}\bar f \cong FS_n e_{T_{\lambda^{(1)}}}
 \oplus \ldots \oplus FS_n e_{T_{\lambda^{(s)}}}$$
 for some $\lambda^{(i)}\vdash n$ and some Young tableaux $T_{\lambda^{(i)}}$
 of shape $\lambda^{(i)}$, $1\leqslant i \leqslant s$, $s\in\mathbb N$.
In particular, $e^*_{T_{\lambda^{(1)}}} FS_n b_{T_\mu} \ne 0$.

Now we notice that $e^*_{T_{\lambda^{(1)}}} FS_n b_{T_\mu} \ne 0$ implies
$a_{T_{\lambda^{(1)}}} \sigma b_{T_\mu} = \sigma a_{\sigma^{-1}T_{\lambda^{(1)}}}  b_{T_\mu} \ne 0$ for some $\sigma \in S_n$. Since $a_{\sigma^{-1}T_{\lambda^{(1)}}}$ is the operator of symmetrization
in the numbers from the rows of the Young tableau $\sigma^{-1}T_{\lambda^{(1)}}$
and $b_{T_\mu}$ is the operator of alternation in the numbers from the
columns of the Young tableau $T_\mu$, all numbers from the first row of $\sigma^{-1}T_{\lambda^{(1)}}$
must be in different columns of $T_\mu$. Thus $\left(\lambda^{(1)}\right)_1 \leqslant \mu_1$.
Moreover, all numbers from each of the first $\mu_r$ columns of $T_\mu$ must be in different rows
of $\sigma^{-1}T_{\lambda^{(1)}}$. Since by Lemma~\ref{LemmaInequalityLambdaUpperFrac}, $\left(\lambda^{(1)}\right)_{r+1}=0$,
we have $\left(\lambda^{(1)}\right)_r \geqslant \mu_r$.

Thus~(\ref{EquationSumMuGamma2FracM2})
implies
\begin{eqnarray*}
  \lefteqn{\sum_{i=1}^r \gamma_i \left(\lambda^{(1)}\right)_i \geqslant \sum_{i=1}^r \gamma_i \left(\lambda^{(1)}\right)_i
- \sum_{i=1}^r \gamma_i \mu_i - 4r}\\
&& \hspace{1,7cm} =
\sum_{i=1}^{|T_0|+|T_1|-|\bar t_0|} \left(\mu_i-\left(\lambda^{(1)}\right)_i\right)+\sum_{i=2|T_0|+3|T_1|-|\bar t_0|+1}^r \left(\left(\lambda^{(1)}\right)_i-\mu_i\right)
-4r \\
&& \hspace{1,7cm} \geqslant\sum_{i=1}^{|T_0|+|T_1|-|\bar t_0|} \left(\mu_1-\left(\lambda^{(1)}\right)_i\right)+\sum_{i=2|T_0|+3|T_1|-|\bar t_0|+1}^r \left(\left(\lambda^{(1)}\right)_i-\mu_r\right)
-6r-2r^2
\end{eqnarray*}  

%
%
%
Since by Lemma~\ref{LemmaInequalityLambdaUpperFrac} we have $\sum_{i=1}^r \gamma_i \left(\lambda^{(1)}\right)_i \leqslant 1$ and both $\left(\mu_1-\left(\lambda^{(1)}\right)_i\right)$
and $\left(\left(\lambda^{(1)}\right)_i-\mu_r\right)$ are nonnegative, we get
 $\mu_1 - (2r^2+6r+1) \leqslant \left(\lambda^{(1)}\right)_i \leqslant \mu_1$
for all $1\leqslant i \leqslant |T_0|+|T_1|-|\bar t_0|$
and
$\mu_r \leqslant \left(\lambda^{(1)}\right)_i \leqslant \mu_r+(2r^2+6r+1)$
for all $2|T_0|+3|T_1|-|\bar t_0|+1\leqslant i \leqslant r$.

Recall that $a_{\sigma^{-1}T_{\lambda^{(1)}}}  b_{T_\mu} \ne 0$ implies
that all numbers from each column
of $T_\mu$ are in different rows of $\sigma^{-1}T_{\lambda^{(1)}}$.
Applying this to the first $\mu_{2|T_0|+3|T_1|-|\bar t_0|}$ columns,
we obtain that in the last $|\bar t_0|+|T_1|+1$ rows of $T_{\lambda^{(1)}}$ we have at least $\sum_{i=2|T_0|+3|T_1|-|\bar t_0|}^r\mu_i$
boxes and $$\sum_{i=2|T_0|+3|T_1|-|\bar t_0|}^r \left(\lambda^{(1)}\right)_i
\geqslant \sum_{i=2|T_0|+3|T_1|-|\bar t_0|}^r\mu_i=(|\bar t_0|+|T_1|)\mu_r
+ \mu_{2|T_0|+3|T_1|-|\bar t_0|}.$$
Thus $$
  \begin{array}{l}
\lambda^{(1)}_{2|T_0|+3|T_1|-|\bar t_0|} \geqslant(|\bar t_0|+|T_1|)\mu_r + \mu_{2|T_0|+3|T_1|-|\bar t_0|} - \sum_{i=2|T_0|+3|T_1|-|\bar t_0|+1}^r \left(\lambda^{(1)}\right)_i \\
\hspace{2,6cm}\geqslant \mu_{2|T_0|+3|T_1|-|\bar t_0|} - (2r^3+6r^2+r)
\end{array}
$$

%
%
%

Let $\lambda=(\lambda_1, \lambda_2, \ldots, \lambda_r)$
where $$ \lambda_i=\left\lbrace\begin{array}{rrr}
\mu_1 - (2r^2+6r+1) & \text{ for } & 1\leqslant i \leqslant |T_0|+|T_1|-|\bar t_0|,\\
\mu_{2|T_0|+3|T_1|-|\bar t_0|} - (2r^3+6r^2+r) & \text{ for } & |T_0|+|T_1|-|\bar t_0|+1 \leqslant i \leqslant 2|T_0|+3|T_1|-|\bar t_0|, \\
\mu_r & \text{ for } & 2|T_0|+3|T_1|-|\bar t_0|+1\leqslant i \leqslant r.
\end{array} \right.$$ Let $n_1 =\sum_{i=1}^r \lambda_i$.
Note that $n - (2r^4+6r^3+r^2) \leqslant n_1 \leqslant n$.

For every $\varepsilon > 0$ there exists $n_0\in\mathbb N$
such that for every $n\geqslant n_0$ we have $\lambda_1\geqslant \ldots \geqslant \lambda_r$ 
and $\Phi\left(\frac{\lambda_1}{n_1},\ldots,\frac{\lambda_r}{n_1}\right) > d-\varepsilon$.
Since $D_\lambda$ is a subdiagram of $D_{\lambda^{(1)}}$,
we have $c^{(FT)^*}_n(A)\geqslant \dim M_\lambda$
 and by the hook and the Stirling formulas, there exist $C_1 > 0$ and $r_1\in\mathbb R$ such that
 we have \begin{equation}\begin{split} c^{(FT)^*}_n(A) \geqslant \dim M(\lambda) = \frac{n_1!}{\prod_{i,j} h_{ij}}
  \geqslant \frac{n_1!}{(\lambda_1+r-1)! \ldots (\lambda_r+r-1)!} \\
  \geqslant \frac{n_1!}{n_1^{r(r-1)}\lambda_1! \ldots \lambda_r!} \geqslant
  \frac{C_1 n_1^{r_1} 
\left(\frac{n_1}{e}\right)^{n_1}}{\left(\frac{\lambda_1}{e}\right)^{\lambda_1}\ldots
\left(\frac{\lambda_r}{e}\right)^{\lambda_r}} \\ \geqslant C_1 n_1^{r_1}\left(\frac{1}
{\left(\frac{\lambda_1}{n_1}\right)^{\frac{\lambda_1}{n_1}}\ldots
\left(\frac{\lambda_r}{n_1}\right)^{\frac{\lambda_r}{n_1}}}\right)^{n_1} \geqslant C_1 n_1^{r_1}
 (d-\varepsilon)^{n - (2r^4+6r^3+r^2)}.\end{split}\end{equation}
 Hence $\mathop{\underline\lim}_{n\rightarrow\infty}\sqrt[n]{c_n^{(FT)^*}(A)}
\geqslant d-\varepsilon$. Since $\varepsilon > 0$ is arbitrary, $\mathop{\underline\lim}_{n\rightarrow\infty}\sqrt[n]{c_n^{(FT)^*}(A)}
\geqslant d$. By Lemma~\ref{LemmaCnGrCnGenH} and~(\ref{EquationMaximumGammaOmegaFracM2}), 
$$\mathop{\underline\lim}\limits_{n\rightarrow\infty}
 \sqrt[n]{c_n^{T\text{-}\mathrm{gr}}(A)}\geqslant  |T_0|+2|T_1|
  + 2\sqrt{(|T_1|+|\bar t_0|)(|T_0|+|T_1|-|\bar t_0|)}.$$

Theorem~\ref{TheoremUpperFrac} and~(\ref{EquationMaximumGammaOmegaFracM2}) implies
$$\mathop{\overline\lim}\limits_{n\rightarrow\infty}
 \sqrt[n]{c_n^{T\text{-}\mathrm{gr}}(A)}\leqslant  |T_0|+2|T_1|
  + 2\sqrt{(|T_1|+|\bar t_0|)(|T_0|+|T_1|-|\bar t_0|)}.$$
  
The condition $|\bar t_0| > \frac{|T_0|}2$
implies $|T_1|+|\bar t_0| > |T_0|+|T_1|-|\bar t_0|$,
$$2\sqrt{(|T_1|+|\bar t_0|)(|T_0|+|T_1|-|\bar t_0|)} < (|T_1|+|\bar t_0|)+(|T_0|+|T_1|-|\bar t_0|)$$
  and $$\lim\limits_{n\rightarrow\infty}
 \sqrt[n]{c_n^{T\text{-}\mathrm{gr}}(A)} = |T_0|+2|T_1|
  + 2\sqrt{(|T_1|+|\bar t_0|)(|T_0|+|T_1|-|\bar t_0|)} < 2|T_0|+4|T_1|=\dim A.$$
\end{proof}
\begin{remark}
A finite dimensional $T$-graded-simple algebra $A$
with any given $|T_0|$, $|T_1|$, $\frac{|T_0|}{2} < |\bar t_0| \leqslant |T_0|$ exists by Proposition~\ref{TheoremImagesOfGradedComponentsReesExistence}.
\end{remark}

\end{document}